\documentclass[12pt]{paper}
\pdfoutput=1
\usepackage[normalem]{ulem}
\usepackage{xifthen}
\usepackage{thm-restate}
\usepackage{tikzstyle}
\usepackage[compatibility=false]{caption}
\usepackage{subcaption}
\DeclareCaptionFont{myit}{\itshape}
\AtBeginDocument{%
  \captionsetup[figure]{labelfont=normalfont,textfont=myit}%
  \captionsetup[subfigure]{labelfont=normalfont,textfont=myit}%
}

\usepackage{tabularx}
\usepackage{ltablex}
\newcommand{\glossarylabel}[1]{\phantomsection\label{#1}}

\colorlet{clou}{orange!80!yellow}

\newenvironment{poc}{\begin{proof}}{\end{proof}}

\newcommand{\p}[1]{\mathbf{P}\left(#1\right)}

\newcommand{\cat}{C}

\DeclareMathOperator{\intr}{int}

\newcommand{\roote}[2]{r_{#1}(#2)}
\newcommand{\inroote}[1]{r_{#1}}
\newcommand{\head}[2]{h_{#1}(#2)}

\newcommand{\area}[2]{A_{#1}(#2)}
\newcommand{\sparea}[3]{A_{#1}^{#3}(#2)}
\newcommand{\evarea}[1]{A_{#1}}
\newcommand{\spevarea}[2]{A_{#1}^{#2}}
\newcommand{\boundary}[2]{B_{#1}(#2)}
\newcommand{\peelv}[2]{p_{#1}(#2)}
\newcommand{\chord}[2]{c_{#1}(#2)}

\newcommand{\ball}[2]{\cB(#2,#1)}
\newcommand{\ballt}[3]{\cB_{#3}(#2,#1)}
\newcommand{\fs}[2]{i_{#1}(#2)}
\newcommand{\segment}[3]{B_{#1}^{#3}(#2)}
\newcommand{\infpath}[1]{P^{#1}}
\newcommand{\vpath}[2]{P^{#1}(#2)}
\newcommand{\qpath}[1]{Q^{#1}}

\newcommand{\nt}{\varphi} 
\newcommand{\ntbase}{\alpha} 
\newcommand{\ntfact}{A} 
\newcommand{\fd}{\mu} 
\newcommand{\fdgen}{Z} 

\title{\vspace{-1cm}Infinite Schnyder Woods}
\author{Louigi Addario-Berry\affiliation{Department of Mathematics and Statistics, McGill University, Montr\'{e}al, Qu\'{e}bec, Canada (\textsf{\href{mailto:louigi.addario@mcgill.ca}{louigi.addario@mcgill.ca}}).} \and Emma Hogan\affiliation{Mathematical Institute, University of Oxford, United Kingdom (\textsf{\{\href{mailto:emma.hogan@maths.ox.ac.uk}{emma.hogan},\href{mailto:lukas.michel@maths.ox.ac.uk}{lukas.michel},\allowbreak\href{mailto:alexander.scott@maths.ox.ac.uk}{alexander.scott}\}@maths.ox.ac.uk}).} 
\and Lukas Michel\affiliationmark \and Alex Scott\affiliationmark}

\begin{document}
\maketitle

\begin{abstract}
    It is well-known that any finite triangulation possesses a unique maximal Schnyder wood. We introduce Schnyder woods of infinite triangulations, and prove there exists a unique maximal Schnyder wood of any infinite triangulation with finite boundary, and of the uniform infinite half-planar triangulation. Furthermore, the maximal Schnyder wood of the uniform infinite planar triangulation is the limit of maximal Schnyder woods of large finite random triangulations. Several structural properties of infinite Schnyder woods are also described.
\end{abstract}

\begin{figure}[h!]
\begin{center}
\includegraphics[width=0.8\textwidth]{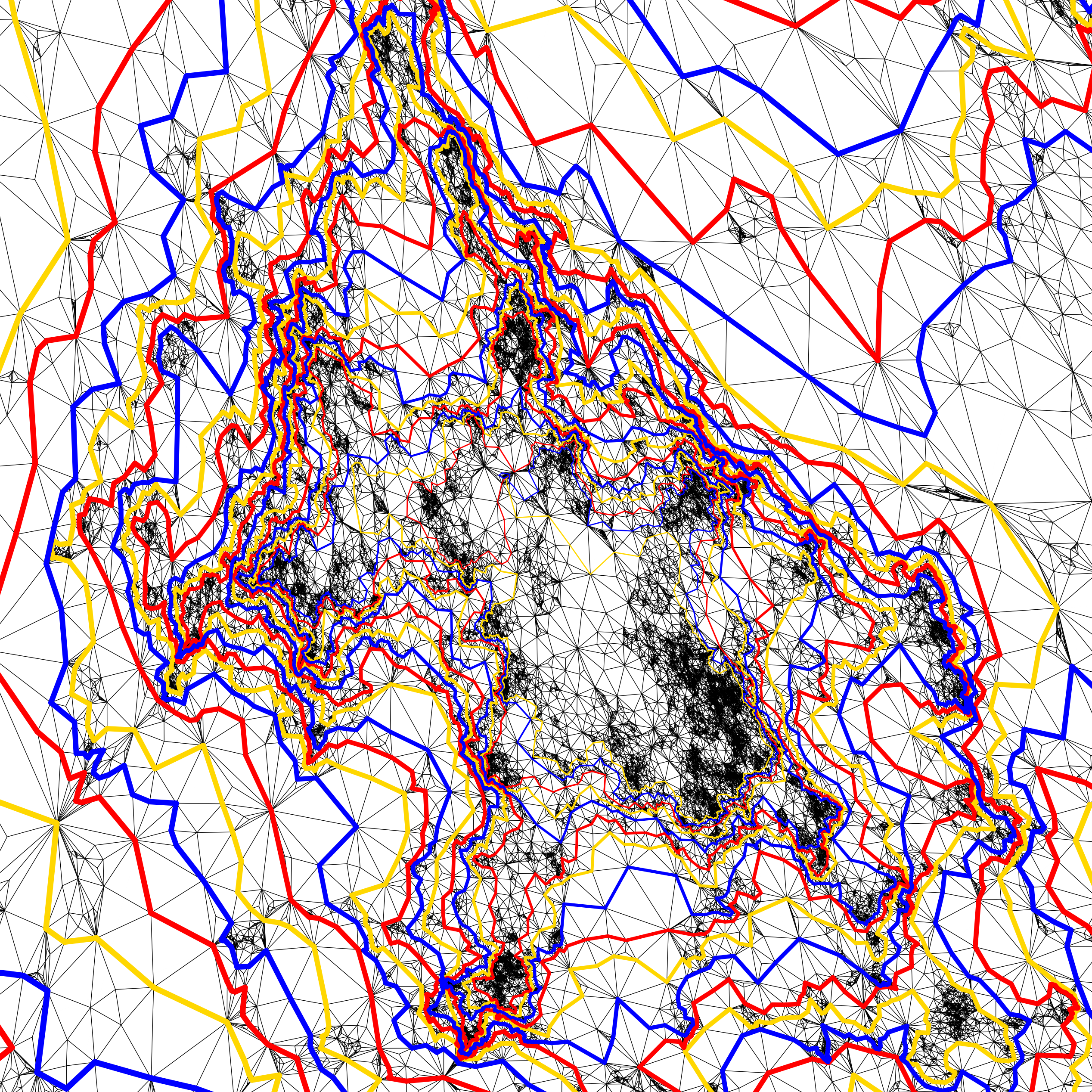}
\end{center}
\caption{The structure of the unique monochromatic paths starting from a vertex chosen uniformly at random from the maximal Schnyder wood of a large random triangulation. See also  \cref{fig:tutte_embedding_big,fig:most_zoom} on \cpageref{fig:tutte_embedding_big,fig:most_zoom}, respectively.}
\label{fig:some_zoom}
\end{figure}

\newpage
\tableofcontents
\normalsize

\newpage
        
\section{Introduction}

Our aim in this work is to describe the structure of Schnyder woods---defined shortly---of (large) finite and infinite random triangulations of the plane and of the half-plane. The two goals are related, since one approach to determining the properties of a typical large finite triangulation is to analyse the limit of a sequence of uniformly random triangulations as its number of vertices tends to infinity. One important limit object of this kind is the \defn{uniform infinite planar triangulation} (UIPT). Let $\cT_n^m$\glossarylabel{gl:ctnm} denote the set of simple, $2$-connected triangulations of an $(m+2)$-gon with $n$ interior vertices that are rooted at an edge on the boundary. Then the UIPT is a sample from the probability measure that is the weak $n \to \infty$ limit of the uniform measure on $\cT_n^1$. The existence of the UIPT was first conjectured by Benjamini and Schramm \cite{benjamini2001recurrence}, and later established by Angel and Schramm \cite{angel2003uipt}. Since then, there has been substantial study of the properties of the UIPT and of similarly defined infinite random maps \cite{angel2003growth,angel2005scaling,angel2015percolations}. One random infinite map with many related properties is the \defn{uniform infinite half-planar triangulation} (UIHPT). The UIHPT can be obtained similarly to the UIPT, as a limit of uniformly random samples from $\cT_n^m$ where $m = m(n) \to \infty$ and $m/n \to 0$ as $n \to \infty$ \cite[Theorem 1.4]{MR3342664}. The UIHPT was shown to exist by Angel \cite{angel2003growth}, and has been studied in more detail in several subsequent works, including \cite{MR3342664,angel2015percolations}.

\defn{Schnyder woods} are a well-known class of structures on finite triangulations. A Schnyder wood is an orientation and a colouring of the edges of a triangulation in red, yellow, and blue that satisfies the three conditions indicated in \cref{fig:schnydercondition1}. Schnyder woods were first introduced by Schnyder \cite{schnyder1989planar} to prove his well-known theorem characterising planar graphs as the set of graphs with order-dimension at most three. Schnyder \cite{schnyder1990} also showed that these woods provide a way to efficiently construct straight-line embeddings of triangulations in the integer lattice. Since then, Schnyder woods have been studied in a variety of combinatorial and computational contexts, where they are sometimes also called realisers. In \cite{aleardi2009schnyder}, Aleardi, Fusy, and Lewiner defined a generalisation of Schnyder woods to higher-genus surfaces and showed that these Schnyder woods can be used to efficiently encode certain higher-genus surfaces. Schnyder woods have also found application in the study of Catalan lattices \cite{bernardi2009intervals} and orthogonal surfaces \cite{felsner2008schnyder}.

\begin{figure}[h]
    \begin{subfigure}[t]{0.3\textwidth}
        \centering
        \begin{tikzpicture}[baseline=(current bounding box.south)]
    \useasboundingbox (-1.9,-1.7) rectangle (1.9,1.9);
    \node[mycircle] (v) at (0,0){};

    \begin{scope}[rotate=210]
        \foreach \x/\y/\r in {0/0/1.9,37/1/1.6,60/15/1.6, 83/2/1.6, 120/3/1.9, 157/4/1.6, 180/45/1.6, 203/5/1.6, 240/6/1.9, 277/7/1.6, 300/75/1.6, 323/8/1.6}{ 
            \node (\y) at (canvas polar cs: radius=\r cm,angle=\x){};
        }
    \end{scope}

    \path[every node/.style={font=\sffamily\small}]
        (v) edge [->-, redge] (0)
        (v) edge [->-, yedge] (3)
        (v) edge [->-, bedge] (6)
        (1) edge [->-, bedge] (v)
        (15) edge [->-, bedge] (v)
        (2) edge [->-, bedge] (v)
        (4) edge [->-, redge] (v)
        (45) edge [->-, redge] (v)
        (5) edge [->-, redge] (v)
        (7) edge [->-, yedge] (v)
        (75) edge [->-, yedge] (v)
        (8) edge [->-, yedge] (v);
    
\end{tikzpicture}
        \caption{The Schnyder condition.}
    \end{subfigure}
    \hspace{0.1cm}
    \begin{subfigure}[t]{0.34\textwidth}
        \centering
        \begin{tikzpicture}[baseline=(current bounding box.south)]
    \useasboundingbox (-1.6,-1.2) rectangle (3.5,1.6);
    \node[mycircle] (vy) at (1.9,0){};
    \node[mycircle] (vr) at (0,0){};

    \begin{scope}[rotate=60]
        \foreach \x/\y/\r in {0/0/1.6,60/1/1.6, 120/2/1.6}{ 
            \node (\y) at (\x:\r){};
        }
    \end{scope}

    \begin{scope}[xshift=1.9cm]
        \foreach \x/\y/\r in {0/3/1.6,60/4/1.6, 120/5/1.6}{ 
            \node (\y) at (\x:\r){};
        }
    \end{scope}

    \path[every node/.style={font=\sffamily\small}]
        (vr) edge [->-, yedge] (vy)
        (0) edge [->-, redge] (vr)
        (1) edge [->-, redge] (vr)
        (2) edge [->-, redge] (vr)
        (3) edge [->-, yedge] (vy)
        (4) edge [->-, yedge] (vy)
        (5) edge [->-, yedge] (vy);

    \draw (0.95,-0.3) node {$r$};
\end{tikzpicture}
        \caption{The Schnyder root condition. The root edge is labelled $r$.}
    \end{subfigure}
    \hspace{0.1cm}
    \begin{subfigure}[t]{0.32\textwidth}
        \centering
        \begin{tikzpicture}[baseline=(current bounding box.south)]
    \useasboundingbox (-1.9,-1.2) rectangle (1.9,1.9);
    \node[mycircle] (v) at (0,0){};

    \begin{scope}[rotate=30]
        \foreach \x/\y/\r in {20/0/1.9,50/1/1.6, 70/2/1.6, 100/3/1.9, 130/4/1.6, 150/45/1.6, 350/7/1.6, 330/75/1.6}{ 
            \node (\y) at (canvas polar cs: radius=\r cm,angle=\x){};
        }
    \end{scope}

    \path[every node/.style={font=\sffamily\small}]
        (v) edge [->-, redge] (0)
        (v) edge [->-, yedge] (3)
        (1) edge [->-, bedge] (v)
        (2) edge [->-, bedge] (v)
        (4) edge [->-, redge] (v)
        (45) edge [->-, redge] (v)
        (7) edge [->-, yedge] (v)
        (75) edge [->-, yedge] (v);
    
\end{tikzpicture}
        \caption{The Schnyder boundary condition. The boundary lies vertically below all edges.}
    \end{subfigure}
    \caption{Edges incident with internal vertices satisfy the condition indicated in (a), edges incident with the root edge satisfy the condition indicated in (b), and edges incident with non-root boundary vertices satisfy the condition in (c).}
    \label{fig:schnydercondition1}
\end{figure}
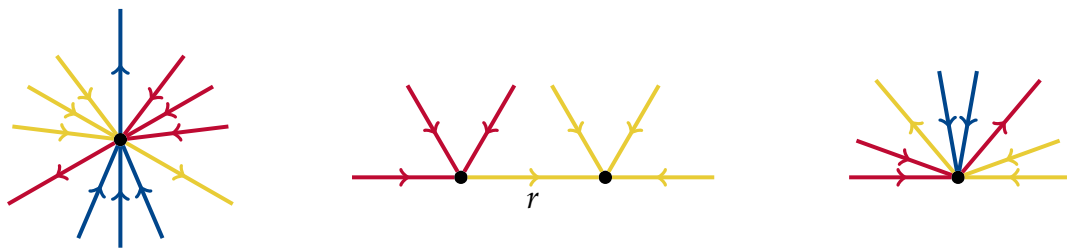

For any given finite triangulation $T$, it is well-known that the set of Schnyder woods of $T$ is naturally endowed with a partial order possessing a unique maximal and minimal element \cite{brehm2000orientations,felsner2008schnyder}, where a Schnyder wood is maximal (minimal) if it contains no anticlockwise (clockwise) directed cycles. To study infinite Schnyder woods, we generalise this definition of maximality to forbid infinite anticlockwise cycles, that is, two-ended infinite directed paths for which the triangulation boundary lies on the right. We justify in \cref{sec:infiniteschnyderwoods} that this slight modification is necessary.

The structure of Schnyder woods of infinite and large finite triangulations is a natural area of study, particularly due to the link between Schnyder woods and canonical straight-line embeddings of the underlying graphs. Schnyder woods associated to large random triangulations have previously been studied in the context of results on the scaling limits of large random triangulations. In particular, \cite{addario2017scaling} used an algorithm for constructing the maximal Schnyder wood of a finite triangulation as an essential part of its proof that the scaling limit of uniformly random simple triangulations is the Brownian sphere. The work \cite{li2017schnyder} describes a conjectural scaling limit for the embeddings associated with uniformly random wooded triangulations (triangulations endowed with a Schnyder wood), which is described in terms of three coupled $\mathrm{SLE}_{16}$ curves, and verifies the conjecture in a weak form. 

A standard strategy for constructing Schnyder woods of finite triangulations is to explore the triangulation iteratively via a peeling process, colouring and directing newly explored edges at each step. Peeling processes are also commonly used for studying properties of random infinite triangulations such as the UIPT and UIHPT. We describe a new peeling process, based on previous work on peeling processes for finite Schnyder woods, and use this process to explore infinite triangulations. For any infinite triangulation with finite boundary, this process can be used to show that a unique maximal Schnyder wood exists.

\begin{theorem}
    Every infinite triangulation with a finite boundary has a unique maximal Schnyder wood.
\end{theorem}

We further show that the maximal Schnyder wood of the UIPT is the weak $n \to \infty$ limit of the maximal Schnyder woods of uniformly random triangulations of $\cT_n^1$. In particular, the maximal Schnyder wood of the UIPT can be used to investigate the structure of Schnyder woods of typical large finite triangulations near the root edge.

\begin{theorem}\label{thm:intro-uipt}
    For all $m \in \N$, the maximal Schnyder wood of the UIPT with boundary length $m$ is the weak limit of the maximal Schnyder wood of a uniformly random triangulation from $\cT_n^m$ as $n \to \infty$.
\end{theorem}

Constructing Schnyder woods of infinite triangulations of the half-plane such as the UIHPT is more challenging. On an infinite triangulation with a finite boundary, our peeling process begins exploring from the root edge of the triangulation and progressively explores along the boundary of the unexplored region in an anticlockwise direction. This constructs the Schnyder wood layer by layer as the process ``winds around the boundary''. However, when the boundary length tends to infinity, this peeling process will explore parts of the triangulation at increasingly large distance from the root edge before returning to edges close to the root edge. In particular, on the UIHPT, a peeling process initiated from a fixed root edge almost surely never explores most of the triangulation. This non-locality poses a substantial challenge, and while we believe that the maximal Schnyder wood of a uniformly random triangulation from $\cT_n^m$ converges as $n \to \infty$ and then $m \to \infty$, we have been unable to prove this. Instead, the majority of this paper is dedicated to establishing the existence of a unique maximal Schnyder wood of the UIHPT, which we conjecture to be this limit object.

To establish the existence of a Schnyder wood of the UIHPT, we show that any edge in the UIHPT can be coloured by a sequence of peeling processes initiated at different root edges. By choosing this sequence carefully, we show that this colouring is well-defined and satisfies the standard properties of a Schnyder wood. We are further able to show that the Schnyder wood constructed in this manner is the unique maximal Schnyder wood of the UIHPT.

\begin{theorem}\label{thm:main1}
    The UIHPT almost surely has a unique maximal Schnyder wood.
\end{theorem}

We also use properties of our peeling process to describe the structure of the monochromatic subgraphs of the unique maximal Schnyder wood of the UIHPT. As in the case of finite Schnyder woods, the monochromatic connected subgraphs of a Schnyder wood of the half-plane are trees. Furthermore, in the maximal Schnyder wood of the UIHPT, these trees have a natural order. There is an infinite yellow tree that lies on the initial boundary, and as we move away from the initial boundary, we see, in order, an infinite yellow tree, an infinite red tree, and then an infinite blue tree, before this cycle repeats indefinitely. The following theorem makes this structure more precise.

\begin{theorem}\label{thm:intro_monochromatic_schnyder_wood_structure}
    The monochromatic subgraphs of the unique maximal Schnyder wood of the UIHPT are as follows.
    \begin{itemize}
        \item A set of finite blue trees rooted at vertices on the boundary.
        \item Three one-ended yellow, red, and blue trees respectively rooted at vertices $a$, $b$, and $c$ along the boundary, with $a$ the head of the root edge of the UIHPT; with $b$ the tail of the root edge; and with $c$ somewhere clockwise from $b$ along the boundary. Each of these trees contains a unique infinite directed path ending at $a$, $b$, and $c$ respectively.
        \item An infinite sequence of two-ended yellow, red, and blue trees, each containing a unique infinite left-directed path, so that the order of the paths moving away from the boundary cycles yellow, red, blue, yellow, red, blue, and so on.
    \end{itemize}
    Furthermore, every vertex on the one-ended infinite yellow path is on the initial boundary, and every vertex on the first two-ended infinite yellow path is either on the initial boundary or has an outgoing blue edge to the one-ended infinite blue path. Every vertex on the remaining infinite paths has an outgoing edge to the path immediately preceding it (also of the same colour as the path immediately preceding it).
\end{theorem}

In fact, the structure of the outgoing edges between the infinite monochromatic paths described in \cref{thm:intro_monochromatic_schnyder_wood_structure} mirrors the structure of the monochromatic paths starting from any fixed vertex. For instance, if we consider the three outgoing monochromatic paths starting from a fixed vertex in a maximal Schnyder wood of a finite triangulation, then every vertex on the blue path has an outgoing red edge to the red path, every vertex on the red path has an outgoing yellow edge to the yellow path, and every vertex on the yellow path is either on the boundary of the triangulation or has an outgoing blue edge to the blue path. As far as we are aware, this structure has not been previously described even for finite triangulations. Using our peeling process, we show in \cref{thm:finite_paths_from_v_wind,thm:finite_boundary_paths_from_v_wind,thm:uihpt_paths_from_v_wind} that the monochromatic paths have this structure in maximal Schnyder woods of finite triangulations, infinite triangulations with finite boundary, and, with appropriate modifications, also the UIHPT.

We leave establishing the UIHPT analogue of \cref{thm:intro-uipt} as a conjecture. We will discuss the difficulties of proving this conjecture as well as some of its potential consequences for the structure of geodesic paths and the winding behaviour of the monochromatic trees in a typical large finite triangulation in \cref{ssec:structure:uihptlimitoffinite}.

\begin{conjecture}\label{conj:uihpt_as_limit}
    The unique maximal Schnyder wood of the UIHPT is the weak limit of the maximal Schnyder wood of a uniformly random triangulation from $\cT_n^m$ as $n \to \infty$ and then $m \to \infty$.
\end{conjecture}

While our results establish that unique maximal Schnyder woods exist for all infinite triangulations with finite boundary, for triangulations of the half-plane we only establish existence and uniqueness on the UIHPT. It follows from our construction that Schnyder woods exist on other triangulations of the half-plane with certain similar properties, but it would be interesting to know whether Schnyder woods exist on all half-plane triangulations, and if not, to characterise the ones for which they do exist. It follows from our results in \cref{sec:infiniteschnyderwoods} that whenever a maximal Schnyder wood of the half-plane exists, it is also unique.

Finally, we note that related research has studied Schnyder woods associated to uniformly random wooded triangulations with $n$ faces, which are quite different from maximal Schnyder woods of uniformly random triangulations with $n$ faces. In particular, it follows from the results of \cite{bernardi2009intervals} that the heights of the former trees are $\Theta(\sqrt{n})$ in probability and in expectation, whereas the height of a uniformly random vertex in one of the latter trees has expected value of order $n^{3/4}$ \cite[Theorem 1.2]{chapuy2024scaling}. It would be interesting to understand the local limit of uniformly random wooded triangulations as the boundary length tends to infinity, and whether such wooded triangulations have a local limit, which would be a random wooded triangulation of the half-plane different from the one constructed in this work.

\subsection{Outline}
In \cref{sec:prelim} we provide necessary definitions and preliminary results on finite and infinite triangulations and the theory of Schnyder woods. In particular, we describe a version of the `peeling process' often used for constructing Schnyder woods of finite triangulations. In \cref{sec:infiniteschnyderwoods} we generalise the definition of a maximal Schnyder wood to infinite triangulations and show that maximal Schnyder woods of infinite triangulations are unique, provided they exist. We also discuss a modified peeling process for exploring infinite triangulations. The remainder of the paper focuses on using this peeling process to prove the existence of Schnyder woods of infinite triangulations.

In \cref{sec:infinite_triang_finite_bdry} we show that the peeling process will construct a maximal Schnyder wood of any infinite triangulation with finite boundary. In particular, on the UIPT, we show that the maximal Schnyder wood constructed by the peeling process is the weak limit of the maximal Schnyder woods of finite triangulations as the number of internal vertices tends to infinity. We further use the process to establish some structural features of maximal Schnyder woods of triangulations with finite boundary.

In \cref{sec:uihpt} we show that for triangulations with infinite boundary, this peeling process is \emph{not} guaranteed to explore the entire triangulation. In particular, the peeling process almost surely does not explore the entire UIHPT, and instead explores a narrow `segment' along the initial boundary right of the root edge. Afterwards, in \cref{ssec:uihpt:segment} and \cref{ssec:uihpt:strip}, we prove a sequence of results about the behaviour of the peeling process to establish that repeated initiations of the process can be used to produce a well-defined colouring and orientation of the UIHPT.

In \cref{sec:structure} we are finally able to prove that there is a Schnyder wood on the UIHPT which can be constructed through repeated use of the peeling process. Furthermore, we show that the Schnyder wood constructed in this manner is a maximal Schnyder wood, and therefore unique. We prove this using a sequence of results that provide a detailed description of the structure of the maximal Schnyder wood of the UIHPT.

\subsection{Acknowledgements}
The authors would like to thank Nicolas Curien and Jean-Fran\c{c}ois Le Gall for an instructive email exchange, and Christina Goldschmidt for many useful conversations. The second and third authors were supported in visiting the first author by the RandNET project, MSCA-RISE -- Marie Sk\l{}odowska-Curie Research and Innovation Staff Exchange Programme (RISE), Grant agreement 101007705. We thank McGill University for hosting us. Research of Louigi Addario-Berry was supported by NSERC and by the Canada Research Chairs program. Research of Emma Hogan was supported by EPSRC grant EP/W524311/1. Research of Alex Scott was supported by EPSRC grant EP/X013642/1.

\section{Preliminaries}\label{sec:prelim}

All graphs in this paper are assumed to be simple, locally finite, and \defn{one-ended}, meaning that after the deletion of any finite set of vertices, at most one infinite connected component remains. For an integer $m \geq 3$, we say that a \defn{triangulation $T$ of an $m$-gon} is a planar embedding of a finite or infinite $2$-connected graph $G$ such that one distinguished face of $G$ is bounded exactly $m$ edges, and every other face is bounded by exactly $3$ edges. When $m=3$, so $T$ is a triangulation of a triangle, we will simply refer to $T$ as a \defn{triangulation} without specifying $m$. Similarly, a \defn{triangulation of the half-plane} is a planar embedding of an infinite $2$-connected graph $G$ such that one distinguished face is incident with an infinite number of edges, and all other faces are bounded by exactly $3$ edges.

Given a triangulation $T$ of an $m$-gon or of the half-plane, we refer to the unique distinguished face as the \defn{exterior} face. We call the boundary of the exterior face the \defn{boundary} of $T$, the vertices on the boundary of $T$ \defn{boundary vertices}, and all vertices of $T$ not on the boundary \defn{interior vertices}. All triangulations $T$ considered in this paper are \defn{rooted} which means that there is a unique distinguished directed edge $(x,y)$ on the boundary of $T$, called the \defn{root edge} of $T$, so that the exterior face lies to the right of this root edge. We call the tail of the root edge the \defn{root vertex} of $T$.

Given a triangulation $T$, we will often refer to traversing some substructure of $T$ in clockwise or anticlockwise direction. Since such triangulations are often defined on the sphere, we specify for clarity that this refers to the clockwise or anticlockwise direction in the plane.

We now provide some background on finite and infinite triangulations which we will need to define Schnyder woods of infinite triangulations such as the UIPT and UIHPT later in the paper. First, in \cref{ssec:prelim:fintriangulations} we provide some preliminary combinatorial and probabilistic facts about finite rooted triangulations. We mainly use this section to recall some counting results on the number of triangulations and to define the free distribution on rooted triangulations. This will be useful for arguing about the distribution of finite sub-triangulations of the UIHPT.

In \cref{ssec:prelim:inftriangulations}, we give the relevant preliminary results relating to infinite triangulations. In particular, we define the UIPT and UIHPT and list some of their properties.

\cref{ssec:prelim:schnyderwoods} focuses on Schnyder woods of finite triangulations. We define Schnyder woods and recall some of their properties. We also define the Schnyder peeling process which finds the unique maximal Schnyder wood of any finite triangulation. We will generalise this process to infinite triangulations in \cref{sec:infiniteschnyderwoods}, which will allow us to establish the existence of the Schnyder wood of arbitrary infinite triangulations with finite boundary in \cref{sec:infinite_triang_finite_bdry}, and on the UIHPT in \cref{sec:structure}. 

Finally, we record a small typographical point. Some of the proofs in the paper have proofs of sub-claims embedded within them. We use {\small $\square$} to indicate the end of top-level proofs, and {\small $\blacksquare$} to indicate the end of the proofs of such sub-claims.

\subsection{Finite triangulations}\label{ssec:prelim:fintriangulations}

Recall that $\cT_n^m$ denotes the set of rooted triangulations of an $(m+2)$-gon that have $n$ interior vertices, and let $\nt_{m,n} \coloneqq \abs{\cT_n^m}$\glossarylabel{gl:ntnm} denote the number of such rooted triangulations. In order to later construct a Schnyder wood of the UIHPT, we will rely on a number of distributional properties of the UIHPT. In particular, we will use several of the enumerative results employed by Angel and Schramm~\cite{angel2003uipt} to establish the existence of the UIPT. The results in this section are based largely on techniques originally pioneered by Tutte~\cite{tutte1962census}. The formula for $\nt_{m,n}$ is derived in~\cite{brown1964}, and a full outline of the subsequent asymptotic results given in this section can be found in~\cite{goulden2004}. The value of $\nt_{m,n}$ is
\[
    \nt_{m,n} = \frac{2(2m+1)!(4n+2m-1)!}{(m-1)!(m+1)!n!(3n+2m+1)!}.
\]
To calculate distributions for the UIHPT, the asymptotic behaviour of $\nt_{m,n}$ will be relevant. As $n \to \infty$, this asymptotic behaviour is
\[
    \nt_{m,n} \sim \ntfact_m \ntbase^n n^{-5/2}
\]
where $\ntbase = 256/27$ and
\[
    \ntfact_m \coloneqq \frac{2(2m+1)!}{6\sqrt{6\pi}(m-1)!(m+1)!} \left(\frac{16}{9}\right)^m \sim \frac{1}{3\pi \sqrt{6}} \left(\frac{64}{9}\right)^m m^{1/2}
\]
as $m \to \infty$.

As a consequence of the definitions of the UIPT and UIHPT (given in the following subsection), we will generally work with balls $\ball{\rho}{v}$\glossarylabel{gl:ball} of radius $\rho$ around the root edge of such a triangulation. For a UIPT or UIHPT $T$, the graph $T \setminus \ball{\rho}{v}$ consists of one infinite triangulation, and a (possibly empty) collection of finite triangulations. It turns out that these finite triangulations are distributed according to the so-called free distribution, which we now define. Let
\[
    \fdgen_m(t) \coloneqq \sum_n \nt_{m,n} t^n.
\]
Then, the \defn{free distribution} on rooted triangulations of an $(m+2)$-gon, denoted by $\fd_m$, is the probability measure on $\cT^m \coloneqq \bigcup_{n \ge 0} \cT_n^m$ that assigns weight
\[
    \fd_m(T) \coloneqq \frac{\ntbase^{-n}}{\fdgen_m(\ntbase^{-1})}
\]
to each triangulation $T \in \cT_n^m$. Note that this is a Boltzmann distribution, but not the only one: in the definition of $\fd_m$ we could replace $\ntbase^{-1}$ by any value $t \in (0, \ntbase^{-1})$ to get a different Boltzmann distribution. For any $t \in (0, \ntbase^{-1}]$, it holds that 
\[
    \fdgen_m(t) = \frac{(2m)!((1-4\theta)m+6\theta)}{m!(m+2)!} (1-\theta)^{-(2m+1)},
\]
where $\theta$ and $t$ are related by the formula $t = \theta (1-\theta)^3$. When $t = \ntbase^{-1}$ this yields that 
\[
    \fdgen_m \coloneqq \fdgen_m(\ntbase^{-1}) = \fdgen_m\left(\frac{27}{256}\right) = \frac{2(2m)!}{m!(m+2)!} \left(\frac{16}{9}\right)^m.
\] 

With the standard convention that $0! = 1$, this formula gives $\fdgen_0 = 1$.

\subsection{Infinite triangulations}\label{ssec:prelim:inftriangulations}

In this paper, we are concerned with two types of infinite triangulations: triangulations of polygons (infinite triangulations with finite boundary), and triangulations of the half-plane (infinite triangulations with infinite boundary). In particular, we will work with certain random infinite triangulations of these two types, which we now define.

Let $\tau_n^m$ denote the uniform distribution on $\cT_n^m$. It is shown in \cite{angel2003uipt} and \cite{angel2003growth} that as $n \to \infty$, the measures $\tau_n^m$ converge to a probability measure $\tau^m$ on infinite planar triangulations of the $(m+2)$-gon. In other words, if $T_n^m$ is $\tau_n^m$-distributed and $T^m$ is $\tau^m$-distributed and $v$ denotes the root vertex of these triangulations, then for any radius $\rho$ and finite triangulation $T$ of an $(m+2)$-gon, we have 
\[
\lim_{n \to \infty} \tau_n^m (\ballt{\rho}{v}{T_n^m} = T) = \tau^m(\ballt{\rho}{v}{T^m} = T)
\]
where $\ballt{\rho}{v}{T}$\glossarylabel{gl:ballt} is the ball of radius $\rho$ around $v$ in $T$ under the graph metric. A planar triangulation sampled from $\tau^m$ is known as a \defn{uniform infinite planar triangulation (UIPT)} of the $(m+2)$-gon. In \cite{angel2005scaling}, it is further shown that as $m \to \infty$, the measures $\tau^m$ tend to a weak limit $\tau$, which is a probability measure supported by infinite planar triangulations of the half-plane. In other words, if $T^\infty$ is $\tau$-distributed and has root vertex $v$, then for any any radius $\rho$ and any finite triangulation $T$, we have
\[
\lim_{m \to \infty} \tau^m (\ballt{\rho}{v}{T^m} = T) = \tau(\ballt{\rho}{v}{T^\infty} = T).
\]
A triangulation sampled according to $\tau$ is known as a \defn{uniform infinite half-planar triangulation (UIHPT)}.

Random infinite triangulations can be sampled via a step-by-step exploration process. This method was formalised in \cite{angel2003growth}, and in particular applied to the UIPT. A formal approach to sampling the UIHPT in this manner is described in \cite{angel2015percolations}. Broadly, the strategy is to start with a boundary of the required length and some selected edge, and sample the unique triangle $F$ that this edge lies on. With some probability, the third vertex $v$ of this triangle is an interior vertex, in which case the remaining unexplored region has the same UIPT or UIHPT distribution which we may continue exploring in the same manner. Otherwise, $v$ lies somewhere on the boundary, and removing $F$ partitions the region into a finite region with boundary length $\ell+2$ and an infinite region. The distribution on the finite region is the free distribution $\mu_\ell$ while the distribution on the infinite region is once again the UIPT or UIHPT distribution, and the triangulations of these regions are distributed independently. We may thus sample the triangulation on the finite region according to $\mu_\ell$, and continue exploring the remainder of the UIPT or UIHPT in the same manner as before. For further details on defining and sampling infinite planar triangulations, see \cite{angel2003uipt, angel2003growth} and for more details on UIHPTs see \cite{angel2005scaling, angel2015percolations}. In \cref{sec:infiniteschnyderwoods} we will explain how to explore infinite triangulations using a modified version of this exploration algorithm based on the Schnyder peeling process, which we define in \cref{ssec:prelim:schnyderwoods}. 

We will use two well-known properties of the UIPT and UIHPT. First, as already noted, we only consider one-ended triangulations in this paper, so we remark that both these objects are indeed one-ended. Secondly, we will sometimes consider the distribution of an infinite triangulation after it has been rerooted at a different edge on the boundary, and rely on the fact that for both the UIPT and UIHPT, the distribution is unaffected by this. For the UIPT, these results are proved in \cite{angel2003uipt}. For the UIHPT, they are shown in \cite{angel2005scaling} and formally stated in \cite[Section 2.1]{angel2015percolations}.

Let $T$ be either a UIPT or UIHPT with boundary denoted by $(b_i, i \in \mathbb{Z})$ (where we take $b_i = b_j$ whenever $i \cong j \mod m$ if $T$ is a UIPT with boundary length $m$).

\begin{lemma}[one-endedness]
    $T$ is almost surely one-ended.
\end{lemma}

\begin{lemma}[translation invariance] \label{lem:translation_invariance}
    For all $k \in \mathbb{Z}$, the distribution of $T$ rerooted from edge $(b_{k-1},b_{k})$ to edge $(b_{k},b_{k+1})$ is the same as the original distribution of $T$.
\end{lemma}
 
Note that we will always represent triangulations of the half-plane with their infinite boundary at the bottom of the triangulation. Because of this, we will often describe vertices anticlockwise along the boundary of any infinite triangulation (that is, following the boundary in the direction determined by the root edge) as being to the right, and vertices clockwise along the boundary of a triangulation as being to the left. We remark for clarity that we will use this description even in cases where the boundary itself is finite.

\subsection{Schnyder woods}\label{ssec:prelim:schnyderwoods}

Let $T \in \cT_n^1$ be a rooted triangulation with boundary vertices $v_b,v_r,v_y$ and root edge $(v_r,v_y)$. A \defn{$3$-orientation} of $T$ is an orientation of the edges of $T$ such that every interior vertex of $T$ has exactly $3$ outgoing edges, and $v_b$, $v_r$, and $v_y$, have $2$, $1$, and $0$ outgoing edges respectively. 

A \defn{Schnyder wood} of $T$ is a $3$-colouring of the edges of a $3$-orientation of $T$ in red, yellow, and blue such that every incoming edge to $v_b$, $v_r$, and $v_y$ is coloured blue, red, and yellow respectively, and such that every interior vertex satisfies the following properties, called the \defn{Schnyder condition}.
\begin{itemize}
    \item The $3$ outgoing edges at each interior vertex are coloured blue, red, and yellow in anticlockwise cyclic order, and
    \item Every incoming edge at each interior vertex enters between the outgoing edges of the two other colours.
\end{itemize}
An example of a vertex satisfying the Schnyder condition is given in \cref{fig:schnydercondition1}.

Well-known work of Schnyder \cite{schnyder1990} established that every triangulation $T$ has a Schnyder wood. Since every Schnyder wood is a $3$-orientation by definition, it follows that every triangulation has a $3$-orientation. Conversely, it can also be shown that every $3$-orientation of $T$ induces a unique Schnyder wood of $T$ \cite{brehm2000orientations}. Schnyder's work further established the following classic theorem about the structure of Schnyder woods. We say a directed tree is \defn{rooted at} $v$ if for every vertex $u$ in the tree, the unique path in the tree from $u$ to $v$ is directed towards $v$.

\begin{theorem}[Schnyder, \cite{schnyder1990}] \label{thm:schnyderwood_structure}
    Let $T_b$, $T_r$, and $T_y$ be respectively, the subgraphs of $T$ induced by the blue, red, and yellow edges in a Schnyder wood of $T$. Then for each $c \in \{b,r,y\}$, $T_c$ is a tree rooted at vertex $v_c$ which contains every interior vertex of $T$.
\end{theorem}

A $3$-orientation or Schnyder wood of $T \in \cT_n^1$ is called \defn{maximal} if it contains no anticlockwise cycle. An example of a maximal Schnyder wood demonstrating the Schnyder condition and the structure in \cref{thm:schnyderwood_structure} is given in \cref{fig:example_schnyder_wood}. Throughout this paper we will largely be concerned with maximal Schnyder woods. The following well-known results about cycles in $3$-orientations of triangulations will be useful. 
\begin{lemma}\label{lem:directed_edges_from_cycle}
    Let $T$ be a $3$-orientation of a triangulation, and let $C$ be a cycle of length $b$ in $T$. Then $C$ has exactly $b-3$ edges directed from the cycle to its interior.
\end{lemma}
\begin{proof}
    Let $n_i$ and $m_i$ denote, respectively, the number of vertices and edges in the interior of $C$, and let $f$ be the number of faces of the subtriangulation of $T$ contained within $C$, including its exterior face. Then Euler's theorem gives that
    $$(b + n_i) - (b+m_i) + f = 2.$$
    Furthermore, since every interior face of $C$ is bordered by exactly $3$ edges, we have $$3(f-1) = 2m_i + b.$$
    Finally, since $T$ is a $3$-orientation, we have that the number of edges directed from $C$ to its interior is precisely $m_i - 3n_i$. Solving for this term using the previous equalities gives the desired result. 
\end{proof}
The next corollary follows immediately from this result.
\begin{corollary}\label{cor:most_clockwise_path}
    Let $T$ be a $3$-orientation of a triangulation, and let $C$ be a directed cycle in $T$ formed by following a path that at each vertex follows the first available outgoing edge as seen clockwise from the incoming edge. Then $C$ is a triangle.
\end{corollary}
The corollary also holds with ``clockwise'' replaced with ``anticlockwise'', and for the same reason. The following useful result of \cite{brehm2000orientations} now shows that it suffices to consider oriented triangles when looking for directed cycles in $T$.
\begin{lemma}[Lemma 1.5.1, \cite{brehm2000orientations}] \label{lem:oriented_triangles}
    Let $T$ be a $3$-orientation of a triangulation. If $T$ contains a directed cycle, then it contains a triangle with the same direction. 
\end{lemma}
\begin{proof}
    Consider a directed cycle $C$ in $T$, and suppose without loss of generality that $C$ is directed anticlockwise. If $C$ is formed by a path that always chooses the first clockwise outgoing edge, then by \cref{cor:most_clockwise_path} we are done. Otherwise, if we instead follow the most clockwise path starting at any edge on the cycle, then since we can not follow an infinite path inside of $C$, we must construct a cycle, which must be a triangle inside of $C$.
\end{proof}
We remark that each of \cref{lem:directed_edges_from_cycle}, \cref{cor:most_clockwise_path} and \cref{lem:oriented_triangles} relies only on the the interior of a given cycle being finite, and as such we will be able to reuse these results for infinite triangulations later.

In \cite{brehm2000orientations} and \cite{de1994orientations} it is shown that the set of $3$-orientations of a triangulation forms a distributive lattice. In particular, every triangulation has a unique maximal $3$-orientation, implying that this is also true for Schnyder woods. 
\begin{theorem} \label{thm:unique_maximal_wood}
    For all $n \geq 0$, every triangulation $T \in \cT_n^1$ has a unique maximal Schnyder wood.
\end{theorem}

Since we aim to define Schnyder woods of infinite triangulations which may have large boundaries, it is useful to consider Schnyder woods of finite triangulations with more than $3$ boundary vertices. Suppose that $T \in \cT_n^m$, and let $T$ have root edge $(v_r, v_y)$. Let $T' \in \cT_{n+m}^1$ be the rooted triangulation obtained by introducing a boundary vertex $v_b$ that is adjacent to every boundary vertex of $T$, and still rooted at $(v_r, v_y)$. A \defn{Schnyder wood of $T$} is an orientation and colouring of the edges of $T$ that is consistent with some Schnyder wood of $T'$. A Schnyder wood of $T$ is maximal if it contains no anticlockwise cycle. 

Note that in any Schnyder wood of $T'$, the edges $(v_b, v_r)$ and $(v_b, v_y)$ are coloured red and yellow respectively, and every other edge incident to $v_b$ is coloured blue and directed to $v_b$, so there is a bijection between Schnyder woods of $T$ and Schnyder woods of $T'$. Further, a maximal Schnyder wood of $T$ is unique.

Next, note that in every Schnyder wood of any triangulation $T \in \cT_n^m$, every vertex $v$ on the boundary of $T$ other than $v_r$ and $v_y$ satisfies the following \defn{Schnyder boundary condition}:
\begin{itemize}
    \item There are exactly $2$ outgoing edges at $v$, which are coloured red and yellow, and
    \item In anticlockwise order starting from the boundary, edges incident with $v$ consist of incoming yellow edges, the outgoing red edge, incoming blue edges, the outgoing yellow edge, and incoming red edges.
\end{itemize}
Furthermore, the vertices $v_r$ and $v_y$ satisfy the same conditions as for $3$-orientations of triangulations in $\cT_n^1$, which we hereafter call the \defn{Schnyder root condition}:
\begin{itemize}
    \item $v_r$ has exactly $1$ outgoing edge and $v_y$ has no outgoing edges, and
    \item All incoming edges to $v_r$ are coloured red and all incoming edges to $v_y$ are coloured yellow.
\end{itemize}
These further conditions are also indicated in \cref{fig:schnydercondition1}, and an example construction of $T'$ from $T$ is given in \cref{fig:example_schnyder_wood}. Observe that the neighbours of $v_b$ other than $v_r$ and $v_y$ satisfy the Schnyder boundary condition in $T$. We will refer to the Schnyder, Schnyder boundary and Schnyder root conditions collectively as the \defn{Schnyder conditions}. The following result about the monochromatic subgraphs of $T$ follows from \cref{thm:schnyderwood_structure}. 

\begin{theorem} \label{thm:schnyder_subgraphs_on_finite_triangulation}
    Let $T_b$, $T_r$ and $T_y$ be respectively, the subgraphs induced by the blue, red and yellow edges in a Schnyder wood of $T \in T_n^m$. Then $T_r$ and $T_y$ are directed trees, rooted at $v_r$ and $v_y$, and $T_b$ is a forest consisting of trees rooted at boundary vertices other than $v_r$ and $v_y$. Furthermore, for each $c \in \{b,r,y\}$, the graph $T_c$ contains every interior vertex of $T$.
\end{theorem}

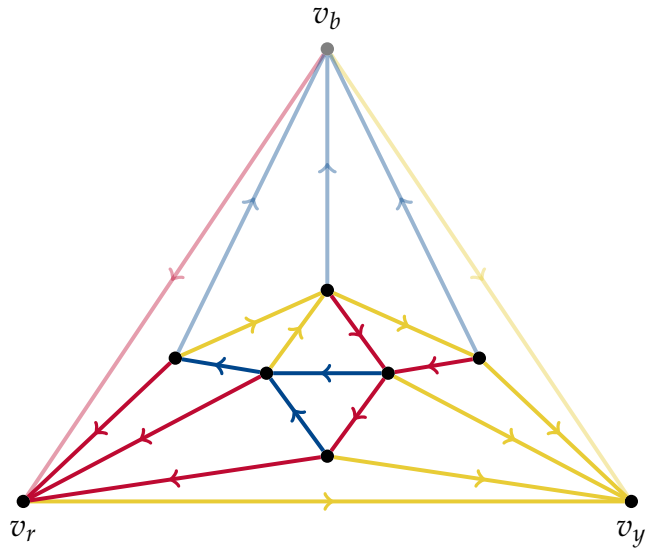
\begin{figure}[h]
    \centering
    \begin{tikzpicture}
    \node[mycircle, label=below:{$v_r$}] (vr) at (-4,0){};
    \node[mycircle, label=below:{$v_y$}] (vy) at (4,0){};
    \node[mycircle, label=above:{$v_b$}, opacity=0.5] (vb) at (0,6){};
    \node[mycircle] (v1) at (-2,1.9){};
    \node[mycircle] (v2) at (0,2.8){};
    \node[mycircle] (v3) at (2,1.9){};
    \node[mycircle] (v4) at (-0.8,1.7){};
    \node[mycircle] (v5) at (0.8,1.7){};
    \node[mycircle] (v6) at (0,0.6){};

    \path[every node/.style={font=\sffamily\small}]
        (vr) edge [->-, yedge] (vy)
        (vb) edge [->-, redge, opacity=0.4] (vr)
        (vb) edge [->-, yedge, opacity=0.4] (vy)
        (v1) edge [->-, bedge, opacity=0.4] (vb)
        (v2) edge [->-, bedge, opacity=0.4] (vb)
        (v3) edge [->-, bedge, opacity=0.4] (vb)
        (v1) edge [->-, redge] (vr)
        (v4) edge [->-, redge] (vr)
        (v6) edge [->-, redge] (vr)
        (v3) edge [->-, yedge] (vy)
        (v5) edge [->-, yedge] (vy)
        (v6) edge [->-, yedge] (vy)
        (v1) edge [->-, yedge] (v2)
        (v2) edge [->-, yedge] (v3)
        (v4) edge [->-, yedge] (v2)
        (v2) edge [->-, redge] (v5)
        (v5) edge [->-, bedge] (v4)
        (v6) edge [->-, bedge] (v4)
        (v5) edge [->-, redge] (v6)
        (v4) edge [->-, bedge] (v1)
        (v3) edge [->-, redge] (v5);
    
\end{tikzpicture}
    \caption{A maximal Schnyder wood of a triangulation $T' \in \cT_6^1$ constructed by adding $v_b$ as a neighbour to every boundary vertex of a triangulation $T \in \cT_3^3$.}
    \label{fig:example_schnyder_wood}
\end{figure}

Next, we describe an algorithm that constructs the maximal Schnyder wood of $T \in \cT_n^m$. Our algorithm is based on an algorithm described in \cite{brehm2000orientations} for finding maximal Schnyder woods, which is in turn based on the original algorithm described by Schnyder \cite{schnyder1990}. We have made minor modifications to these more well-known algorithms to design a process that maintains necessary distributional properties of the UIHPT. 

Let $T$ have boundary vertices $v_0, \dots, v_{m+1}$, listed in anticlockwise order of appearance around the exterior face, and root edge $(v_1, v_2)$ coloured yellow. The \defn{Schnyder peeling process} on $T$ is defined recursively as follows. We call $v_0$ the \defn{peeling vertex} used by this step of the Schnyder peeling process, and for any edge $e = v_0 u$ incident with $v_0$, we call $e$ a \defn{chord} if $u = v_c$ for some boundary vertex $v_c$ with $c \neq 1$. Note that we do consider the edge $e = v_0 v_{m+1}$ to be a chord. 
\begin{itemize}
    \item Consider the first chord $v_0 v_c$ anticlockwise from $v_0v_1$, so $c \geq 2$ is the smallest index such that $v_0 v_c$ is an edge. This splits $T$ into a triangulation $T_\ell$ rooted at $(v_0, v_c)$ with $m-c+3$ boundary vertices to the left of $(v_0, v_c)$, and a triangulation $T_r'$ rooted at $(v_1, v_2)$ with $c+1$ boundary vertices to the right of $(v_0, v_c)$. (Note that in the trivial case $c=m+1$, $T_\ell$ may be empty.)
    \item Suppose that $v_0$ is connected to $k \geq 0$ interior vertices of $T_r'$. Let $T_r$ be the triangulation obtained by deleting $v_0$ from $T_r'$. Then, $T_r$ is rooted at $(v_1, v_2)$ and has $c+k$ boundary vertices.
    \item Assign the following colours and directions. Let $(v_0, v_c)$ be yellow, $(v_0, v_1)$ be red, and for each neighbour $u_1,\dots,u_k$ of $v_0$ in $T_r$, let $(u_i, v_0)$ be blue.
    \item Apply the Schnyder peeling process to $T_\ell$ and $T_r$.
\end{itemize}
Note that in the base case where $T$ is a triangle with no interior vertices, then we choose the chord $(v_0, v_2)$, and the algorithm colours the edge $(v_0, v_2)$ yellow and $(v_0, v_1)$ red as desired. An example step of the Schnyder peeling process is given in \cref{fig:finite_peeling_step}. Prior to the Schnyder peeling process being recursively applied to any rooted sub-triangulations that arise in the course of the Schnyder peeling process, we refer to these sub-triangulations as \defn{unexplored regions} of $T$; see for example, the grey regions indicated in \cref{fig:finite_peeling_step}. Note that at each stage, we may choose to first explore $T_\ell$ and afterwards continue exploring in $T_r$, or vice-versa. 

\begin{figure}[h]
    \centering
    \begin{tikzpicture}[line width=1]
    \begin{scope}[rotate=180]
        \node[mycircle, opacity=0.4] (0) at (canvas polar cs: radius=3cm,angle=0){};
        \node[mycircle, label=below:$v_1$] (1) at (canvas polar cs: radius=3cm,angle=60){};
        \node[mycircle, label=below:$v_2$] (2) at (canvas polar cs: radius=3cm,angle=120){};
        \foreach \x/\y in {180/3, 240/4, 300/5}{ 
            \node[mycircle] (\y) at (canvas polar cs: radius=3cm,angle=\x){};
        }
    \end{scope}

    \node[mycircle, label={[]$v_c$}] (4repeat) at (canvas polar cs: radius=3cm,angle=60){};

    \node[mycircle, opacity=0.4] (6) at (-1.8,0.3){};
    \node[mycircle] (7) at (-0.4,1.4){};
    \node[mycircle, label={[xshift=-0.25cm, yshift=-0.05cm]$v_0$}] (8) at (-0.4,0){};
    \node[mycircle] (9) at (0.4,-1.2){};
    \node[mycircle] (10) at (1.2,0){};

    \path[every node/.style={font=\sffamily\small}]
        (1) edge [->-, yedge, line width=2] (2)
        (0) edge [->-, redge, opacity=0.4] (1)
        (6) edge [->-, redge, opacity=0.4] (1)
        (6) edge [->-, bedge, opacity=0.4] (0)
        (6) edge [->-, yedge, opacity=0.4] (5)
        (0) edge [->-, yedge, opacity=0.4] (5)
        (7) edge [->-, bedge, opacity=0.4] (6)
        (8) edge [->-, bedge, opacity=0.4] (6)
        (2) edge [-, line width=2] (3)
        (3) edge [-, line width=2] (4)
        (4) edge [-, line width=2] (5)
        (4) edge [-, line width=0.8] (7)
        (5) edge [-, line width=2] (7)
        (8) edge [-, line width=2] (7)
        (1) edge [-, line width=2] (8)
        (4) edge [-, line width=0.8] (8)
        (4) edge [-, line width=0.8] (10)
        (3) edge [-, line width=0.8] (10)
        (8) edge [-, line width=0.8] (10)
        (1) edge [-, line width=0.8] (9)
        (8) edge [-, line width=0.8] (9)
        (10) edge [-, line width=0.8] (9)
        (2) edge [-, line width=0.8] (9)
        (2) edge [-, line width=0.8] (10);

        \fill[fill=gray, opacity=0.1]  (8.center) to (1.center) to (2.center) to (3.center) to (4.center) to (5.center) to (7.center) to cycle;

    \begin{scope}[xshift=8cm]
        \begin{scope}[rotate=180]
            \node[mycircle, opacity=0.4] (0) at (canvas polar cs: radius=3cm,angle=0){};
            \node[mycircle, label=below:$v_1$] (1) at (canvas polar cs: radius=3cm,angle=60){};
            \node[mycircle, label=below:$v_2$] (2) at (canvas polar cs: radius=3cm,angle=120){};
            \foreach \x/\y in {180/3, 240/4, 300/5}{ 
                \node[mycircle] (\y) at (canvas polar cs: radius=3cm,angle=\x){};
            }
        \end{scope}

    \node[mycircle, label={[]$v_c$}] (4repeat) at (canvas polar cs: radius=3cm,angle=60){};

    \node[mycircle, opacity=0.4] (6) at (-1.8,0.3){};
    \node[mycircle] (7) at (-0.4,1.4){};
    \node[mycircle, label={[xshift=-0.25cm, yshift=-0.05cm]$v_0$}] (8) at (-0.4,0){};
    \node[mycircle] (9) at (0.4,-1.2){};
    \node[mycircle] (10) at (1.2,0){};

    \path[every node/.style={font=\sffamily\small}]
        (1) edge [->-, yedge, line width=2] (2)
        (0) edge [->-, redge, opacity=0.4] (1)
        (6) edge [->-, redge, opacity=0.4] (1)
        (6) edge [->-, bedge, opacity=0.4] (0)
        (6) edge [->-, yedge, opacity=0.4] (5)
        (0) edge [->-, yedge, opacity=0.4] (5)
        (7) edge [->-, bedge, opacity=0.4] (6)
        (8) edge [->-, bedge, opacity=0.4] (6)
        (2) edge [-, line width=2] (3)
        (3) edge [-, line width=2] (4)
        (4) edge [-, line width=2] (5)
        (4) edge [-, line width=0.8] (7)
        (5) edge [-, line width=2] (7)
        (8) edge [-, line width=2] (7)
        (8) edge [->-, redge] (1)
        (8) edge [->-, yedge, line width=2] (4)
        (4) edge [-, line width=2] (10)
        (3) edge [-, line width=0.8] (10)
        (10) edge [->-, bedge] (8)
        (1) edge [-, line width=2] (9)
        (9) edge [->-, bedge] (8)
        (10) edge [-, line width=2] (9)
        (2) edge [-, line width=0.8] (9)
        (2) edge [-, line width=0.8] (10);

    \fill[fill=gray, opacity=0.1]  (1.center) to (2.center) to (3.center) to (4.center) to (10.center) to (9.center) to cycle;

    \fill[fill=gray, opacity=0.1]  (8.center) to (4.center) to (5.center) to (7.center) to cycle;

    \node[] (Tr) at (1.8,-1){$T_r$};
    \node[] (Tl) at (0,2){$T_\ell$};
    
    \end{scope}
    
\end{tikzpicture}
    \caption{The first non-trivial Schnyder peeling step on a triangulation $T$. Initially the root edge is $(v_1, v_2)$; the boundary of the unexplored region is indicated in bold. The first chord $(v_0, v_c) = (v_0, v_4)$ anticlockwise from $(v_0, v_1)$ in the unexplored region is selected and coloured yellow, and all remaining neighbours of $v_0$ are also explored and coloured, splitting the unexplored region in two. The left unexplored region $T_\ell$ is rooted at the new root $(v_0, v_c)$ while the right unexplored region $T_r$ retains the original root $(v_1,v_2)$. }
    \label{fig:finite_peeling_step}
\end{figure}
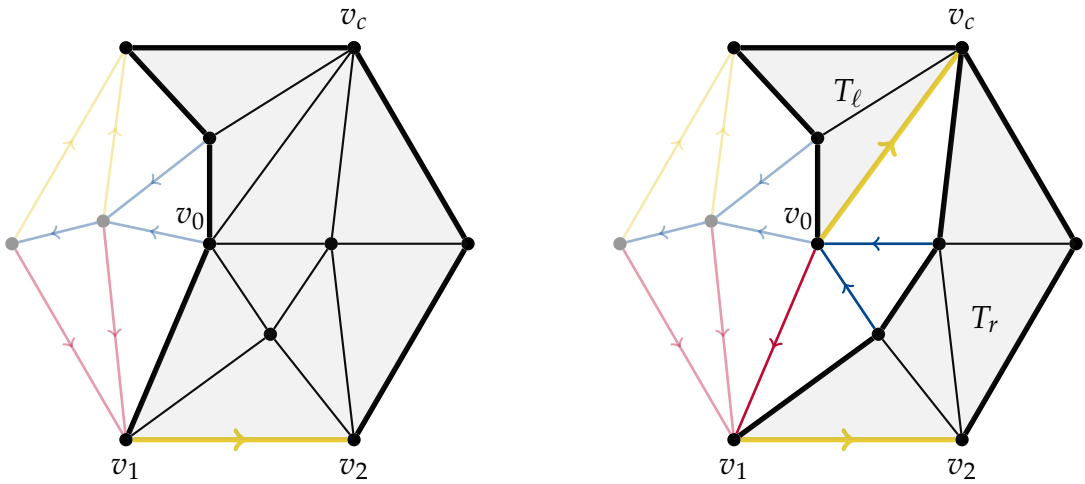

We include a proof that the Schnyder peeling process constructs the unique maximal Schnyder wood of a triangulation for completeness. However, the result is well-known for the very similar algorithms described in \cite{brehm2000orientations} and \cite{schnyder1990}.

\begin{theorem}\label{thm:peeling_process_maximal_wood}
    The Schnyder peeling process constructs the unique maximal Schnyder wood.
\end{theorem}

\begin{proof}
    Let $T \in \cT_n^m$, and let the boundary of $T$ be $v_0, v_1, \dots, v_{m+1}$ with root edge $(v_1, v_2)$.
    
    \begin{claim} \label{claim:peeling_schnyder_wood}
        The Schnyder peeling process constructs a Schnyder wood of $T$.
    \end{claim}
    
    \begin{poc}
        We proceed by induction on $n+m$. If $n+m=1$, it is trivial to check that the Schnyder peeling process constructs the unique colouring where $v_1$ and $v_2$ satisfy the Schnyder root condition and $v_0$ satisfies the Schnyder boundary condition. Suppose this statement holds for $n+m<j$, and consider $T \in \cT_n^{m}$ with $n+m = j$.

        The first step of the Schnyder peeling process on $T$ selects a chord $(v_0, v_c)$, and reveals $k$ interior neighbours of $v_0$, say $u_1, \dots, u_k$. This splits $T$ into two triangulations $T_\ell$ and $T_r$ in $\cT_{n_\ell}^{m-c+1}$ and $\cT_{n-n_\ell-k}^{c+k-2}$ respectively for some $n_\ell \in \{0,\dots,n-k\}$, with $T_\ell$ rooted at $(v_0, v_c)$ and $T_r$ rooted at $(v_1, v_2)$. At this stage, we have that $(v_1, v_2)$ and $(v_0, v_c)$ are coloured yellow, $(v_0, v_1)$ is coloured red, and $(u_i, v_0)$ is coloured blue for each $i \in \set{1,\dots,k}$. By the induction hypothesis, the Schnyder peeling process will colour both $T_\ell$ and $T_r$ so that they satisfy the Schnyder root, Schnyder boundary, and Schnyder conditions. 

        It remains to check the colouring of the edges incident with $v_0, v_1, v_2, v_c, u_1, \dots, u_k$, since all other vertices are handled entirely by the induction hypothesis.
        \begin{itemize}
            \item If $c > 2$, the vertices $v_1$ and $v_2$ satisfy the Schnyder root condition in $T_r$. Every other edge incident with them is directed towards $v_1$ or $v_2$ and coloured red and yellow respectively. Hence, $v_1$ and $v_2$ satisfy the Schnyder root condition in $T$. Otherwise, if $c=2$, then $v_1$ still satisfies the Schnyder root condition for the same reason, and since $v_2$ is the head of the root edge in both $T_\ell$ and $T_r$, every edge incident with $v_2$ is coloured yellow and directed into $v_2$, so $v_2$ again satisfies the root condition.
            \item The vertex $v_c$ (if $c \neq 2$) satisfies the Schnyder boundary condition in $T_r$. Since $v_c$ satisfies the Schnyder root condition in $T_\ell$, every other edge incident to $v_c$ is directed towards $v_c$ and coloured yellow. Hence, $v_c$ satisfies the Schnyder boundary condition in $T$.
            \item The vertex $v_0$ satisfies the Schnyder root condition in $T_\ell$, so every edge incident to $v_0$ in $T_\ell$, apart from $v_0 v_c$, is directed towards $v_0$ and coloured red. Hence, together with the other edges incident with $v_0$, it follows that $v_0$ satisfies the Schnyder boundary condition in $T$.
            \item Finally, each of the vertices $u_1, \dots, u_k$ satisfies the Schnyder boundary condition in $T_r$. Since each of these vertices is connected with an outgoing blue edge to $v_0$, it follows that $u_1, \dots, u_k$ satisfy the Schnyder condition in $T$. \qedhere
        \end{itemize}
    \end{poc}

    \begin{claim} \label{clm:maximal_peeling_process}
        The Schnyder wood produced by the Schnyder peeling process on $T$ is maximal.
    \end{claim}
    
    \begin{poc}
        By \cref{lem:oriented_triangles}, it suffices to rule out the existence of anticlockwise triangles in $T$. Suppose for a contradiction that $T$ is minimal with the property that the Schnyder peeling process creates an anticlockwise triangle $\Delta$ on $T$. Then the first step of the peeling process must colour at least one edge of $\Delta$, otherwise $\Delta$ is contained entirely within the subtriangulation $T_\ell$ or $T_r$, contradicting the minimality of $T$. Hence, one of the edges of $\Delta$ is either $(v_0, v_c)$, $(v_0, v_1)$, or $(u_i, v_0)$ for some neighbour $u_i$ of $v_0$ in $T_r$.
        
        Now, since $v_c$ satisfies the Schnyder root condition in $T_\ell$, it is not in any directed cycle in $T_\ell$. Furthermore, by the orientation of $(v_0, v_c)$, this edge cannot be in anticlockwise cycles outside of $T_\ell$, and hence cannot lie on $\Delta$. The only remaining possibility for the anticlockwise triangle $\Delta$ is that two of its edges are $(u_i, v_0)$ and $(v_0, v_1)$ for some neighbour $u_i$ of $v_0$ in $T_r$. However, by the Schnyder root condition, an edge between $v_1$ and $u_i$ would be directed towards $v_1$ and so would have the wrong direction. Hence, $\Delta$ does not exist.
    \end{poc}
    
    Hence, the colouring produced on $T$ by the Schnyder peeling process is a maximal Schnyder wood. Uniqueness follows from \cref{thm:unique_maximal_wood} and the bijection between $\cT_n^m$ and $\cT_{n+m}^1$.
\end{proof}

We conclude this section with some results on the structure of monochromatic paths in the maximal Schnyder wood of a triangulation $T \in \cT_n^m$. For each vertex $v$ of $T$ and each colour $c \in \{y, r, b\}$, let $\vpath{c}{v}$\glossarylabel{gl:vpath} be the unique directed path of colour $c$ that starts at $v$ in the maximal Schnyder wood of $T$. Note that if the root edge of $T$ is $(v_r,v_y)$, then $\vpath{y}{v}$ terminates at $v_y$, $\vpath{r}{v}$ terminates at $v_r$, and $\vpath{b}{v}$ terminates at a boundary vertex of $T$ other than $v_r$ or $v_v$. The next lemma will be useful in later sections.

\begin{lemma}\label{lem:finite_yellow_path}
    Let $T \in \cT_n^m$ have root edge $(v_r,v_y)$, and let $v$ be a vertex on the boundary of $T$. Then, $\vpath{y}{v}$ consists entirely of boundary vertices clockwise between $v$ and $v_y$ on $T$.
\end{lemma}

\begin{proof}
    It follows from \cref{thm:peeling_process_maximal_wood} that we can assume that the edges of $T$ are coloured by the Schnyder peeling process. We induct on the total number of vertices of $T$, that is, on $n + m + 2$. Clearly, the statement holds when $n=0$, as all vertices of $T$ are boundary vertices. Suppose that the statement holds for all triangulations with fewer than $n + m + 2$ vertices, and consider $T \in \cT_n^m$. Suppose $T$ has boundary $v_0,\dots,v_{m+1}$ in anticlockwise order, and root edge $(v_1,v_2)$. Applying the Schnyder peeling process to $T$, we choose a chord $(v_0,v_c)$, reveal some $k \geq 0$ interior vertices, and obtain two unexplored regions of $T$ that we call $T_r$ and $T_\ell$. Note that $T_r$ is rooted at $(v_1,v_2)$, $T_\ell$ is rooted at $(v_0,v_c)$, and $T_r$ does not contain $v_0$, while $T_\ell$ does not contain $v_1$. It follows from the induction hypothesis that for any vertex $v$ on the boundary of $T_\ell$, the directed yellow path from $v$ to $v_c$ consists of boundary vertices on $T_\ell$ clockwise between $v$ and $v_c$. Similarly, for any vertex $v$ on the boundary of $T_r$, the directed yellow path from $v$ to $v_2$ consists of boundary vertices of $T_r$ clockwise between $v$ and $v_2$. In particular, since $v_c$ is on the boundary of $T_r$, the directed yellow path from $v_c$ to $v_2$ consists of boundary vertices of $T_r$.

    Suppose $v$ is on the boundary of $T$. If $v$ is on the boundary of $T_\ell$, then every vertex clockwise between $v$ and $v_c$ on the boundary of $T_\ell$ is also on the boundary of $T$. Furthermore, every vertex on the boundary of $T_r$ clockwise between $v_c$ and $v_2$ is on the boundary of $T$, and so $\vpath{y}{v}$ runs clockwise along the boundary of $T$ from $v$ to $v_2$, as desired. Otherwise, $v$ is on the boundary of $T_r$. Then, every vertex clockwise between $v$ and $v_2$ on the boundary of $T_r$ is on the boundary of $T$, and so $\vpath{y}{v}$ again satisfies the claim.
\end{proof}

We also deduce that the monochromatic paths in the maximal Schnyder wood satisfy the following structure. As far as we are aware, this structure has not previously been described elsewhere.

\begin{theorem}\label{thm:finite_paths_from_v_wind}
    Let $T \in \cT_n^m$, and let $v$ be a vertex of $T$. Then every vertex on $\vpath{b}{v}$ has an outgoing red edge to $\vpath{r}{v}$, every vertex on $\vpath{r}{v}$ has an outgoing yellow edge to $\vpath{y}{v}$, and every vertex on $\vpath{y}{v}$ is either a boundary vertex of $T$ or has an outgoing blue edge to $\vpath{b}{v}$.
\end{theorem}

\begin{proof}
    First, we show that $\vpath{y}{v}$, $\vpath{r}{v}$, and $\vpath{b}{v}$ are vertex-disjoint except at $v$. Indeed, suppose that two of these paths intersect, say $\vpath{y}{v}$ and $\vpath{r}{v}$, and let $u$ be the first vertex at which they intersect. Then, the path segments of $\vpath{y}{v}$ and $\vpath{r}{v}$ from $v$ to $u$ form a cycle $C$. By the Schnyder conditions, for every vertex of $C$ apart from $v$ and $u$, exactly one edge is directed from that vertex to the interior of $C$. So, if $C$ has length $b$, there are at least $b - 2$ edges directed from $C$ to its interior. This contradicts \cref{lem:directed_edges_from_cycle}.

    We now prove the lemma by an induction argument similar to the one used in the proof of \cref{lem:finite_yellow_path}. By \cref{thm:peeling_process_maximal_wood}, we can assume that the edges of $T$ are coloured by the Schnyder peeling process. We induct on $n + m + 2$. The statement clearly holds when $n = 0$ and $m = 1$. Suppose that the statement holds for all triangulations with fewer than $n + m + 2$ vertices. Let $T \in \cT_n^m$ have boundary $v_0, \dots, v_{m+1}$ in anticlockwise order and root edge $(v_1,v_2)$, and suppose that when applying the Schnyder peeling process to $T$, we choose a chord $(v_0,v_c)$, reveal $k \ge 0$ interior vertices $u_1, \dots, u_k$, and obtain two unexplored regions $T_r$ and $T_\ell$ of $T$, where $T_r$ is rooted at $(v_1,v_2)$ and $T_\ell$ is rooted at $(v_0,v_c)$. By the induction hypothesis, we know that the statement holds for $T_r$ and $T_\ell$.

    First suppose that $v$ is in $T_\ell$. For $c \in \{y, r, b\}$, let $Q_c$ be the unique directed path of colour $c$ in $T_\ell$ that starts at $v$. In particular, $\qpath{y}$ terminates at $v_c$, $\qpath{r}$ terminates at $v_0$, and $\qpath{b}$ terminates at the boundary of $T_\ell$. Note that $\vpath{b}{v} = \qpath{b}$, $\vpath{r}{v} = \qpath{r} v_1$, and $\vpath{y}{v} = \qpath{y} Y$ where $Y$ is the directed yellow path from $v_c$ to $v_2$ in $T_r$. In particular, every vertex on $\vpath{b}{v} = \qpath{b}$ has an outgoing red edge to $\qpath{r} \subseteq \vpath{r}{v}$. Every vertex on $\vpath{r}{v}$ is either on $\qpath{r}$ and so has an outgoing yellow edge to $\qpath{y} \subseteq \vpath{y}{v}$, or is $v_1$ and so has an outgoing yellow edge to $v_2 \in \vpath{y}{v}$. Finally, by \cref{lem:finite_yellow_path}, we know that $Y$ consists entirely of boundary vertices of $T$. So, every vertex on $\vpath{y}{v}$ that is not a boundary vertex of $T$ is on $\qpath{y}$, and therefore has an outgoing blue edge to $\qpath{b} = \vpath{b}{v}$.

    Otherwise, $v$ is in $T_r$. For $c \in \{y, r, b\}$, let $Q_c$ be the unique directed path of colour $c$ in $T_r$ that starts at $v$. In particular, $\qpath{y}$ terminates at $v_2$, $\qpath{r}$ terminates at $v_1$, and $\qpath{b}$ terminates at the boundary of $T_r$. Note that $\vpath{r}{v} = \qpath{r}$ and $\vpath{y}{v} = \qpath{y}$. Moreover, $\vpath{b}{v} = \qpath{b} v_0$ if $\qpath{b}$ terminates at one of the vertices $u_1, \dots, u_k$, and $\vpath{b}{v} = \qpath{b}$ otherwise. In particular, every vertex on $\vpath{b}{v}$ is either on $\qpath{b}$ and so has an outgoing red edge to $\qpath{r} = \vpath{r}{v}$, or is $v_0$ and so has an outgoing red edge to $v_1 \in \vpath{r}{v}$. Also, every vertex on $\vpath{r}{v} = \qpath{r}$ has an outgoing yellow edge to $\qpath{y} = \vpath{y}{v}$.
    
    Finally, consider a vertex $w$ on $\vpath{y}{v} = \qpath{y}$ that is not a boundary vertex of $T$. If $w$ is also not a boundary vertex of $T_r$, then $w$ has an outgoing blue edge to $\qpath{b} \subseteq \vpath{b}{v}$. Otherwise, $w = u_i$ for some $i \in [k]$, and so $w$ has an outgoing blue edge to $v_0$. It remains to show that $v_0 \in \vpath{b}{v}$, or, equivalently, that $\qpath{b}$ terminates at one of the vertices $u_1, \dots, u_k$. Suppose for a contradiction that this is not the case, and so $\qpath{b}$ terminates at a boundary vertex $v_j$ for some $3 \le j \le c$. Then, by the Schnyder conditions at vertex $v$, the first edge of $\qpath{r}$ is directed to the interior of the region enclosed by the path $\qpath{b}$ from $v$ to $v_j$, the path segment of $\qpath{y}$ from $v$ to $u_i$, and the segment of the boundary vertices of $T_r$ clockwise between $u_i$ and $v_j$. However, $\qpath{r}$ terminates at $v_1$, and so this would imply by planarity that $\qpath{r}$ intersects $\qpath{b}$ or $\qpath{y}$. This contradicts the fact that $\vpath{y}{v}$, $\vpath{r}{v}$, and $\vpath{b}{v}$ are vertex-disjoint as shown at the beginning of the proof.
\end{proof}

\section{Schnyder woods of infinite triangulations}\label{sec:infiniteschnyderwoods}

In this section we introduce Schnyder woods of infinite triangulations. For now, these triangulations may have either finite boundary (triangulations of polygons), or infinite boundary (triangulations of the half-plane). We will establish differing results on these two classes of triangulation in later sections. We begin in \cref{ssec:infiniteschnyderwoods:definitions} by defining Schnyder woods of infinite triangulations, and generalising the notion of maximality introduced in \cref{ssec:prelim:schnyderwoods}. We show that our definitions imply that a maximal Schnyder wood of an infinite triangulation is unique, provided it exists. In \cref{ssec:infiniteschnyderwoods:peelingproc}, we generalise the Schnyder peeling process introduced for finite triangulations in \cref{ssec:prelim:schnyderwoods}, and describe some of the coloured substructures it produces as it explores an infinite triangulation.

\subsection{Defining infinite Schnyder woods} \label{ssec:infiniteschnyderwoods:definitions}

Recall from \cref{ssec:prelim:schnyderwoods} that a Schnyder wood of a finite triangulation with boundary, $T$, rooted at $(v_r, v_y)$, is a $3$-colouring of the edges of a $3$-orientation of $T$ in red, yellow, and blue satisfying the following Schnyder conditions.

\textbf{The Schnyder condition:}
\begin{itemize}
    \item The $3$ outgoing edges at each interior vertex are coloured blue, red, and yellow in anticlockwise cyclic order, and
    \item Every incoming edge at each interior vertex enters between the outgoing edges of the two other colours.
\end{itemize}

\textbf{The Schnyder root condition:}
\begin{itemize}
    \item $v_r$ has exactly $1$ outgoing edge and $v_y$ has no outgoing edges, and
    \item All incoming edges to $v_r$ are coloured red and all incoming edges to $v_y$ are coloured yellow.
\end{itemize}

\textbf{The Schnyder boundary condition:} 
For any boundary vertex $v$ apart from $v_r$ and $v_y$,
\begin{itemize}
    \item There are exactly $2$ outgoing edges at $v$ which are coloured red and yellow, and
    \item In anticlockwise order starting from the boundary, edges incident with $v$ consist of incoming yellow edges, the outgoing red edge, incoming blue edges, the outgoing yellow edge, and incoming red edges.
\end{itemize}

We now define Schnyder woods of infinite triangulations analogously. A \defn{Schnyder wood of an infinite triangulation $T$ rooted at $(v_r, v_y)$} is a colouring and orientation of the edges of $T$ that satisfies the Schnyder conditions. Recall from \cref{thm:schnyder_subgraphs_on_finite_triangulation} that for finite triangulations, the monochromatic subgraphs of a Schnyder wood consist of rooted forests with a certain structure. We will prove analogous structural results for the monochromatic subgraphs on various infinite triangulations in later sections. However, for now, we observe that the Schnyder wood definition immediately implies that these subgraphs are still forests in the infinite case.

\begin{lemma}\label{lem:infiniteschnyderwood_monochromatic}
    Every monochromatic subgraph of a Schnyder wood of an infinite triangulation is a forest. 
\end{lemma}
\begin{proof}
    Suppose that a Schnyder wood contains a monochromatic blue cycle. Since the Schnyder conditions imply that each vertex has at most one blue outgoing edge, it follows that this cycle must be directed. The Schnyder conditions therefore imply that at each vertex on the cycle, the yellow outgoing edge is the unique outgoing edge directed into the interior of the cycle, contradicting \cref{lem:directed_edges_from_cycle}. A similar contradiction is obtained by assuming the presence of a monochromatic red or yellow cycle. It follows that the monochromatic subgraphs of any Schnyder wood are forests.
\end{proof}

Recall also that a Schnyder wood of a finite triangulation is called maximal if every directed cycle in the Schnyder wood is oriented clockwise.  We similarly consider an analogous definition for Schnyder woods of infinite triangulations. We say that a two-ended infinite directed path in an infinite triangulation $T$ is \defn{right-directed} (resp. \defn{left-directed}) if it divides $T$ into two regions so that the region containing the boundary of $T$ lies to the right (resp. left) of the path. Observe that in an infinite triangulation, a left-directed path can be thought of as an infinite clockwise cycle, and a right-directed path can be thought of as an infinite anticlockwise cycle. We therefore say a Schnyder wood of an infinite triangulation is \defn{maximal} if it contains no anticlockwise cycles or right-directed paths. \cref{fig:distinct_max_schnyder_halfplane} gives an example of a triangulation of the half-plane with two distinct Schnyder woods containing no finite anticlockwise cycles; however, this is only possible due to the Schnyder wood on the right containing right-directed paths. 

\begin{figure}[h]
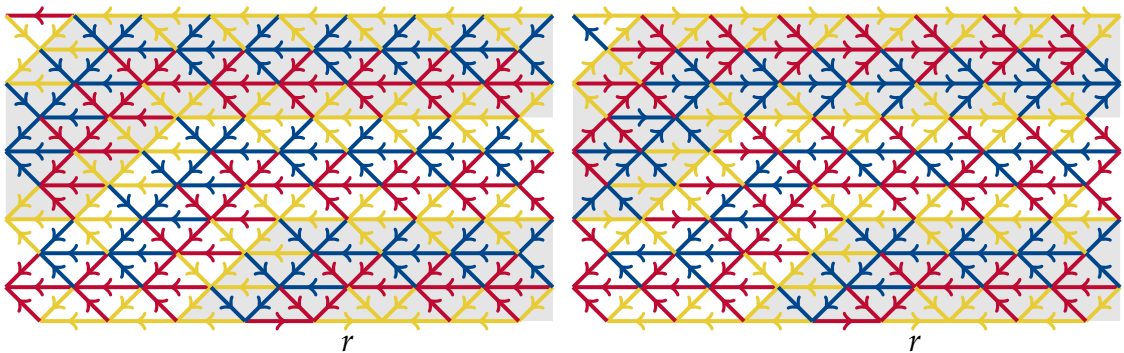

    \centering
    \input{Figures/maximal_no_rightpath_example.tex}
    \input{Figures/maximal_rightpath_example.tex}
    \caption{Two distinct Schnyder woods of the half-plane triangular lattice in which every finite directed cycle is oriented clockwise. The left Schnyder wood is the one produced by the Schnyder chiselling algorithm that we later define in \cref{sec:structure}, and contains no right-directed paths. The right Schnyder wood contains right-directed paths.}
    \label{fig:distinct_max_schnyder_halfplane}
\end{figure}

We now establish that if a maximal Schnyder wood exists on some infinite triangulation $T$, then it is unique. The argument is the same as the standard proof that maximal Schnyder woods of finite triangulations are unique (see \cite{brehm2000orientations}), but applied using our definition of maximality for infinite triangulations.
\begin{theorem}\label{thm:max_schnyder_unique_infinite}
    If an infinite triangulation $T$ has a maximal Schnyder wood, then $T$ has a unique maximal Schnyder wood.
\end{theorem}
\begin{proof}
    Suppose that there exist two distinct maximal Schnyder woods $S_1$ and $S_2$ on $T$, and consider the subgraph $H$ of $T$ consisting of edges that are oriented differently in $S_1$ and $S_2$. Since both $S_1$ and $S_2$ are Schnyder woods, every vertex $v$ in $T$ has the same number of outgoing edges in both $S_1$ and $S_2$ (with the exact number depending on whether $v$ is an interior, boundary, or root vertex of $T$). To preserve these outdegrees, for every edge incident with $v$ that is oriented into $v$ in $S_1$ and out of $v$ in $S_2$, there must be another edge that is oriented out of $v$ in $S_1$ and into $v$ in $S_2$, and vice-versa. It follows that every vertex in $H$ must have even degree, and that $H$ can be decomposed into subgraphs that correspond to directed cycles or two-ended infinite directed paths in $S_1$ and $S_2$. Since the edges of these cycles and paths in $S_1$ and $S_2$ have opposite orientations, there must be an anticlockwise cycle or a right-directed path in $S_1$ or $S_2$, and so it follows that $S_1$ and $S_2$ cannot both be maximal. It follows that a maximal Schnyder wood of any infinite triangulation $T$ is unique.
\end{proof}

\subsection{The Schnyder peeling process on infinite triangulations}\label{ssec:infiniteschnyderwoods:peelingproc}

We now describe a modification of the Schnyder peeling process which allows it to explore infinite one-ended triangulations. We first describe the details of this modified peeling process, and provide terminology and notation that will be used to discuss the regions explored by this process in detail in future sections. We also provide an overview of the structures explored by this process after a finite number of steps have elapsed. We will make use of these structures in later sections when we attempt to explore infinite triangulations in their entirety.

\subsubsection{One step of the peeling process}\label{ssec:peelproc:singlestep}

Suppose that we are given an infinite one-ended triangulation $T$ of a polygon with finite boundary or of the half-plane, rooted at an edge on the boundary. In the following, we describe a single step of the Schnyder peeling process on $T$. Such a step will reveal a part of $T$ near the root edge and assign colours and directions to the revealed edges. This will split $T$ into a finite region, which is explored immediately using the Schnyder peeling process for finite triangulations, and an infinite region which remains unexplored. At the end of the step, the boundary of $T$ is updated to the boundary of the unexplored infinite region, and a certain edge on that new boundary is designated as the new root edge, so that the process explores within this (unexplored) infinite region in future steps.

To define a peeling step, let the boundary of $T$ be $\dots, v_{-2}, v_{-1}, v_0, v_1, v_2, \dots$ with root edge $(v_1, v_2)$ coloured yellow. We again call $v_0$ the \defn{peeling vertex} used by this peeling step, and define chords incident with $v_0$ similarly to the finite case; that is, an edge $v_0 u$ is a \defn{chord} if $u = v_c$ for some boundary vertex $v_c$ with $c \neq 1$. Consider the first chord $v_0v_c$ anticlockwise from $v_0v_1$ that connects $v_0$ to the boundary, and let $u_1, \dots, u_k$ be the neighbours of $v_0$ between the edges $v_0v_1$ and $v_0v_c$ in anticlockwise order. We assign the following colours and directions. Let $(v_0, v_c)$ be yellow, $(v_0, v_1)$ be red, and for each $i \in [k]$ let $(u_i, v_0)$ be blue. Then, depending on which side of the chord $v_0 v_c$ encloses a finite triangulation, we perform the following steps. 
\begin{itemize}
    \item If the boundary of $T$ is infinite and $c > 0$, or the boundary is finite and $v_0, v_1, \dots, v_c, v_0$ encloses a finite triangulation, we perform a \defn{right step}.

    In this case, let $T_r$ be the finite triangulation enclosed by $v_1, \dots, v_c, u_k, \dots, u_1, v_1$, rooted at $(v_1, v_2)$, and assign the unique finite maximal Schnyder wood to $T_r$. Moreover, let the updated boundary be $\dots, v_{-2}, v_{-1}, v_0, v_c, v_{c+1}, \dots$ and let the new root edge be $(v_0,v_c)$. This is the boundary of the unexplored infinite region of $T$ to the left of $(v_0, v_c)$.
    
    \item If the boundary of $T$ is infinite and $c < 0$, or the boundary is finite and $v_0, v_c, v_{c+1}, \dots, v_0$ encloses a finite triangulation, we perform a \defn{left step}.

    In this case, let $T_\ell$ be the finite triangulation enclosed by $v_0, v_c, v_{c+1}, \dots, v_0$, rooted at $(v_0, v_c)$, and assign the unique finite maximal Schnyder wood to $T_\ell$. Moreover, let the updated boundary be $\dots, v_{c-1}, v_c, u_k, \dots, u_1, v_1, v_2, \dots$ and keep the root edge $(v_1, v_2)$. This is the boundary of the unexplored infinite region of $T$ to the right of $(v_0, v_c)$.
\end{itemize}
See \cref{fig:example_schnyder_steps} for an example of a left and right step. The \defn{partial triangulation $T'$ revealed by the peeling step} is the part of $T$ enclosed between the old boundary and the updated boundary, including the parts on the boundaries. Its \defn{interior}, $\intr T'$\glossarylabel{gl:intr}, is the part of $T'$ that is not on the updated boundary. After the peeling step, the infinite region of $T$ that remains unexplored is $T \sm \intr T'$. We call this the \defn{unexplored region} of $T$. The process continues exploring the unexplored region of $T$ in the next step.

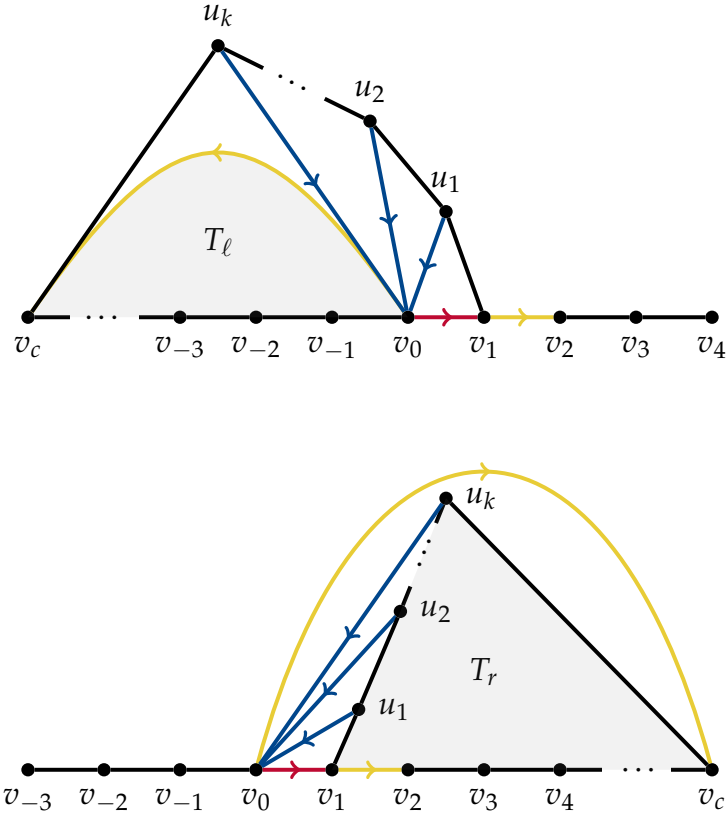
\begin{figure}[h]
    \centering
    \begin{tikzpicture}

    \foreach \x/\nm/\lbl in {0/bk/$v_c$,2/bn3/$v_{-3}$,3/bn2/$v_{-2}$,4/bn1/$v_{-1}$,5/b0/$v_0$, 6/b1/$v_1$, 7/b2/$v_2$, 8/b3/$v_3$, 9/b4/$v_4$}{
        \node[mycircle, label=below:{\lbl}] (\nm) at (\x,0){};
    }

    \path[every node/.style={font=\sffamily\small}]
        (bn3) edge (bn2)
        (bn2) edge (bn1)
        (bn1) edge (b0)
        (b0) edge [->-, redge] (b1)
        (b1) edge [->-, yedge] (b2)
        (b2) edge (b3)
        (b3) edge (b4);

    \node[mycircle, label=above:{$u_k$}] (uj) at (2.5,3.6){};
    \node[mycircle, label=above:{$u_2$}] (u2) at (4.5,2.6){};
    \node[mycircle, label=above:{$u_1$}] (u1) at (5.5,1.4){};

    \draw[dots] (bk) to (bn3);
    \draw[dots] (uj) to (u2);

    \node[] (hole) at (2.5,1){$T_\ell$};
    \fill[fill=gray, opacity=0.1]  (b0) to [bend right=55, looseness=1.8] (bk) to (bk.center) to (b0.center) to cycle;

    \draw[] (b0) edge [->-, yedge, bend right=55, looseness=1.8] node[font=\small,above, yshift=-0.08cm] {} (bk);

    \draw[] (uj) edge [->-, bedge] node[font=\small,above, yshift=-0.08cm] {} (b0);
    \draw[] (u2) edge [->-, bedge] node[font=\small,above, yshift=-0.08cm] {} (b0);
    \draw[] (u1) edge [->-, bedge] node[font=\small,above, yshift=-0.08cm] {} (b0);

    \draw[every node/.style={font=\sffamily\small}]
        (bk) to (uj)
        (u2) to (u1)
        (u1) to (b1);

    \begin{scope}[yshift=-6cm]
        \foreach \x/\nm/\lbl in {0/vn3/$v_{-3}$,1/vn2/$v_{-2}$,2/vn1/$v_{-1}$,3/v0/$v_{0}$,4/v1/$v_1$, 5/v2/$v_2$, 6/v3/$v_3$, 7/v4/$v_4$, 9/vc/$v_c$}{
        \node[mycircle, label=below:{\lbl}] (\nm) at (\x,0){};
    }

    \node[mycircle, label=right:{$u_k$}] (uj) at (5.5,3.6){};
    \node[mycircle, label=right:{$u_2$}] (u2) at (4.9,2.1){};
    \node[mycircle, label=right:{$u_1$}] (u1) at (4.35,0.8){};

    \draw[dots] (vc) to (v4);
    \draw[dots] (uj) to (u2);

    \node[] (hole) at (6,1.3){$T_r$};
    \fill[fill=gray, opacity=0.1]  (vc.center) to (v1.center) to (u1.center) to (u2.center) to (uj.center) to (vc.center);

    \path[every node/.style={font=\sffamily\small}]
        (vn3) edge (vn2)
        (vn2) edge (vn1)
        (vn1) edge (v0)
        (v0) edge [->-, redge] (v1)
        (v1) edge [->-, yedge] (v2)
        (v2) edge (v3)
        (v3) edge (v4);

    \draw[] (v0) edge [->-, yedge, bend left=75, looseness=2.3] node[font=\small,above, yshift=-0.08cm] {} (vc);

    \draw[] (uj) edge [->-, bedge] node[font=\small,above, yshift=-0.08cm] {} (v0);
    \draw[] (u2) edge [->-, bedge] node[font=\small,above, yshift=-0.08cm] {} (v0);
    \draw[] (u1) edge [->-, bedge] node[font=\small,above, yshift=-0.08cm] {} (v0);

    \draw[every node/.style={font=\sffamily\small}]
        (vc) to (uj)
        (u2) to (u1)
        (u1) to (v1);
    \end{scope}
    
\end{tikzpicture}
    \caption{A left step (top) and right step (bottom) of the Schnyder peeling process.}
    \label{fig:example_schnyder_steps}
\end{figure}

Note that during each step, the enclosed finite triangulation is assigned its unique maximal Schnyder wood, which is the same colouring that would be assigned to it by recursively applying the finite peeling process to this region. Since the only other distinction to the finite peeling process is that the remaining unexplored region is infinite, it follows that the region of $T$ explored by any finite number of steps of the Schnyder peeling process is assigned the same colouring as it would be if it were contained in a large finite triangulation\footnote{More precisely, in any finite triangulation obtained by replacing the infinite unexplored region of $T$ with an arbitrary finite triangulation.} and explored by the finite peeling process.

\subsubsection{Multiple steps of the peeling process}\label{sssec:peelproc:multiplesteps}

We now consider the behaviour of the Schnyder peeling process over multiple steps. Let $T$ be an infinite triangulation with initial boundary $B = \dots,b_{-2},b_{-1},b_0,b_1,b_2,\dots$ and initial root edge $(b_1,b_2)$. Then step $i$ of the Schnyder peeling process performs a peeling step based on the boundary and root edge before step $i$. This step updates the boundary and the process chooses a new root edge on the updated boundary for the next peeling step. In the following, we deduce some basic properties about how the Schnyder peeling process behaves over multiple steps.

We say that a vertex $v$ is \defn{covered} by some step of the Schnyder peeling process if $v$ is contained in the interior of the finite partial triangulation that has been explored after that step, and $v$ is \defn{eventually covered} if there exists some step after which it is covered. By the definition of a peeling step, a simple induction argument shows that the set of covered initial boundary vertices always forms a segment of the initial boundary.

\begin{observation}\label{obs:finitestepscoversegmentofboundary}
    After every step of the Schnyder peeling process, the set of covered initial boundary vertices forms a segment $b_\ell, \dots, b_r \subseteq B$. Moreover, if some vertices of $B$ have not been covered, then the head of the current root edge is $b_{r+1}$ (or $b_{r+2}$ if the process has not yet taken a right step).
\end{observation}

Using the fact that $T$ is locally finite, we now show that the Schnyder peeling process performs an infinite number of right steps. 

\begin{observation}\label{obs:infiniterightsteps}
    The Schnyder peeling process performs an infinite number of right steps which update the head of the current root edge to a vertex further right along the current boundary. In particular, the process will eventually cover all initial boundary vertices to the right of the initial root edge.
\end{observation}

\begin{proof}
    We claim that the Schnyder peeling process only performs a finite number of left steps before performing a right step. Indeed, every left step explores one additional edge incident to the tail of the current root edge without updating the current root edge. So, the number of consecutive left steps is at most the degree of the tail of the current root edge, which is finite. Also note that every right step updates the current root edge to a previously unexplored edge. Since the degree of the head of the current root edge is finite, this implies that after a finite number of right steps, the head of the current root edge must change. By the definition of a right step, this means that the head of the current root edge moves to a vertex further right along the current boundary.
    
    As long as some initial boundary vertices are not covered, such a right step extends the segment of covered initial boundary vertices by at least one vertex to the right. This implies that all initial boundary vertices to the right of the initial root edge will eventually be covered.
\end{proof}
We remark that this implies that if the initial boundary is finite, the process must eventually cover the entire initial boundary.

\subsubsection{Some additional notation}

In the remainder of this paper, we will prove the existence of Schnyder woods of several classes of infinite triangulation. This will require arguments that make use of various structures produced by the Schnyder peeling process defined above. To assist with this, we now provide an overview of all notation we will use henceforth in descriptions of the Schnyder peeling process.

Let $T$ be a triangulation with boundary $B = (b_k, k \in \Z)$\glossarylabel{gl:initialboundary}, where we identify $b_x$ and $b_y$ if $x \cong y \mod m$ whenever $T$ has $m$ boundary vertices. Let $P_x$\glossarylabel{gl:px} be the Schnyder peeling process on $T$ initiated with root edge $(b_x,b_{x+1})$. We write $\boundary{x}{0} \coloneqq B$ and let $\inroote{x} = \roote{x}{0} \coloneqq (b_x, b_{x+1})$ denote the initial root edge of $P_x$. For each step $i$ of $P_x$:
\begin{itemize}
    \item Let $\roote{x}{i}$\glossarylabel{gl:roote} and $\boundary{x}{i}$\glossarylabel{gl:boundary} be the updated root edge and boundary after step $i$.
    \item Let $\peelv{x}{i}$\glossarylabel{gl:peelv} and $\chord{x}{i}$\glossarylabel{gl:chord} denote the peeling vertex and chord used during step $i$.
    \item Let the partial triangulation revealed by step $i$ be the part of $T$ that is enclosed between $\boundary{x}{i-1}$ and $\boundary{x}{i}$, including the parts on the boundaries.
    \item Let $\area{x}{i}$\glossarylabel{gl:area} be the revealed area up to and including step $i$, that is, the union of the partial triangulations revealed over all steps $j$ with $j \leq i$. Let $\evarea{x} \coloneqq \bigcup_i \area{x}{i}$\glossarylabel{gl:evarea}. Also, write $\intr \area{x}{i} \coloneqq \area{x}{i} \sm \boundary{x}{i}$ for the interior of the revealed area, and let $\intr \evarea{x} \coloneqq \bigcup_i \intr \area{x}{i}$.

    The \defn{upper boundary} of the revealed area $\area{x}{i}$ is $\boundary{x}{i} \sm (B \cup \roote{x}{i})$ and its \defn{lower boundary} is $B \cap \area{x}{i}$. Moreover, if the lower boundary is not the entire initial boundary and its leftmost vertex is not incident to $\roote{x}{i}$, then the leftmost vertex on the lower boundary is called the \defn{corner vertex} of $\area{x}{i}$.
\end{itemize}

So in our notation, step $i$ of $P_x$ performs a peeling step based on boundary $\boundary{x}{i-1}$ and root edge $\roote{x}{i-1}$ by choosing a chord $\chord{x}{i}$ incident with peeling vertex $\peelv{x}{i}$. This updates the boundary to some new boundary $\boundary{x}{i}$, and $P_x$ chooses a new root edge $\roote{x}{i}$ on $\boundary{x}{i}$. The area revealed after this step is $\area{x}{i}$.

\subsubsection{Structure after a finite number of steps}\label{sssec:structurepeelprocfinitesteps}

We now deduce some properties about the structure of the colouring and the directions assigned by the Schnyder peeling process to the edges of the revealed area $\area{x}{i}$ after some finite number of steps. We begin by investigating the colouring assigned to vertices on the upper boundary and to the corner vertex of $\area{x}{i}$.

\begin{lemma}\label{obs:finitestepsupperboundary}
    After step $i$ of $P_x$, for each vertex $v$ on the upper boundary of the revealed area $\area{x}{i}$,
    \begin{itemize}
        \item There is exactly one outgoing edge at $v$ which is coloured blue, and
        \item In anticlockwise order starting from the upper boundary, edges incident with $v$ consist of the outgoing blue edge and incoming yellow edges.
    \end{itemize}
    Moreover, if $\area{x}{i}$ has a corner vertex, then every edge incident with the corner vertex of $\area{x}{i}$ is an incoming yellow edge.
\end{lemma}

\begin{proof}
    We first verify the claim for the vertices on the upper boundary of $\area{x}{i}$. Suppose that the claim holds at step $i-1$, and suppose that we take a left step at step $i$. This step chooses the chord $\chord{x}{i}$, and reveals the vertices $u_1,\dots,u_k$, which then lie on the upper boundary of $\area{x}{i}$. Each of the vertices $u_1,\dots,u_k$ satisfies the claim, as each of these vertices only has one outgoing blue edge that has been explored. If the head $h$ of $\chord{x}{i}$ lies on the initial boundary, we have nothing further to check. Otherwise, $h$ lies on the upper boundary of $\area{x}{i-1}$, so $h$ already has one outgoing blue edge, and possibly some incoming yellow edges in anticlockwise order between the outgoing blue edge and the upper boundary of $\area{x}{i-1}$. By \cref{thm:peeling_process_maximal_wood} the newly explored region under the chord $\chord{x}{i}$ satisfies the Schnyder root condition at $h$, and so all further edges incident with $h$ in anticlockwise order between the outgoing blue edge and the upper boundary of $\area{x}{i}$ are also yellow and directed towards $h$. Hence, $h$ also still satisfies the claim.

    Otherwise, if step $i$ is a right step, then the upper boundary of $\area{x}{i}$ is a strict subset of the upper boundary of $\area{x}{i-1}$, and so the claim is still satisfied. Hence, inductively, the claim always holds for vertices on the upper boundary.
    
    Finally, suppose that $\area{x}{i}$ has a corner vertex $v$. Observe that the only explored edges incident with $v$ are any chords selected by the peeling process that are incident with $v$, along with coloured edges within any of the finite regions explored after those chords were selected. These chords are always directed towards $v$ and coloured yellow by definition of the Schnyder peeling process. Furthermore, since $v$ is the head of the chord in each of the finite regions, it is the head of the root edge for each of the finite explorations, and so it follows from \cref{thm:peeling_process_maximal_wood} that all of the remaining explored edges at $v$ are also directed towards $v$ and coloured yellow. Hence, the claim also holds for the corner vertex.
\end{proof}

Next, we investigate the structure of the monochromatic subgraphs of $\area{x}{i}$. We begin with a description of the edges added to each monochromatic subgraph at a left or right step of $P_x$. Let $t_\ell$ and $h_\ell$ denote the tail and head of the chord used by the $\ell$-th right step of $P_x$ (taking $t_0$ and $h_0$ to be, respectively, the tail and head of the initial root edge, so $t_0 = b_x$ and $h_0 = b_{x+1}$). Suppose that $P_x$ has performed $j$ right steps before step $i$, so $t_j$ is the tail of the root edge $\roote{x}{i-1}$ and is on the boundary $\boundary{x}{i-1}$ one vertex right of $\peelv{x}{i}$. See \cref{fig:example_left_right_steps_1} for an example of such a scenario.

In every step, most of the newly revealed edges are coloured according to the peeling process applied to a finite region of the triangulation. Since this process constructs a Schnyder wood of that region, we know by \cref{thm:schnyder_subgraphs_on_finite_triangulation} that these finite regions will always be composed of: a directed red tree rooted at the tail of the region's root edge, a directed yellow tree rooted at the head of the region's root edge, and a forest of blue trees, each rooted at a vertex on the region's boundary (other than the head or tail of the region's root edge). We now describe the overall structure of the edges added to the revealed area at a given step $i$ of the process. It is straightforward to verify, using the definition of the Schnyder peeling process, that this description is valid.

If step $i$ is a left step, then the finite region explored during step $i$ is rooted at $\chord{x}{i}$, and after step $i$, the updated boundary between $t_j$ and the head of $\chord{x}{i}$ is now the path segment composed of the neighbours $u_1,\dots,u_k$ of $\peelv{x}{i}$. The overall structure added by this left step is a single red tree rooted at $t_j$, a single yellow tree rooted at the head of $\chord{x}{i}$, and a forest of blue trees rooted at vertices on the previous boundary $\boundary{x}{i-1}$ strictly between $t_j$ and the head of $\chord{x}{i}$. Note in particular that for each $1 \leq \ell \leq k$, the edge $(u_\ell, \peelv{x}{i})$ is always added to the blue subgraph.

If step $i$ is a right step, and there have been $j$ right steps performed before step $i$, then the tail of the chosen chord $\chord{x}{i}$ is $t_{j+1}$, and its head is $h_{j+1}$. In this case, we reveal the neighbours $u_1,\dots,u_k$ of $\peelv{x}{i}$ below the chord, and the explored finite region is rooted at $(t_j, h_j)$. The overall structure added by this right step consists of a red tree rooted at $t_j$, a yellow tree rooted at $h_j$, and a forest of blue trees rooted at vertices on the previous boundary $\boundary{x}{i-1}$ between $t_{j+1}$ and $h_{j+1}$ (excluding $t_j$ and $h_j$). Note in particular that the edge $(t_{j+1}, t_j)$ is always added to the red subgraph. Moreover, as a consequence of the structure of the finite region, some directed path from $h_{j+1}$ to $h_j$ is necessarily added to the yellow tree. Since the Schnyder wood assigned to the explored finite region during step $i$ is maximal, we know from \cref{lem:finite_yellow_path} that the yellow path from $h_{j+1}$ to $h_j$ consists entirely of vertices on the previous boundary $\boundary{x}{i-1}$ between $h_j$ and $h_{j+1}$.
\begin{figure}[h]
    \centering
    \begin{tikzpicture}[line width=1.2]

    \foreach \x in {2,...,8,11,12}{
        \node[mycircle, opacity=0.5] (bn\x) at (\x,0) {};
    }
    \foreach \x in {15,16,17}{
        \node[mycircle] (bn\x) at (\x,0) {};
    }
    \node[mycircle, label={[red]below:{$t_1$}}, color=red, opacity=0.5] (bn9) at (9,0) {};
    \node[mycircle, label={[xshift=-0.2cm, red]below:{$t_0$}}, color=red, opacity=0.5] (bn10) at (10,0) {};
    \node[mycircle, opacity=0, label={[xshift=0.2cm]below:{$h_0$}}] (bn11) at (11,0) {};
    \node[mycircle, opacity=0.5, label=below:{$h_1$}] (bn13) at (13,0) {};
    \node[mycircle, label=below:{$h_2$}] (bn14) at (14,0) {};
    \node[mycircle, label=below:{$h_3$}] (bn18) at (18,0) {};
    \node[color=myyellow] (lbl0) at (10.5, -0.3) {\small{$c_x(0)$}};
    
    \foreach \x in {2,...,7,11,12,13}{
        \pgfmathtruncatemacro{\next}{\x + 1};
        \draw (bn\x) edge[opacity=0.5] (bn\next);
    }
    \foreach \x in {14,...,17}{
        \pgfmathtruncatemacro{\next}{\x + 1};
        \draw (bn\x) -- (bn\next);
    }
    
    \draw (bn8) edge[->-,redge, opacity=0.5] (bn9)
          (bn9) edge[->-,redge, opacity=0.5] (bn10)
          (bn10) edge[->-,yedge, opacity=0.5] (bn11);

    \draw (bn9) edge[->-, yedge, opacity=0.5, bend left=55, looseness=1.55] (bn13);
    \node[color=myyellow] (lbl1) at (11.1, 1.8) {\small{$c_x(1)$}};
    \node[mycircle, opacity=0.5] (u1r1) at (10.5,0.6) {};
    \node[mycircle, opacity=0.5] (u2r1) at (11,1.2) {};
    \draw (bn10) edge[opacity=0.5] (u1r1)
    (u1r1) edge[opacity=0.5] (u2r1)
    (u2r1) edge[opacity=0.5] (bn13);
    
    \draw (u1r1) edge[->-, bedge, opacity=0.5] (bn9);
    \draw (u2r1) edge[->-, bedge, opacity=0.5] (bn9);
    \draw (bn13) edge[->-, yedge, bend right=25, opacity=0.5] (bn11);

    \draw (bn8) edge[->-, yedge, opacity=0.5, bend right=30] (bn2);
    \node[color=myyellow] (lbl2) at (5, 0.6) {\small{$c_x(2)$}};
    \node[mycircle,color=red] (u1l1) at (8,1.3) {};
    \node[mycircle,color=red] (u2l1) at (7.2,1.5) {};
    \node[mycircle, opacity=0.5] (u3l1) at (6.4,1.7) {};
    \node[mycircle, opacity=0.5] (u4l1) at (5.4,1.8) {};
    \node[mycircle, opacity=0.5] (u5l1) at (4.3,1.6) {};
    
    \draw (u2l1) edge[opacity=0.5] (u3l1)
    (u3l1) edge[opacity=0.5] (u4l1)
    (u4l1) edge[opacity=0.5] (u5l1)
    (u5l1) edge[opacity=0.5, bend right=10] (bn2);
    
    \draw (u1l1) edge[->-, bedge, opacity=0.5, bend left=10] (bn8);
    \draw (u2l1) edge[->-, bedge, opacity=0.5, bend left=10] (bn8);
    \draw (u3l1) edge[->-, bedge, opacity=0.5, bend left=10] (bn8);
    \draw (u4l1) edge[->-, bedge, opacity=0.5, bend left=10] (bn8);
    \draw (u5l1) edge[->-, bedge, opacity=0.5, bend left=10] (bn8);
    
    \draw (u1l1) edge[->-, redge, opacity=0.5] (bn9);
    \draw (u1l1) edge[->-, yedge, opacity=0.5, bend left=60, looseness=1.2] (bn14);
    \node[color=myyellow] (lbl3) at (13, 2.2) {\small{$c_x(3)$}};
    \node[mycircle, opacity=0.5] (u1r2) at (8.8,0.8) {};
    \node[mycircle, opacity=0.5] (u2r2) at (9.5,1.8) {};
    \draw (bn9) edge[opacity=0.5] (u1r2)
    (u1r2) edge[opacity=0.5] (u2r2)
    (u2r2) edge[opacity=0.5, bend left=50, looseness=1] (bn14);
    
    \draw (u1r2) edge[->-, bedge, opacity=0.5] (u1l1);
    \draw (u2r2) edge[->-, bedge, opacity=0.5] (u1l1);
    \draw (bn14) edge[->-, yedge, bend right=60, opacity=0.5] (bn13);


    \draw (u2l1) edge[->-, redge, line width=1.7] (u1l1);
    \draw (u2l1) edge[->-, yedge, bend left=80, looseness=1.5, line width=1.7] (bn18);
    \node[mycircle] (u1r3) at (8,2.1) {};
    \node[mycircle] (u2r3) at (10,4.3) {};
    \draw (u1l1) to (u1r3) to (u2r3) edge[bend left=55] (bn18);
    \draw (u1r3) edge[->-, bedge,line width=1.7] (u2l1);
    \draw (u2r3) edge[->-, bedge, line width=1.7, bend right=10, bend right=13] (u2l1);
    \draw (bn18) edge[->-, yedge, bend right=75, line width=1.7] (bn16);
    \draw (bn16) edge[->-, yedge, bend right=75, line width=1.7] (bn14);

    \node[color=red] (lbl) at (7, 1.8) {$t_3$};
    \node[color=red] (lbl3) at (7.8, 1.1) {$t_2$};
    \node[color=myyellow] (lbl2) at (13.5, 5.85) {\small{$c_x(4)$}};
    
    \draw (10,3) edge[redge,dotted,line width=1.7,bend left=10] (11.5,3.2)
                 (11.5,3.2) edge[redge,dotted,line width=1.7,bend left=5] (12.5,3.3)
                 (11.5,3.2) edge[redge,dotted,line width=1.7,bend right=5] (12.5,2.9)
          (10,3) edge[redge,dotted,line width=1.7,bend left=10] (12,2.5)
          (10,3) edge[->-, redge, dotted, bend right=20, line width=1.7] (u1l1);

    \draw (11, 3.8) edge[->-, bedge, dotted, bend right=15, line width=1.7] (u1r3);
    
    \draw (12, 4.2) edge[->-, bedge, dotted, bend right=15, line width=1.7] (u2r3);
    \draw (14.5, 3.5) edge[-, bedge, dotted, bend right=10, line width=1.7] (12, 4.2);
    \draw (13.5, 3.2) edge[-, bedge, dotted, bend right=10, line width=1.7] (12, 4.2);
    
    \draw (15.5, 2) edge[->-, bedge, dotted, bend left=25, line width=1.7] (bn18);

    \draw (14.5, 1.8) edge[->-, yedge, dotted, bend right=10, line width=1.7] (bn14);
    \draw (14, 2.5) edge[-, yedge, dotted, bend right=10, line width=1.7] (14.5, 1.8);
    \draw (15, 2.7) edge[-, yedge, dotted, bend right=10, line width=1.7] (14.5, 1.8);

    \coordinate (bluetree3) at (15, 0.5); 
    \coordinate (bluetree4) at (15.5, 1.5); 
    \draw (bluetree3) edge[->-, bedge, dotted, line width=1.7] (bn15);
    \draw (bluetree4) edge[->-, bedge, bend left=15, dotted, line width=1.7] (bn16);
\end{tikzpicture}
    \begin{tikzpicture}[line width=1.2]

    \node[mycircle] (bn2) at (2,0) {};
    \foreach \x in {3,...,8,12,17,18}{
        \node[mycircle, opacity=0.5] (bn\x) at (\x,0) {};
    }
    \foreach \x in {15,16}{
        \node[mycircle, opacity=0.5] (bn\x) at (\x,0) {};
    }
    \node[mycircle, label={[red]below:{$t_1$}}, color=red, opacity=0.5] (bn9) at (9,0) {};
    \node[mycircle, label={[xshift=-0.2cm, red]below:{$t_0$}}, color=red, opacity=0.5] (bn10) at (10,0) {};
    \node[mycircle, opacity=0.5, label={[xshift=0.2cm]below:{$h_0$}}] (bn11) at (11,0) {};
    \node[mycircle, opacity=0.5, label=below:{$h_1$}] (bn13) at (13,0) {};
    \node[mycircle, opacity=0.5, label=below:{$h_2$}] (bn14) at (14,0) {};
    \node[color=myyellow] (lbl0) at (10.5, -0.3) {\small{$c_x(0)$}};
    
    \foreach \x in {2,...,7,11,12,13}{
        \pgfmathtruncatemacro{\next}{\x + 1};
        \draw (bn\x) edge[opacity=0.5] (bn\next);
    }
    \foreach \x in {14,...,17}{
        \pgfmathtruncatemacro{\next}{\x + 1};
        \draw (bn\x) -- (bn\next);
    }
    
    \draw (bn8) edge[->-,redge, opacity=0.5] (bn9)
          (bn9) edge[->-,redge, opacity=0.5] (bn10)
          (bn10) edge[->-,yedge, opacity=0.5] (bn11);

    \draw (bn9) edge[->-, yedge, opacity=0.5, bend left=55, looseness=1.55] (bn13);
    \node[color=myyellow] (lbl1) at (11.1, 1.8) {\small{$c_x(1)$}};
    \node[mycircle, opacity=0.5] (u1r1) at (10.5,0.6) {};
    \node[mycircle, opacity=0.5] (u2r1) at (11,1.2) {};
    \draw (bn10) edge[opacity=0.5] (u1r1)
    (u1r1) edge[opacity=0.5] (u2r1)
    (u2r1) edge[opacity=0.5] (bn13);
    
    \draw (u1r1) edge[->-, bedge, opacity=0.5] (bn9);
    \draw (u2r1) edge[->-, bedge, opacity=0.5] (bn9);
    \draw (bn13) edge[->-, yedge, bend right=25, opacity=0.5] (bn11);

    \draw (bn8) edge[->-, yedge, opacity=0.5, bend right=30] (bn2);
    \node[color=myyellow] (lbl2) at (5, 0.6) {\small{$c_x(2)$}};
    \node[mycircle,color=red] (u1l1) at (8,1.3) {};
    \node[mycircle,color=red] (u2l1) at (7.2,1.5) {};
    \node[mycircle] (u3l1) at (6.4,1.7) {};
    \node[mycircle] (u4l1) at (5.4,1.8) {};
    \node[mycircle] (u5l1) at (4.3,1.6) {};

    \draw (u2l1) edge[->-, redge, line width=1.7] (u1l1);
    \draw (u2l1) edge[] (u3l1)
    (u3l1) edge[] (u4l1)
    (u4l1) edge[] (u5l1)
    (u5l1) edge[bend right=10] (bn2);
    
    \draw (u1l1) edge[->-, bedge, opacity=0.5, bend left=10] (bn8);
    \draw (u2l1) edge[->-, bedge, opacity=0.5, bend left=10] (bn8);
    \draw (u3l1) edge[->-, bedge, opacity=0.5, bend left=10] (bn8);
    \draw (u4l1) edge[->-, bedge, opacity=0.5, bend left=10] (bn8);
    \draw (u5l1) edge[->-, bedge, opacity=0.5, bend left=10] (bn8);
    
    \draw (u1l1) edge[->-, redge, opacity=0.5] (bn9);
    \draw (u1l1) edge[->-, yedge, opacity=0.5, bend left=60, looseness=1.2] (bn14);
    \node[color=myyellow] (lbl3) at (13, 2.2) {\small{$c_x(3)$}};
    \node[mycircle, opacity=0.5] (u1r2) at (8.8,0.8) {};
    \node[mycircle, opacity=0.5] (u2r2) at (9.5,1.8) {};
    \draw (bn9) edge[opacity=0.5] (u1r2)
    (u1r2) edge[opacity=0.5] (u2r2)
    (u2r2) edge[opacity=0.5, bend left=50, looseness=1] (bn14);
    
    \draw (u1r2) edge[->-, bedge, opacity=0.5] (u1l1);
    \draw (u2r2) edge[->-, bedge, opacity=0.5] (u1l1);
    \draw (bn14) edge[->-, yedge, bend right=60, opacity=0.5] (bn13);


    \draw (u2l1) edge[->-, yedge, bend right=70, looseness=1.7, line width=1.6] (bn2);
    \node[color=myyellow] (lbl4) at (5.2, 3.5) {\small{$c_x(4)$}};
    \node[mycircle] (u1l2) at (7.5,2.8) {};
    \node[mycircle] (u2l2) at (6,3.9) {};

    \draw (u1l1) to (u1l2) to (u2l2) edge[bend right=50, looseness=1.2] (bn2);
    \draw (u1l2) edge[->-, bedge,line width=1.7] (u2l1);
    \draw (u2l2) edge[->-, bedge,line width=1.7] (u2l1);

    \node[color=red] (lbl3) at (7.8, 1.1) {$t_2$};

    \draw (3, 1.8) edge[->-, yedge, dotted, bend right=10, line width=1.7] (bn2);
    \draw (3.7, 2.8) edge[-, yedge, dotted, bend right=10, line width=1.7] (3, 1.8);
    \draw (3.5, 2) edge[-, yedge, dotted, bend right=10, line width=1.7] (3, 1.8);
    
    \draw (5.6, 2.6) edge[->-, redge, dotted, bend left=10, line width=1.7] (u2l1);
    \draw (4.2, 2.9) edge[->-, redge, dotted, bend left=10, line width=1.7] (5.6, 2.6);
    \draw (4.6, 2.4) edge[->-, redge, dotted, bend left=10, line width=1.7] (5.6, 2.6);
    
    \draw (3.8, 2.2) edge[->-, bedge, dotted, line width=1.7] (u5l1);
    \draw (4.5, 2.2) edge[->-, bedge, dotted, line width=1.7] (u5l1);
    \draw (5, 2.2) edge[->-, bedge, dotted, line width=1.7] (u4l1);
    \draw (6, 2.1) edge[->-, bedge, dotted, line width=1.7] (u3l1);
    
\end{tikzpicture}
    \caption{The various monochromatic components added to the revealed Schnyder wood after $P_x$ takes either a right (top) or left (bottom) step.}
    \label{fig:example_left_right_steps_1}
\end{figure}
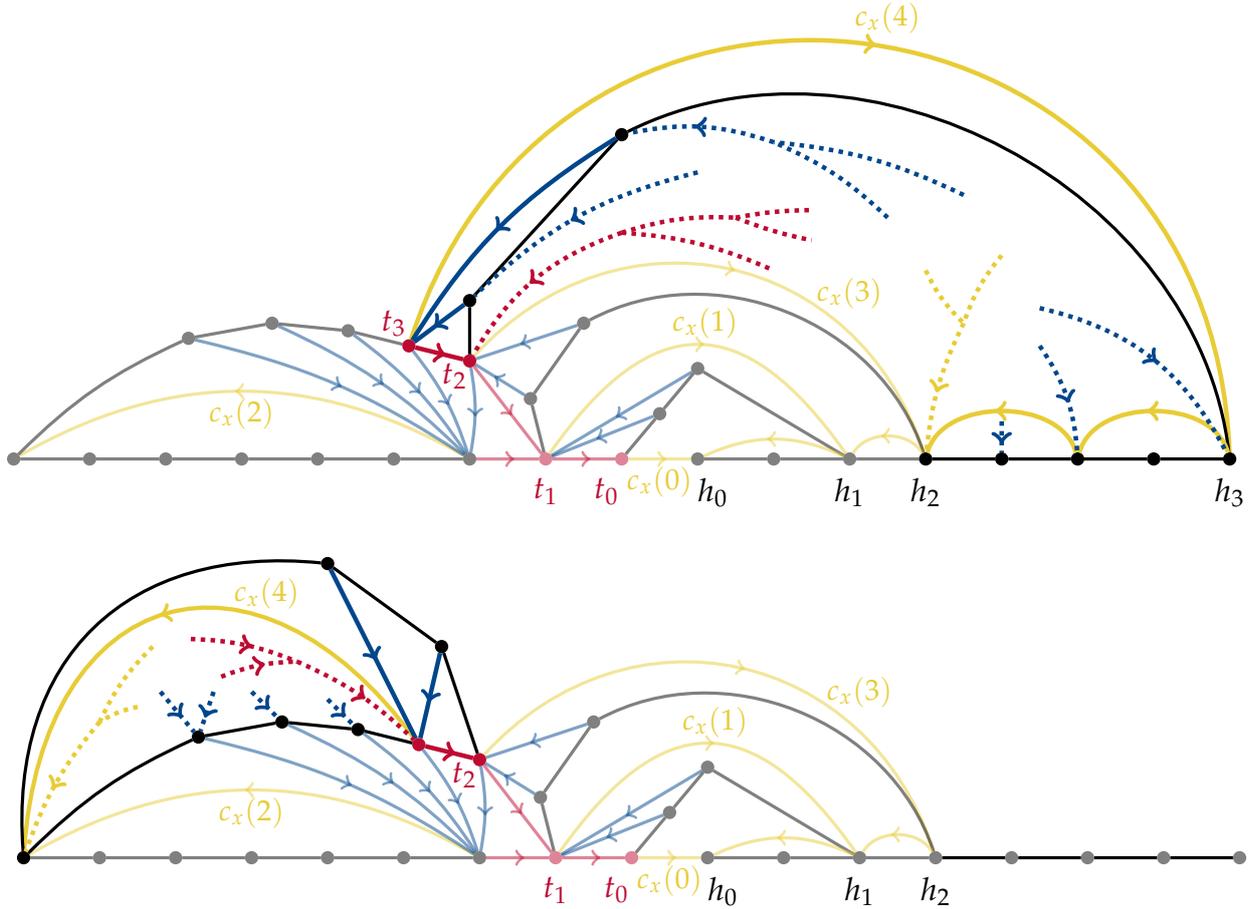

\begin{figure}[h!]
    \centering
    \begin{tikzpicture}[line width=1.2]
    \begin{scope}[line width=1.7]
        \draw (265:6) arc (265:-85:6);
        \draw[myyellow, ->-] (265:6) arc (265:275:6);
        \node[text=myyellow] at (275:6.4) {$b_{x+1}$};
        \node[text=red] at (265:6.4) {$b_{x}$};
        \node[text=blue] at (250:6.5) {$\peelv{x}{i_1}$};
        
        \draw (170:3.5) arc (170:-60:3.5);
        \draw (240:6) edge[out=60, in=210] (-60:3.5);
        \draw (170:3.5) edge[out=260, in=220, looseness=1.3] (-40:1.7);
        \draw (-40:1.7) edge[out=45, in=200] (-10:2.5);
        \draw (-10:2.5) edge[yedge, ->-, out=20, in=180] (-5:3.5);
    
        \path[every node/.style={font=\sffamily\small}]
            (-5:3.5) edge[yedge, ->-, out=220, in=80] (-50:3.5)
            (-50:3.5) edge[yedge, ->-, out=190, in=40] (-115:4.6)
            (-115:4.6) edge[yedge, ->-, out=160, in=20] (220:6)
            (220:6) edge[yedge, ->-, bend right=15] (180:6)
            (180:6) edge[yedge, ->-, bend right=15] (140:6)
            (140:6) edge[yedge, ->-, bend right=15] (100:6)
            (100:6) edge[yedge, ->-, bend right=15] (60:6)
            (60:6) edge[yedge, ->-, out=280, in=135] (10:6)
            (10:6) edge[yedge, ->-, bend right=15] (-20:6)
            (-20:6) edge[yedge, ->-, bend right=15] (-55:6)
            (-55:6) edge[yedge, ->-, bend right=15] (-85:6);

        \draw[red, ->-] (190:4.75) arc (190:-70:4.75);
        \path[every node/.style={font=\sffamily\small}]
            (-10:2.5) edge[redge, ->-, out=270, in=280, looseness=1.07] (190:4.75);
        \path[every node/.style={font=\sffamily\small}]
            (-70:4.75) edge[redge, ->-, out=200, in=80] (265:6);

        \draw[blue, ->-] (190:4.13) arc (190:-70:4.13);
        \path[every node/.style={font=\sffamily\small}]
            (-80:1.73) edge[bedge, ->-, out=245, in=280, looseness=0.90] (190:4.13);
        \path[every node/.style={font=\sffamily\small}]
            (-70:4.13) edge[bedge, ->-, out=200, in=80, looseness=0.90] (250:6);
    \end{scope}
    
    \path[every node/.style={font=\sffamily\small}]
        (260:6) edge[yedge, ->-, bend right=15] (250:6)
        (250:6) edge[yedge, ->-, bend right=15] (240:6)
        (240:6) edge[yedge, ->-, bend right=15] (220:6);
    \path[every node/.style={font=\sffamily\small}]
        (255:4.3) edge[yedge, ->-, bend left=10] (-115:4.6); 
    \path[every node/.style={font=\sffamily\small}]
        (-35:4.5) edge[yedge, ->-, out=150, in=-50] (-50:3.5);
    \path[every node/.style={font=\sffamily\small}]
        (-20:3.9) edge[yedge, ->-, bend right=10] (-25:3.5)
        (-25:3.5) edge[yedge, ->-, bend right=10] (-50:3.5);
    \path[every node/.style={font=\sffamily\small}]
        (2:4.4) edge[yedge, ->-, bend right=10] (-5:3.5);
    \path[every node/.style={font=\sffamily\small}]
        (45:4.4) edge[yedge, ->-, bend right=10] (40:3.5)
        (50:3.8) edge[yedge, ->-, bend left=10] (40:3.5);

    \path[every node/.style={font=\sffamily\small}]
        (220:4.6) edge[bedge, ->-, bend left=10] (-115:4.6)
        (-115:4.6) edge[bedge, ->-, bend right=10] (-103:4.7);
    \path[every node/.style={font=\sffamily\small}]
        (270:3.3) edge[bedge, ->-, bend right=10] (260:3.8)
        (260:3.8) edge[bedge, ->-, bend left=10] (-103:4.7);
    \path[every node/.style={font=\sffamily\small}]
        (-50:2.9) edge[bedge, ->-, bend right=10] (-50:3.5)
        (-50:3.5) edge[bedge, ->-, bend left=10] (-55:4.13);
    \path[every node/.style={font=\sffamily\small}]
        (-5:3.5) edge[bedge, ->-] (-5:4.13);
    \path[every node/.style={font=\sffamily\small}]
        (40:3.5) edge[bedge, ->-] (40:4.13);
    \path[every node/.style={font=\sffamily\small}]
        (40:5.2) edge[bedge, ->-,out=60,in=-140] (35:6)
        (30:5.5) edge[bedge, ->-,bend right=10] (35:6);
    \path[every node/.style={font=\sffamily\small}]
        (-50:5.1) edge[bedge, ->-,out=-110,in=90] (-55:6)
        (-60:5.3) edge[bedge, ->-,bend left=10] (-55:6);
\end{tikzpicture}
    \caption{Structure of the monochromatic subgraphs in $\area{x}{i}$ after some step $i$ at which the head of the root edge $\roote{x}{i}$ is no longer on the initial boundary.}
    \label{fig:loopy_boy}
\end{figure}
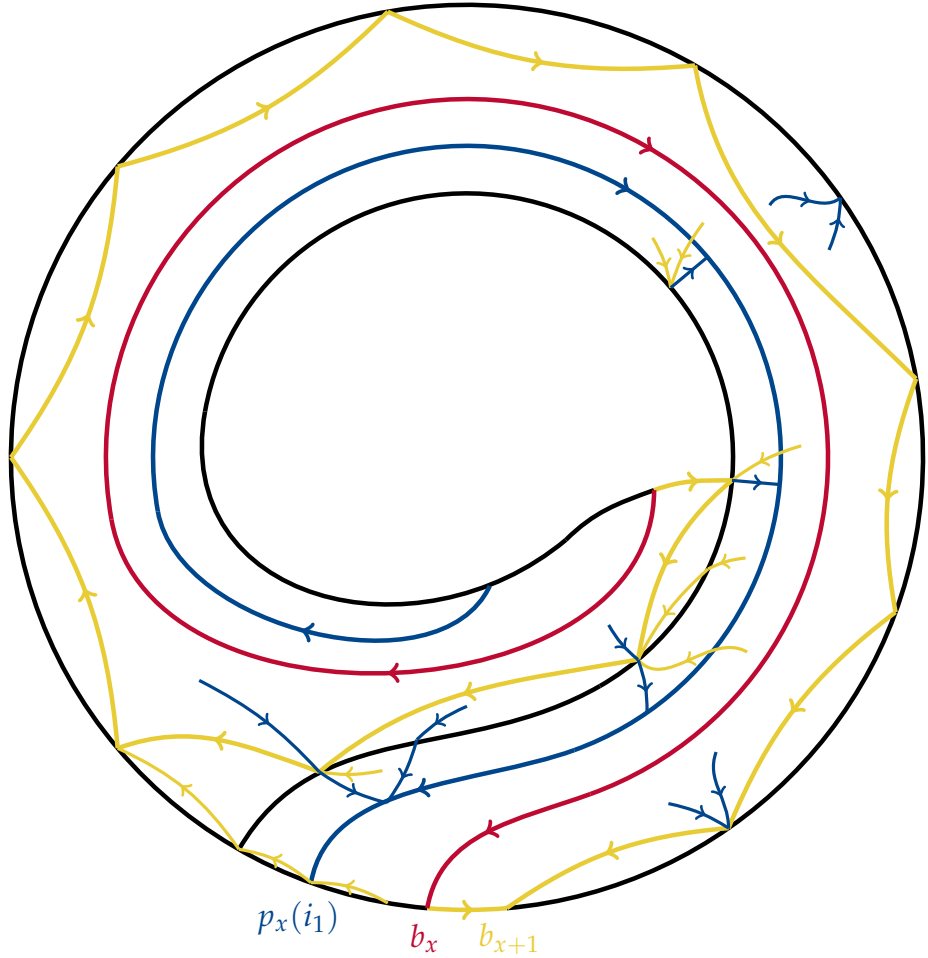

We now use the above facts to determine the structure of the colouring assigned to the revealed area after a finite number of steps, checking each monochromatic subgraph in turn. See \cref{fig:loopy_boy} for an demonstration of the structure of the monochromatic subgraphs after the head of the root edge no longer lies on the initial boundary.

\begin{lemma}\label{lem:schnyderareayellowforest}
    After step $i$ of $P_x$, the yellow subgraph of the revealed area $\area{x}{i}$ consists of
    \begin{itemize}
        \item A tree rooted at $b_{x+1}$, containing the current root edge $\roote{x}{i}$ and a path whose vertices include the heads $h_\ell$ of every chord chosen by a right step of $P_x$, and not containing any vertices on the upper boundary or the corner vertex of $\area{x}{i}$, and
        \item A forest of trees rooted at vertices on the upper boundary or at the corner vertex of $\area{x}{i}$.
    \end{itemize}
    Moreover, when step $i$ is the $j$-th right step, the yellow path from $h_{j+1}$ to $h_j$ consists entirely of vertices on $\boundary{x}{i-1}$ between $h_j$ and $h_{j+1}$.
\end{lemma}

\begin{proof}
    Suppose that the statement holds after $i-1$ steps, at which point the process $P_x$ has taken $j$ right steps. We show that the statement holds after $i$ steps. 

    If step $i$ is a right step, the component added to the yellow subgraph is a finite tree rooted at the head $h_j$ of the previous root edge. By the induction hypothesis, the yellow subgraph after $j$ right steps includes a directed path $h_j,\dots,h_{j-1},\dots,h_1,\dots,h_0 = b_{x+1}$, and so the yellow edges added at step $i$ are added to the tree rooted at $b_{x+1}$. In particular, $\roote{x}{i}$ is added to the tree rooted at $b_{x+1}$. Any edges of the yellow trees rooted at any vertex between $t_{j+1}$ and $h_{j+1}$ on $\boundary{x}{i-1}$ will also be added to the same tree. Furthermore, it follows from the structure of the right step that apart from vertices between $t_{j+1}$ and $h_{j+1}$, the vertices on the upper boundary and the corner vertex of $\area{x}{i-1}$ will form the upper boundary and the corner vertex of $\area{x}{i}$. So, the induction hypothesis implies that, after step $i$, the yellow trees not rooted at $b_{x+1}$ are rooted at vertices on the upper boundary or at the corner vertex of $\area{x}{i}$. Finally, as remarked in the paragraph preceding \cref{lem:schnyderareayellowforest}, the yellow path from $h_{j+1}$ to $h_j$ consists entirely of vertices on $\boundary{x}{i-1}$ between $h_j$ and $h_{j+1}$. This shows that the claim holds after step $i$.

    Otherwise, if step $i$ is a left step, the component added to the yellow subgraph is a finite tree rooted at the head of $\chord{x}{i}$, which is either a vertex on the upper boundary or the corner vertex of $\area{x}{i}$. Any finite yellow tree rooted at a vertex covered by the left step will also be added to this new finite tree rooted at the head of $\chord{x}{i}$. Note that the tree rooted at $b_{x+1}$ is unaffected and so satisfies the properties from the induction hypothesis also after step $i$. Again, this shows that the claim holds after step $i$.
\end{proof}

\begin{lemma}\label{lem:schnyderarearedforest}
    After step $i$ of $P_x$, the red subgraph of the revealed area $\area{x}{i}$ is a tree rooted at $b_x$, containing a path whose vertices are the tails $t_j$ of every chord chosen by a right step of $P_x$.
\end{lemma}

\begin{proof}
    In both left and right steps, the component added to the red subgraph is a finite tree rooted at the tail $t_j$ of the current root edge. Furthermore, the edge $(t_{j+1}, t_j)$ is added to the red subgraph at the $(j+1)$-th right step, and so the red subgraph contains the path $t_j,t_{j-1},\dots,t_1,t_0 = b_x$.
\end{proof}

\begin{lemma}\label{lem:schnyderareablueforest}
    For any step $i$, let $i_0 = 0$ and inductively define $i_\ell$ as the last step from $\{i_{\ell-1}~+~1, \dots, i\}$ for which $\peelv{x}{i_\ell}$ is on the boundary $\boundary{x}{i_{\ell-1}}$. After step $i$ of $P_x$, the blue subgraph of the revealed area $\area{x}{i}$ consists of
    \begin{itemize}
        \item A tree rooted at $\peelv{x}{i_1}$, containing a path whose vertices are the peeling vertices $\peelv{x}{i_\ell}$ with $\ell \ge 1$, and
        \item A forest of trees rooted at vertices on the lower boundary of $\area{x}{i}$.
    \end{itemize}
    Moreover, every vertex on the upper boundary of $\area{x}{i}$ has an outgoing blue edge to one of the peeling vertices $\peelv{x}{i_\ell}$ with $\ell \ge 1$.
\end{lemma}

\begin{proof}
    We first prove inductively that after step $i$ of $P_x$, the blue subgraph consists of a forest of trees rooted at vertices on the initial boundary. Indeed, in both left and right steps, the component added to the blue subgraph is a forest of trees rooted at vertices of the previous boundary. A vertex of the previous boundary is either on the initial boundary or, by the induction hypothesis, is in a blue tree rooted at a vertex of the initial boundary. In both cases, every blue edge added during step $i$ will be added to a blue tree rooted at a vertex on the initial boundary, as claimed.

    Since $\peelv{x}{i_1}$ is on the initial boundary, if $i_1$ is a right step, the upper boundary of $\area{x}{i_1}$ contains no vertices, and if $i_1$ is a left step, the upper boundary of $\area{x}{i_1}$ consists only of the neighbours $u_1, \dots, u_k$ of $\peelv{x}{i_1}$ revealed during step $i_1$. In particular, after step $i_1$, every vertex on the upper boundary of $\area{x}{i_1}$ has an outgoing blue edge to $\peelv{x}{i_1}$. Since $\peelv{x}{i_2}$ is on $\boundary{x}{i_1}$ but not on the initial boundary and not incident to $\roote{x}{i_1}$, it follows that $\peelv{x}{i_2}$ must be on the upper boundary of $\area{x}{i_1}$, and so there is a blue edge from $\peelv{x}{i_2}$ to $\peelv{x}{i_1}$.

    We now apply the same arguments for step $i_2$ instead of $i_1$ and with the updated boundary $\boundary{x}{i_1}$ instead of the initial boundary. This shows that after step $i_2$, every vertex that is on the upper boundary of $\area{x}{i_2}$ and not on $\boundary{x}{i_1}$ has an outgoing blue edge to $\peelv{x}{i_2}$, and there is a blue edge from $\peelv{v}{i_3}$ to $\peelv{x}{i_2}$. By iterating this argument, we obtain a blue path $\peelv{x}{i_\ell}, \peelv{x}{i_{\ell-1}}, \dots, \peelv{x}{i_1}$, and every vertex on the upper boundary of $\area{x}{i_\ell}$ has an outgoing blue edge to one of the vertices $\peelv{x}{i_1}, \dots, \peelv{x}{i_\ell}$. Since $i_\ell = i$ for some $\ell$, this proves the claim.
\end{proof}

\section{Infinite triangulations with finite boundary}\label{sec:infinite_triang_finite_bdry}

In this section, we study the behaviour of the Schnyder peeling process on infinite triangulations with finite boundary. We show that if $T$ is such a triangulation, then the Schnyder peeling process explores all of $T$, which we use to establish the existence of infinite Schnyder woods. We further show that the Schnyder wood constructed is maximal and therefore the unique maximal Schnyder wood of $T$. We also describe some of the structure of this Schnyder wood, and show that the Schnyder wood of the UIPT is a weak limit of Schnyder woods of random finite triangulations as their number of vertices tends to infinity. We begin by establishing the following theorem.

\begin{theorem}\label{thm:unique_maximal_schnyder_wood_finite_bdry}
    Every infinite triangulation with a finite boundary has a unique maximal Schnyder wood.
\end{theorem}

Showing that the vertices of $T$ explored by the Schnyder peeling process satisfy the conditions for Schnyder woods is straightforward -- the verification is the same as for the finite case. However, to show that the Schnyder peeling process constructs a Schnyder wood of $T$, we additionally must demonstrate that the Schnyder peeling process explores the entire triangulation $T$. The following lemma justifies that this is true, and that additionally, this infinite exploration constructs no finite anticlockwise cycles on $T$. Suppose $T$ is rooted at $(b_x,b_{x+1})$.

\begin{lemma}\label{lem:infinite_triangulation_noanticlockwise}
    The Schnyder peeling process $P_x$ constructs a Schnyder wood of $T$ without anticlockwise cycles.
\end{lemma} 

\begin{proof}
    By \cref{obs:infiniterightsteps}, $P_x$ will have covered all initial boundary vertices after a finite number of steps. Since the updated boundary after these steps is still finite, iterating this argument shows that the process will explore all vertices at any fixed distance from the initial boundary after a finite number of steps, and so the process will explore the entire triangulation $T$.

    Recall that the region of $T$ explored by any finite number of steps of the Schnyder peeling process is assigned the same colouring as if it were contained in a finite triangulation and were explored by the finite peeling process. Since $P_x$ explores any finite collection of edges of $T$ after a finite number of steps, \cref{thm:peeling_process_maximal_wood} therefore implies that the colouring of $T$ constructed by $P_x$ will satisfy the Schnyder conditions, and it will not contain any anticlockwise cycles.
\end{proof}

To establish that the constructed Schnyder wood is maximal, it remains to show that no infinite right-directed paths exist. Define a \defn{leftmost walk} in an orientation of $T$ to be a directed walk that at each vertex always takes the first outgoing edge clockwise from the edge at which it entered, if an outgoing edge exists, and of maximum length subject to this. Let $L(e)$\glossarylabel{gl:le} denote the leftmost walk starting from a directed edge $e$. We will use the following property of leftmost walks in the Schnyder wood of $T$ constructed by the Schnyder peeling process.

\begin{lemma} \label{lem:leftmostwalk_finite_bdry}
    Every leftmost walk in the Schnyder wood of $T$ constructed by $P_x$ terminates at the head of the root edge.
\end{lemma}

\begin{proof}
    Let $e$ be an edge of $T$, and consider the leftmost walk $L(e)$ in the Schnyder wood of $T$ constructed by $P_x$. We claim that $L(e)$ is a directed path. Indeed, it follows from \cref{lem:infinite_triangulation_noanticlockwise} that $L(e)$ contains no anticlockwise cycles. Suppose that $L(e)$ contains a clockwise cycle $C$. Since $L(e)$ always follows the leftmost outgoing edge at each vertex of $C$, at every vertex of $C$ that is an internal vertex of $T$, there are two edges directed to the interior of $C$. Similarly, at every vertex of $C$ that is a boundary vertex and not incident to the initial root edge, there is one edge directed to the interior of $C$. Since there are only two other vertices in $T$, it follows that all but at most two vertices of $C$ have at least one edge directed to the interior of $C$, contradicting \cref{lem:directed_edges_from_cycle}. Hence, $L(e)$ is acyclic, as claimed.

    Since $e$ is at some finite distance from the initial boundary of $T$, the edge $e$ will be explored by the Schnyder peeling process within a finite number of steps. Consider the revealed area of $T$ after $e$ and the entire initial boundary has been explored. Then, the boundary of the unexplored region consists of the upper boundary of the revealed area and the current root edge, which we call $r$. Note that by \cref{obs:finitestepsupperboundary}, the path $L(e)$ can only enter an upper boundary vertex via an incoming yellow edge, but then it leaves that vertex via the blue outgoing edge, which is part of the revealed area. Moreover, any unexplored edge incident to the head or tail of the $r$ will, in the final Schnyder wood, be coloured as an incoming red or yellow edge, and so the path $L(e)$ cannot leave via these vertices either. Thus, $L(e)$ can never leave the revealed area. It follows that $L(e)$ is finite and must therefore terminate at the unique vertex in $T$ with no outgoing edges, which is the head of the initial root edge.
\end{proof}

We can now prove that the constructed Schnyder wood is maximal.

\begin{lemma}\label{lem:maximal_wood_finite_bdry}
    The Schnyder wood constructed on $T$ by $P_x$ is maximal.
\end{lemma}

\begin{proof}
    By \cref{lem:infinite_triangulation_noanticlockwise}, it remains to show that the Schnyder wood constructed on $T$ by $P_x$ has no right-directed paths. Suppose for a contradiction that $T$ contains a right-directed path $R$, and let $e$ be any edge on $R$. By \cref{lem:leftmostwalk_finite_bdry}, the leftmost walk $L(e)$ ends at the head $h$ of the initial root edge (which cannot be on $R$ since $h$ has no outgoing edge), so there must be an edge $e'$ on $L(e)$ that is directed from a vertex $v$ on $R$ towards the side of $R$ containing $h$. Consider the first such edge. The previous edge on $L(e)$ lies either on $R$, or on the side of $R$ not containing $h$. In either case, the outgoing edge from $v$ that is on the path $R$ is further left than $e'$, contradicting the definition of $L(e)$. It follows that no right-directed path exists.
\end{proof}

Since we have now shown that the Schnyder peeling process constructs a maximal Schnyder wood of $T$, \cref{thm:unique_maximal_schnyder_wood_finite_bdry} now follows immediately from \cref{thm:max_schnyder_unique_infinite}.

We conclude our study of Schnyder woods of arbitrary infinite triangulations with finite boundary by establishing a result analogous to Schnyder's characterisation of the monochromatic subgraphs in finite Schnyder woods, as given in \cref{thm:schnyder_subgraphs_on_finite_triangulation}.

\begin{theorem}\label{thm:monochromatic_subgraph_finite_bdry}
    Let $T$ be an infinite triangulation with finite boundary $B$ and root $(b_x,b_{x+1})$, and consider the unique maximal Schnyder wood of $T$. Then, the monochromatic subgraphs of $T$ consist of
    \begin{itemize}
        \item An infinite yellow tree rooted at $b_{x+1}$,
        \item An infinite red tree rooted at $b_x$,
        \item An infinite blue tree rooted at some vertex $b_z$ with $z \notin \{x,x+1\}$, and a forest of finite blue trees rooted at vertices other than $b_x$ and $b_{x+1}$ on $B$.
    \end{itemize}
    Furthermore, each monochromatic infinite tree contains exactly one infinite path. Every vertex on the infinite blue path has an outgoing red edge to the infinite red path, every vertex on the infinite red path has an outgoing yellow edge to the infinite yellow path, and every vertex on the infinite yellow path is either a vertex on $B$, or has an outgoing blue edge to the infinite blue path.
\end{theorem}

\begin{proof}
    We first verify the structure of each of the monochromatic subgraphs.
    
    \begin{claim}
        The yellow subgraph of $T$ consists of an infinite tree rooted at $b_{x+1}$ that contains a unique infinite path whose vertices include the heads of every chord chosen by a right step of $P_x$.
    \end{claim}
    
    \begin{poc}
    We showed in \cref{lem:schnyderareayellowforest} that after $i$ steps of $P_x$, the yellow subgraph of the revealed area $\area{x}{i}$ consists of a tree rooted at $b_{x+1}$ that is disjoint from the current upper boundary and corner vertex, and that contains a path whose vertices include the heads $h_\ell$ of every chord chosen by a right step of $P_x$ up to step $i$, along with a forest of trees rooted at vertices on the upper boundary or at the corner vertex of $\area{x}{i}$. By \cref{obs:infiniterightsteps}, $P_x$ will take an infinite number of right steps while exploring $T$, so the yellow subgraph of $T$ will include an infinite yellow tree rooted at $b_{x+1}$ that contains an infinite path whose vertices include $(h_\ell,\ell \geq 0)$. Note that $h_0 = b_{x+1}$.
    
    Consider now a finite yellow tree rooted at a vertex $v$ on the upper boundary or at the corner vertex of $\area{x}{i}$. If $v$ is later covered by a right step $j$ of $P_x$, then the Schnyder conditions for the finite region explored during step $j$ imply that $v$ will have a yellow path to the head of $\roote{x}{j-1}$. Since this head is one of the vertices $h_\ell$, it follows that $v$ will be added to the tree rooted at $b_{x+1}$. Alternatively, if $v$ is later covered by a left step $j$, then the Schnyder conditions imply that after this step, $v$ will be contained in a yellow tree rooted at the head of $\chord{x}{j}$. Observe that no left step of $P_x$ that covers $v$ can choose a chord to a vertex on $\boundary{x}{i}$ anticlockwise between $v$ and the head of the root edge. It follows that the head of $\chord{x}{j}$ is a vertex on $\boundary{x}{i}$ clockwise between $v$ and the head of the current root, and is still not covered after step $j$. Since we know from \cref{obs:infiniterightsteps} that all vertices of $\boundary{x}{i}$ will eventually be covered, it follows that eventually there must be a right step that covers the root of the finite yellow tree containing $v$, and by the same argument as before, this implies that the finite yellow tree containing $v$ is added to the infinite tree rooted at $b_{x+1}$.
    
    Finally, recall that if $P_x$ has performed $\ell$ right steps after step $i$, then the tree rooted at $b_{x+1}$ contains a directed path from $h_\ell$ to $h_0 = b_{x+1}$, and this tree intersects the boundary $\boundary{x}{i}$ only in the root edge $\roote{x}{i}$, the head of which is $h_\ell$. In particular, a path from any vertex revealed after step $i$ to $b_{x+1}$ must include $h_\ell$. This shows that the infinite path whose vertices include $(h_\ell,\ell \ge 0)$ is the unique infinite path in the yellow tree rooted at $b_{x+1}$.
    \end{poc}

    \begin{claim}
        The red subgraph of $T$ consists of an infinite tree rooted at $b_{x}$ that contains a unique infinite path whose vertices are the tails of every chord chosen by a right step of $P_x$.
    \end{claim}
    
    \begin{poc}
        It follows immediately from \cref{lem:schnyderarearedforest} that the red subgraph of $T$ consists of a single infinite red tree rooted at $b_x$, and that this tree contains the infinite path whose vertices are the tails $t_\ell$ of every chord chosen by a right step of $P_x$.

        In particular, if $P_x$ has performed $\ell$ right steps after step $i$, then the tree rooted at $b_x$ contains the directed path $t_\ell, t_{\ell-1}, \dots, t_0$. Moreover, by \cref{obs:finitestepsupperboundary}, no upper boundary or corner vertex of $\area{x}{i}$ can be contained in this tree, and so the red tree can only intersect the boundary $\boundary{x}{i}$ at $t_\ell$ and $h_\ell$. The Schnyder root condition implies that no red edges can be added to $h_\ell$ during subsequent steps, so a path from any vertex revealed after step $i$ to $b_x$ must include $t_\ell$. This shows that the infinite path $\dots, t_2, t_1, t_0$ is the unique infinite path in the red tree rooted at $b_x$.
    \end{poc}

    Let $i_0 \coloneqq 0$, and inductively define $i_\ell$ as the last step of $P_x$ such that $\peelv{x}{i_\ell}$ is on the boundary $\boundary{x}{i_{\ell-1}}$. Each $i_\ell$ exists, since $P_x$ eventually covers the entire boundary $\boundary{x}{i_{\ell-1}}$ by \cref{obs:infiniterightsteps}. Note that whenever $i = i_\ell$, the steps $i_0, \dots, i_\ell$ coincide with the steps defined in \cref{lem:schnyderareablueforest}.
    
    \begin{claim}
        The blue subgraph of $T$ consists of an infinite blue tree rooted at some vertex $b_z$ with $z \notin \{x,x+1\}$, and a forest of finite blue trees rooted at vertices on the initial boundary. Moreover, the infinite blue tree contains a unique infinite path whose vertices are the peeling vertices $\peelv{x}{i_\ell}$ with $\ell \geq 1$.
    \end{claim}
    
    \begin{poc}
        It follows immediately from \cref{lem:schnyderareablueforest} that the blue subgraph of $T$ consists of blue trees rooted at vertices on the initial boundary, and that the blue tree rooted at $\peelv{x}{i_1}$ contains an infinite path whose vertices are the peeling vertices $\peelv{x}{i_\ell}$ with $\ell \geq 1$. Note that $\peelv{x}{i_1} = b_z$ for some $z \notin \{x,x+1\}$.
        
        By \cref{obs:infiniterightsteps}, the process $P_x$ will eventually cover the entire initial boundary of $T$. Let $k$ be sufficiently large such that $P_x$ has covered the entire initial boundary after $i_k$ steps, and denote the upper boundary of $\area{x}{i_k}$ by $U$. Then, \cref{lem:schnyderareablueforest} shows that every vertex on $U$ has an outgoing blue edge to one of the peeling vertices $\peelv{x}{i_\ell}$ for $\ell \le k$. In particular, every vertex on $U$ is contained in the tree rooted at $\peelv{x}{i_1}$, so the blue trees rooted at all other vertices of the initial boundary can no longer be extended after step $i_k$ and will therefore be finite.

        Finally, note that running $P_x$ after step $i_\ell$ for some $\ell \ge 1$ corresponds to running the Schnyder peeling process on the unexplored region of $T$ with initial boundary $\boundary{x}{i_\ell}$ and root edge $\roote{x}{i_\ell}$. So, by the same arguments as above, every blue (sub-)tree that is rooted at a vertex of $\boundary{x}{i_\ell} \sm \set{\peelv{x}{i_{\ell+1}}}$ will be finite. That is, for only a finite number of vertices $v$ in the blue tree of $T$ rooted at $\peelv{x}{i_1}$, the blue path from $v$ to $\peelv{x}{i_1}$ will not include $\peelv{x}{i_{\ell+1}}$. This shows that the infinite path $\dots, \peelv{x}{i_3}, \peelv{x}{i_2}, \peelv{x}{i_1}$ is the unique infinite path in the blue tree rooted at $\peelv{x}{i_1}$.
    \end{poc}
    
    It remains to establish the structure of the outgoing edges between the unique infinite paths. Since the infinite red path consists of the tails of every chord chosen by a right step of $P_x$ while the infinite yellow path contains the heads of each of these chords, it follows immediately that every vertex on the infinite red path has an outgoing yellow edge to a vertex on the infinite yellow path. Furthermore, since the infinite blue path consists entirely of peeling vertices and every peeling vertex has an outgoing red edge to the tail of the current root edge when it is peeled, it follows that every vertex on the infinite blue path has an outgoing red edge to a vertex on the infinite red path.
    
    Finally, consider a vertex $v$ on the infinite yellow path that is not on the initial boundary $B$ of $T$. By \cref{lem:schnyderareayellowforest}, $v$ was added to the infinite yellow path at some right step $i$, at which it lies on the boundary $\boundary{x}{i-1}$. Since $v$ is not on the initial boundary $B$, it follows that $v$ must be on the upper boundary of $\area{x}{i-1}$. Let $i_0' = 0$ and inductively define $i_\ell'$ as the last step from $\set{i_{\ell-1}'+1, \dots, i-1}$ such that $\peelv{x}{i_\ell'}$ is on the boundary $\boundary{x}{i_{\ell-1}'}$. Then, by \cref{lem:schnyderareablueforest}, $v$ has an outgoing blue edge to one of the peeling vertices $\peelv{v}{i_\ell'}$. Since $v$ is covered by the right step $i$ and has an edge to $\peelv{v}{i_\ell'}$, this implies that after step $i$, the entire boundary $\boundary{x}{i_{\ell-1}'}$ is covered. In particular, no later peeling vertex of $P_x$ can ever be on the boundary $\boundary{x}{i_{\ell-1}'}$. This shows that $i_\ell' = i_\ell$, and so $v$ has an outgoing blue edge to one of the peeling vertices $\peelv{x}{i_\ell}$ contained in the infinite blue path.
\end{proof}

We remark that the infinite monochromatic paths described in \cref{thm:monochromatic_subgraph_finite_bdry} have a structure resembling that of the monochromatic paths from any fixed vertex in a finite maximal Schnyder wood, as described in \cref{thm:finite_paths_from_v_wind}. In fact, we can show that the same structure holds for the monochromatic paths from any fixed vertex in the maximal Schnyder wood of an infinite triangulation $T$ with finite boundary. As in \cref{ssec:prelim:schnyderwoods}, for each vertex $v$ of $T$ and each colour $c \in \{y, r, b\}$, denote by $\vpath{c}{v}$\glossarylabel{gl:vpathuipt} the unique directed path of colour $c$ that starts at $v$ in the maximal Schnyder wood of $T$.

\begin{theorem}\label{thm:finite_boundary_paths_from_v_wind}
    Let $T$ be an infinite triangulation with finite boundary, and let $v$ be a vertex of $T$. Then every vertex on $\vpath{b}{v}$ has an outgoing red edge to $\vpath{r}{v}$, every vertex on $\vpath{r}{v}$ has an outgoing yellow edge to $\vpath{y}{v}$, and every vertex on $\vpath{y}{v}$ is either a boundary vertex of $T$ or has an outgoing blue edge to $\vpath{b}{v}$.
\end{theorem}

\begin{proof}
    It follows from the structure in \cref{thm:monochromatic_subgraph_finite_bdry} that each of $\vpath{y}{v}$, $\vpath{r}{v}$, and $\vpath{b}{v}$ is finite. Since the Schnyder peeling process $P_x$ eventually explores all of $T$, we have that after some finite number of steps, it has explored every vertex on $\vpath{y}{v}$, $\vpath{r}{v}$, and $\vpath{b}{v}$, and all of their incident edges. After this step, the region explored by the peeling process is finite. Since this finite region is assigned the same colouring by $P_x$ as if it were contained in a finite triangulation and explored by the finite peeling process, it follows that $\vpath{y}{v}$, $\vpath{r}{v}$, and $\vpath{b}{v}$ satisfy the conditions for maximal Schnyder woods of finite triangulations given in \cref{thm:finite_paths_from_v_wind}.
\end{proof}
 
We conclude this section by establishing that the unique maximal Schnyder wood of the UIPT is in fact the limit of the unique maximal Schnyder woods of uniform finite triangulations.

\begin{theorem} \label{thm:uipt_schnyder_wood_convergence}
    For all $m \in \N$, the maximal Schnyder wood of the UIPT with boundary length $m$ is the weak limit of the maximal Schnyder wood of a uniformly random triangulation from $\cT_n^m$ as $n \to \infty$.
\end{theorem}

\begin{proof}
    Let $T$ be the UIPT with boundary length $m$, and let $T_n$ be a uniformly random triangulation from $\cT_n^m$. Let $r$ denote the root edge of these triangulations. By \cref{thm:unique_maximal_schnyder_wood_finite_bdry}, the Schnyder peeling process constructs the maximal Schnyder wood of $T$. In particular, the peeling process will explore the entire ball $\ballt{\rho}{r}{T}$ of radius $\rho$ around $r$ in $T$ in a finite number of steps. Note that the peeling process will explore only a finite connected region of $T$ during those steps. So, for every $\eps > 0$, there exists a $\rho'$ such that with probability at least $1 - \eps$, the peeling process will explore the entire ball $\ballt{\rho}{r}{T}$ while only exploring edges within the ball $\ballt{\rho'}{r}{T}$. If additionally $\ballt{\rho'}{r}{T_n} = \ballt{\rho'}{r}{T}$, this implies that the maximal Schnyder wood constructed by the peeling process on $T_n$ agrees on the ball $\ballt{\rho}{r}{T_n}$ with the maximal Schnyder wood of $T$ on the ball $\ballt{\rho}{r}{T}$. Since the distribution of $\ballt{\rho'}{r}{T_n}$ converges to the distribution of $\ballt{\rho'}{r}{T}$ as $n \to \infty$ and we may choose $\eps$ arbitrarily small, it follows that the distribution induced on $\ballt{\rho}{r}{T_n}$ by the maximal Schnyder wood of $T_n$ converges to the distribution induced on $\ballt{\rho}{r}{T}$ by the maximal Schnyder wood of $T$. Since this holds for all $\rho$, the maximal Schnyder wood of $T$ is the weak limit of the maximal Schnyder wood of $T_n$ as $n \to \infty$.
\end{proof}

\section{The peeling process on the UIHPT} \label{sec:uihpt}

We now turn to the more challenging setting of infinite triangulations of the half-plane, and immediately restrict our attention to the random triangulation which is our focus: the UIHPT. In this section, we prove a number of preliminary results that will allow us to establish the existence of a Schnyder wood of the UIHPT in the following section. The section has two parts. In \cref{ssec:uihpt:segment}, we show that unlike for triangulations with finite boundary, the Schnyder peeling process is not guaranteed to explore the entire triangulation if it has an infinite boundary. In particular, on the UIHPT, the peeling process almost surely only explores edges within a one-ended narrow region along the initial boundary which extends only a finite distance to the left of the initial root edge. We name the region that is so explored the \defn{Schnyder wood segment}.

In \cref{ssec:uihpt:strip}, we show that we may colour a two-ended narrow strip along the entire initial boundary by rerunning the Schnyder peeling process from boundary vertices progressively further left. The bulk of the work consists of showing that this process provides a consistent definition of the colour and orientation assigned to a given edge. We will call the region coloured by this process the \defn{Schnyder wood strip}.

In \cref{sec:structure}, we will use the Schnyder wood segment and Schnyder wood strip to construct a unique maximal Schnyder wood of the UIHPT. Our strategy will be to show that by repeatedly removing a segment or strip from the UIHPT, we will eventually colour the entire half-plane and thereby construct a unique maximal Schnyder wood. To facilitate the iterated removal of these narrow regions along the boundary, we prove in \cref{sssec:uihpt:schnydersegmentremoved} and \cref{sssec:uihpt:schnyderstripremoved} that removing, respectively, a Schnyder wood segment or a Schnyder wood strip from the UIHPT leaves a triangulation of the half-plane that is also distributed like a UIHPT. This means that the Schnyder peeling process behaves in a distributionally identical way on the remaining unexplored region.

\subsection{The Schnyder wood segment} \label{ssec:uihpt:segment}

We begin by calculating the probability that the Schnyder peeling process takes a certain step when applied to the UIHPT. Equivalently, this is the probability that the partial triangulation revealed by the step is present in the UIHPT. By definition of the UIHPT as a local limit in \cref{ssec:prelim:inftriangulations}, we can obtain this probability by first calculating the probability that the partial triangulation is present in a uniformly random finite triangulation from $\cT_n^m$, and then taking the limit of this probability as first $n \to \infty$ and then $m \to \infty$.

Suppose that we want to calculate the probability that a step of the Schnyder peeling process in the UIHPT chooses the chord $(v_0, v_c)$ and also reveals $k$ neighbours of $v_0$ along with some given triangulation $T_s$ (where $s = r$ if $c > 0$ and $s = \ell$ if $c < 0$). For any sufficiently large $m \in \bN$, we associate the following partial triangulation $S'$ of an $(m+2)$-gon with this step.
\begin{itemize}
    \item Denote the boundary vertices of the $(m+2)$-gon by $v_0, \dots, v_{m+1}$ and the root edge by $(v_1, v_2)$. To obtain $S'$, add the chord $v_0 v_c$ where $v_c$ denotes the vertex $v_i$ with $i \equiv c \bmod{(m+2)}$, and add $k$ neighbours $u_1, \dots, u_k$ to $v_0$ in anticlockwise order between $v_0 v_1$ and $v_0 v_c$. Also, add the edges of the path $v_1, u_1, \dots, u_k, v_c$.
    \item If $c > 0$, fill the region of $S'$ enclosed by $v_1, \dots, v_c, u_k, \dots, u_1, v_1$ and rooted at $(v_1, v_2)$ with the triangulation $T_s$. Otherwise, if $c < 0$, fill the region enclosed by $v_0, v_c, v_{c+1}, \dots, v_0$ and rooted at $(v_0, v_c)$ with the triangulation $T_s$.
\end{itemize}

Now, if $S \in \cT_n^m$ is a uniformly random finite triangulation with $n$ sufficiently large, we need to calculate the probability that $S' \subs S$. This is equivalent to counting the number of ways to extend $S'$ to a triangulation with $n$ interior vertices.

Let $T_s$ have $m_s+2$ boundary vertices and $n_s$ interior vertices. Then $S'$ has $n_s+k$ interior vertices and a single unexplored region with $m-m_s+1+k$ boundary vertices. This implies that the number of ways to extend $S'$ to a triangulation with $n$ interior vertices is $\nt_{m_h,n_h}$ where $m_h = m-m_s-1+k$ and $n_h = n-n_s-k$. So, by the asymptotics for $\nt_{m,n}$ given in \cref{ssec:prelim:fintriangulations},
\[
    \pr(S' \subs S) = \frac{\nt_{m_h,n_h}}{\nt_{m,n}} \sim \frac{\ntfact_{m_h} \ntbase^{n_h} (n_h)^{-5/2}}{\ntfact_m \ntbase^n n^{-5/2}} \sim \frac{\ntfact_{m_h}}{\ntfact_m} \cdot \ntbase^{-n_s-k}
\]
as $n \to \infty$. Since $m_h = m-m_s-1+k$, we have that
\[
    \frac{\ntfact_{m_h}}{\ntfact_m} \cdot \ntbase^{-n_s-k} \sim \frac{(64/9)^{m_h} (m_h)^{1/2}}{(64/9)^m m^{1/2}} \cdot \ntbase^{-n_s-k} \sim \ntbase^{-n_s-k} \cdot \left(\frac{9}{64}\right)^{m_s+1-k} \eqqcolon p_{s,k,T_s}
\]
as $m \to \infty$. So, $p_{s,k,T_s}$ is the probability that the Schnyder peeling process takes this specific step in the UIHPT. Importantly, $p_{s,k,T_s}$ only depends on $T_s$ via its size $n_s$ and boundary length $m_s$.

Next, we calculate the probability that a step of the Schnyder peeling process chooses the chord $(v_0, v_c)$ and reveals $k$ additional neighbours of $v_0$, no matter what exact triangulation $T_s$ it reveals. This probability is
\begin{align*}
    p_{s,k,m_s} \coloneqq {} & \sum_{T_s \in \cT^{m_s}} p_{s,k,T_s} = \sum_{n_s} \sum_{T_s \in \cT_{n_s}^{m_s}} p_{s,k,T_s} = \sum_{n_s} \nt_{m_s,n_s}\cdot \ntbase^{-n_s-k}\left(\frac{9}{64}\right)^{m_s+1-k}\\
    = {} & \fdgen_{m_s}\cdot \ntbase^{-k}\left(\frac{9}{64}\right)^{m_s+1-k} = \frac{9}{32}\frac{(2m_s)!}{m_s!(m_s+2)!}4^{-m_s}\left(\frac{3}{4}\right)^k
\end{align*}
Using the easy-to-check identity
\[
    \frac{6}{(x+1)(x+2)} \binom{2x}{x} = 4 \cat_x - \cat_{x+1},
\]
where $C_x$ is the $x$-th Catalan number, we get that
\begin{equation}\label{eq:transition_prob}
    p_{s,k,m_s} = \frac{3}{64} 4^{-m_s} (4 \cat_{m_s} - \cat_{m_s+1}) \left(\frac{3}{4}\right)^k.
\end{equation}

We now calculate the total probability that the Schnyder peeling process takes a left step ($s = \ell$; $c < 0$). In this case, there is no restriction on $m_\ell$ and $k$, and so summing over $m_\ell$ and $k$ yields
\begin{align*}
    p_\ell \coloneqq {} & \sum_{m_\ell,k \geq 0} p_{\ell,k,m_\ell} = \frac{3}{64} \sum_{k \geq 0} \left(\frac{3}{4}\right)^k \sum_{m_\ell \geq 0} 4^{-m_\ell} (4 \cat_{m_\ell} - \cat_{m_\ell+1}) \\
    = {} & \frac{3}{16} \cdot\lim_{m_\ell \to \infty} \left(4 \cat_0 - \frac{\cat_{m_\ell+1}}{4^{m_\ell}}\right) = \frac{3}{16} 4 \cat_0 = \frac{3}{4}.
\end{align*}
It follows that the probability of taking a right step ($s = r$; $c > 0$) is $1/4$, but for consistency, we verify this by computation. Note that the only restriction on $T_r$ when taking a right step is that $m_r \geq k$. So, summing over $m_r$ and $k$ and using the fact that the generating function for the Catalan numbers is $c(x)=(1-\sqrt{1-4x})/(2x)$ yields
\begin{align*}
    p_r \coloneqq {} & \sum_{k \geq 0} \sum_{m_r \geq k} p_{c,k,m_r} = \frac{3}{64} \sum_{k \ge 0} \left(\frac{3}{4}\right)^k \sum_{m_r \geq k} 4^{-m_r} (4 \cat_{m_r} - \cat_{m_r+1}) \\
    = {} & \frac{3}{64} \sum_{k \geq 0} \left(\frac{3}{4}\right)^k \lim_{m_r \to \infty} \left(\frac{4 \cat_k}{4^k} - \frac{\cat_{m_r+1}}{4^{m_r}}\right) \\
    = {} & \frac{3}{16} \sum_{k \geq 0} \cat_k \left(\frac{3}{16}\right)^k = \frac{3}{16} \cdot c\left(\frac{3}{16}\right) = \frac{1}{4}.
\end{align*}
In particular, note that the probabilities $p_\ell$ and $p_r$ of taking left and right steps sum to 1, as expected.

In addition to just giving the probability of a step, these calculations also tell us the probability distributions of the finite triangulation $T_s$ and of the unexplored region of $T'$. For the distribution of $T_s$, note that the dependence of the probability $p_{s,k,T_s}$ on $n_s$ is $\alpha^{-n_s}$. So, for fixed $s$, $k$, and $m_s$, the distribution of $T_s$ must be the free distribution. We now show that the unexplored region is distributed like a UIHPT. This is important since it implies that the probability distributions of all future steps of the Schnyder peeling process are the same as the probability distribution of the first step.

\begin{lemma}\label{lem:hole_uihpt}
    Let $T'$ be a fixed partial triangulation revealed by a peeling step, and let $r$ be the new root edge after the peeling step. If $T$ is a UIHPT conditioned on $T' \subs T$, then $T \sm \intr T'$ has the same distribution as a UIHPT rooted at $r$.
\end{lemma}

\begin{proof}
    Let $s$, $k$, and $T_s$ be the parameters of the peeling step that reveals $T'$, and let $S'$ be the associated partial triangulation of an $(m+2)$-gon defined above. Fix $m$ and $n$ sufficiently large, and let $S$ be a uniformly random element of $\cT_n^m$ conditioned on the event that $S' \subs S$.
    
    Consider the ball $\ballt{\rho}{r}{S \sm S'}$ of radius $\rho$ around $r$ in $S \sm S'$. By the above calculation, $\ballt{\rho}{r}{S \sm S'}$ has the same distribution as a ball of radius $\rho$ in a uniformly random finite triangulation of $\cT_{n_h}^{m_h}$ where $n_h = n - n_s - k$ and $m_h = m - m_s - 1 + k$. So, as $n \to \infty$ and then $m \to \infty$, the distribution of $\ballt{\rho}{r}{S \sm S'}$ converges to the distribution of a ball of radius $\rho$ in a UIHPT rooted at $r$.
    
    Since the distribution of $\ballt{\rho}{r}{T \sm T'}$ is the limit of the distribution of $\ballt{\rho}{r}{S \sm S'}$, we get that $\ballt{\rho}{r}{T \sm T'}$ is distributed like a ball of radius $\rho$ in a UIHPT rooted at $r$. As this holds for all $\rho$, it follows that $T \sm T'$ has the same distribution as a UIHPT rooted at $r$.
\end{proof}

We now consider the behaviour of the Schnyder peeling process on the UIHPT over multiple steps. For this purpose, let $T$ be a UIHPT with boundary $B = \dots,b_{-2},b_{-1},b_0,b_1,b_2,\dots$ and let $P_x$ be the Schnyder peeling process initiated with root edge $(b_x,b_{x+1})$. Step $i$ of $P_x$ performs a peeling step based on the boundary and root edge before step $i$. This step updates the boundary and $P_x$ chooses a new root edge on the updated boundary for the next peeling step.

Note that if we perform a single step of the Schnyder peeling process on the UIHPT, then by \cref{lem:hole_uihpt} we know that the unexplored region that remains after the step will have the same distribution as the UIHPT. Therefore, by induction it follows that the unexplored region after any number of steps will have the same distribution as the UIHPT. In particular, the transition probabilities of all steps of the Schnyder peeling process are the same ones as calculated above.

We now show that, unlike for triangulations with finite boundary, the Schnyder peeling process on the UIHPT will almost surely not explore the whole plane. In fact, the following lemma shows that the Schnyder peeling process will only explore a finite number of vertices on the initial boundary left of its root edge. Recall that a vertex $v$ is \defn{covered} by some step of $P_x$ if $v$ is contained in the interior of the partial triangulation that has been explored after that step, and that $P_x$ \defn{eventually covers} $v$ if there exists some step $i$ after which $v$ is covered.

\begin{lemma}\label{lem:peeling_moves_right}
    Almost surely for all $x \in \Z$ there exists some $y < x$ such that the Schnyder peeling process $P_x$ eventually covers $b_z$ if and only if $z \ge y$.
\end{lemma}

\begin{proof}
    It follows from \cref{obs:infiniterightsteps} that $P_x$ will eventually cover all $b_z$ with $z \geq x$. It remains to determine how many initial boundary vertices $b_z$ with $z \leq x$ will eventually be covered. Consider a single step of $P_x$:
    \begin{itemize}
        \item If the step is a right step, it will cover one boundary vertex left of the root edge. The probability of a right step is $1/4$.
        \item If the step is a left step, it will cover some number $j$ of boundary vertices left of the root edge, but also add $k$ new vertices to the boundary directly to the left of the root edge. By applying \eqref{eq:transition_prob}, we see that the probability of this is
        \[
            p_{j,k} \coloneqq p_{\ell,k,j-1} = \frac{3}{4} q_j r_k \text{ where } q_j = 4^{-j} (4 C_{j-1} - C_j) \text{ and } r_k = \frac{1}{4} \left(\frac{3}{4}\right)^k,
        \]
        Note that $\sum_{j \ge 1} q_j = 1$ and $\sum_{k \ge 0} r_k = 1$. In particular, conditioned on the step being a left step, $j$ and $k$ are independent with probability mass functions $(q_j)_{j \geq 1}$ and $(r_k)_{k \geq 0}$, respectively.
    \end{itemize}
    Therefore, the distribution of the index of the leftmost boundary vertex that is eventually covered by $P_x$ is equal in distribution to the minimal value of a random walk $(S_n)_{n\geq0}$ with $S_0 = x$, and $S_n = x + \sum_{i=1}^n \xi_i$ where $\xi_i$ are i.i.d. random variables with the following distribution:
    \begin{itemize}
        \item $\xi_i = -1$ with probability $1/4$ (a right step), and
        \item $\xi_i = k-j$ with probability $p_{j,k}$ (a left step).
    \end{itemize}

    Note that if a random variable $J$ has probability mass function $(q_j)_{j \geq 1}$, then
    \[
        \ex(J) = \sum_{i \ge 1} \pr(J \ge i) = \sum_{i \ge 1} \sum_{j \ge i} q_j = \sum_{i \ge 0} 4^{-i} \cat_i = c\left(\frac{1}{4}\right) = 2
    \]
    where we recall that $c(x) = (1 - \sqrt{1-4x})/(2x)$ is the generating function for Catalan numbers. Also, if a random variable $K$ has probability mass function $(r_k)_{k \geq 0}$, then
    \[
        \ex(K) = \sum_{i \ge 1} \pr(K \ge i) = \sum_{i \ge 1} \sum_{k \ge i} r_k = \sum_{i \ge 1} \left(\frac{3}{4}\right)^i = 3.
    \]
    So, $\ex(\xi_i \st \text{left step}) = \ex(K) - \ex(J) = 1$ which implies that $\ex(\xi_i) = \frac{3}{4}\cdot 1 - \frac{1}{4}\cdot 1 = 1/2$. Furthermore, we have 
    \[
        \ex(\abs{\xi_i} \st \text{left step}) \leq \ex(\abs{K}) + \ex(\abs{J}) = 5
    \]
    and so $\ex(\abs{\xi_i}) \leq \frac{3}{4} \cdot 5 + \frac{1}{4} \cdot 1 < \infty$. 
    
    Hence, the strong law of large numbers gives that almost surely,
    \[
        \lim_{n\to\infty} \frac{S_n}{n} = \frac{1}{2},
    \]
    and so the Schnyder process $P_x$ almost surely eventually covers only a finite number of boundary vertices $b_z$ with $z \leq x$. 
\end{proof}

Observe that at each step of the Schnyder peeling process on an infinite triangulation of the half-plane, the head of the root edge is always on the initial boundary. As the peeling vertex is one vertex left on the current boundary from the tail of the root edge, the distance from the peeling vertex to the initial boundary is therefore at most two, and hence the distance from every vertex on the updated boundary to the initial boundary is at most three. This implies that the Schnyder peeling process can only explore within the region of those vertices of the UIHPT whose distance to the initial boundary is at most three (plus any finite component completely contained within that region). \cref{lem:peeling_moves_right} shows that in fact, on the UIHPT the Schnyder peeling process almost surely only explores a one-ended infinite segment along the initial boundary. We call this the \defn{Schnyder wood segment}. In particular, the peeling process almost surely does not assign a Schnyder wood to the entire UIHPT.

\subsubsection{Removing the Schnyder wood segment}\label{sssec:uihpt:schnydersegmentremoved}

In \cref{sec:structure} we will describe how we can use a strategy of repeatedly colouring and removing infinite strips along the boundary to define a Schnyder wood of the entire UIHPT. To guarantee that we may repeat this process, we now prove that the triangulation left behind after removing a Schnyder wood segment is distributed like a UIHPT. To state this result, recall that $\area{x}{i}$ denotes the revealed area up to and including step $i$, and $\evarea{x} \coloneqq \bigcup_i \area{x}{i}$ is the revealed area of the Schnyder wood segment rooted at $(b_x, b_{x+1})$. Also, $\intr \area{x}{i}$ is the interior of the revealed area, and $\intr \evarea{x} \coloneqq \bigcup_i \intr \area{x}{i}$. Recall that the interior $\intr \area{x}{i}$ is defined to include those vertices and edges of $\area{x}{i}$ that lie on the initial boundary of $T$, but not the updated boundary after step $i$. 

\begin{lemma}\label{lem:remove_one_ended_strip_still_uihpt}
    Let $f_i$ be the first edge on the initial boundary left of $b_x$ that is not contained in $\area{x}{i}$. Then almost surely $\lim_{i \to \infty} f_i = f$ for some edge $f$ on the initial boundary, and the triangulation $T \sm \intr \evarea{x}$ has the same distribution as a UIHPT rooted at $f$.
\end{lemma} 

\begin{proof}
    By \cref{lem:peeling_moves_right} there almost surely exists some $y < x$ such that $P_x$ eventually covers $b_z$ if and only if $z \geq y$. So, $\lim_{i\to \infty} f_i = (b_{y-1}, b_y) \eqqcolon f$ almost surely.
    
    Let $T_i \coloneqq T \sm \intr \area{x}{i}$ be the unexplored area of $T$ after $i$ steps, and consider the ball $\ballt{\rho}{f}{T_i}$ of radius $\rho$ around $f$ in $T_i$. Since $\ballt{\rho}{f}{T}$ is almost surely finite and we have that $\area{x}{i} \subs \area{x}{i+1}$ for every step $i$, there almost surely exists some step $j$ such that $\ballt{\rho}{f}{T_i} = \ballt{\rho}{f}{T \sm \intr \evarea{x}}$ for all $i \ge j$. By choosing $j$ sufficiently large we may also assume that $f_i = f$ for all $i \ge j$. In particular, as $i$ increases, $\ballt{\rho}{f_i}{T_i}$ almost surely converges to $\ballt{\rho}{f}{T \sm \intr \evarea{x}}$.
    
    For every fixed $i$, \cref{lem:hole_uihpt} implies that $\ballt{\rho}{f_i}{T_i}$ is distributed like a ball of radius $\rho$ in a UIHPT rooted at $f_i$. Since $\ballt{\rho}{f_i}{T_i}$ almost surely converges to $\ballt{\rho}{f}{T \sm \intr \evarea{x}}$, it follows that $\ballt{\rho}{f}{T \sm \intr \evarea{x}}$ is almost surely distributed like a ball of radius $\rho$ in a UIHPT rooted at $f$. Since this holds for all $\rho$, it follows that $T \sm \intr \evarea{x}$ has the same distribution as a UIHPT rooted at $f$.
\end{proof}

\subsection{The Schnyder wood strip}\label{ssec:uihpt:strip}

As shown in \cref{ssec:uihpt:segment}, the Schnyder peeling process on the UIHPT will only colour a one-ended infinite Schnyder wood segment along the initial boundary, and will therefore never reveal edges far left of the initial root edge. To define a Schnyder wood of the entire UIHPT, we need to find a way of colouring the remainder of the half-plane.

One possible attempt to do this could be to colour a Schnyder wood segment, remove it from the UIHPT, choose a new root edge on the new boundary, and then repeat these steps on the remainder of the UIHPT. However, it is unclear how to choose the new root edge on the new boundary, and different choices of root edge yield different colourings of the UIHPT. Moreover, this approach would result in many vertices on the initial boundary which do not satisfy the Schnyder boundary condition, while in the finite case there are only two boundary vertices which do not satisfy this condition.

Instead, we will now define the \defn{Schnyder wood strip}, which is a colouring of a two-ended infinite strip along the initial boundary of the UIHPT. This strip will not depend on the choice of a root edge, and all of its boundary vertices will satisfy the Schnyder boundary condition. To construct the strip, we rerun the Schnyder peeling process from root edges progressively further left along the initial boundary. This presents an issue: the colour and direction assigned to an edge depend on the position of the initial root edge of the process. To use this approach to assign a well-defined colour and direction to a given edge, we will prove that if we run two Schnyder peeling processes from root edges sufficiently far left along the initial boundary, then either both processes will assign the same colour and direction to the given edge that we want to colour, or neither process will colour the edge. This will allow us to define the Schnyder wood strip to be the limit of the Schnyder wood segments coloured by the peeling processes $P_x$ as $x \to -\infty$. In this section, we show that this limit exists. In fact, we will prove that from some point onward, both of the Schnyder peeling processes will perform the exact same steps and therefore assign the exact same colouring and orientation to all subsequently explored edges. 

Our argument is inspired by the behaviour of the Schnyder peeling process determined in \cref{lem:peeling_moves_right}. Since we expect any individual Schnyder peeling process to explore only a finite number of edges to the left on its boundary, we may guess that any two Schnyder peeling processes will eventually begin taking the same steps, provided that there exist steps at which the two processes agree on their root edges and a sufficiently long segment of their boundaries left of the common root edge. To show that the processes never explore left of such a common boundary segment, we first prove that vertices that lie far apart on the UIHPT boundary are unlikely to be joined by short paths, and then use this to guarantee that no length-$3$ path has endpoints on both sides of the common boundary segment. We then argue that this suffices to show that the peeling processes will never explore far enough left to overshoot the common boundary segment. As a result, all subsequent peeling steps of the two processes will be the same.

Formally, our proof consists of the following steps.
\begin{enumerate}[(1)]
    \item \textbf{No short paths.} We show that there are unlikely to be short paths whose endpoints are vertices of the UIHPT that are far apart along its boundary, and on either side of a fixed boundary edge.
    \item \textbf{Roots eventually coincide.} We show that for any $x$ and $y$, there exist steps at which the processes $P_x$ and $P_y$ have the same root edge.
    \item \textbf{Segments of boundaries eventually coincide.} We show that for any fixed vertex $v$, there exist steps at which all Schnyder peeling processes rooted sufficiently far left share a common boundary segment left of their common root edge, and at that point none of the processes have explored $v$. Furthermore, at these steps, none of the respective processes have short paths connecting boundary vertices on either side of the common boundary segment. 
    \item \textbf{Existence of the Schnyder wood strip.} We show that the previous conditions suffice to prove that for any fixed $\rho$, all processes rooted sufficiently far left will produce consistent colourings on $\ball{\rho}{b_0}$, and we use this to formally define the Schnyder wood strip.
\end{enumerate}

\subsubsection{Step 1: No short paths}

We begin by showing that it is unlikely that the UIHPT has short paths connecting two vertices on its boundary which are far apart along the boundary. Recall that a vertex $v$ is \defn{eventually covered} by $P_x$ if there exists a step $i$ such that $v$ is contained in $\intr \area{x}{i}$. The following result is a straightforward corollary of \cref{lem:peeling_moves_right}.

\begin{corollary}\label{cor:boundedleftmvmt}
    For all $\eps > 0$, there exists $k \in \N$ such that for all $x \in \Z$
    \[
        \pr(P_x \text{ eventually covers } b_z \text{ for some } z < x-k) \le \eps.
    \]
\end{corollary}

\begin{proof}
    Let $E_y$ be the event that $P_x$ eventually covers $b_z$ if and only if $z \ge y$. Note that these events are disjoint for different $y$. By \cref{lem:peeling_moves_right},
    \[
        \sum_{y \le x} \pr(E_y) = \pr\left(\bigcup_{y \le x} E_y\right) = 1.
    \]
    In particular, there exists some $k \in \N$ such that $\sum_{y = x-k}^x \pr(E_y) \ge 1 - \eps$. Note that if any of the events $E_y$ with $x-k \le y \le x$ occurs, then no boundary vertex $b_z$ with $z < x-k$ is eventually covered by $P_x$. Therefore,
    \[
        \pr(P_x \text{ eventually covers } b_z \text{ for some } z < x-k) \le 1 - \pr\left(\bigcup_{y = x-k}^x E_y\right) \le \eps.
    \]
    Finally, note that $k$ does not depend on $x$, since we know by \cref{lem:translation_invariance} that the UIHPT is shift invariant.
\end{proof}

 We now use our previous results about the behaviour of the modified Schnyder peeling process to demonstrate that short spanning paths are unlikely in the UIHPT. If $B'$ is a segment of the boundary $B$ of a UIHPT $T$, we say that a path in $T$ \defn{spans} $B'$ if it connects two vertices of $B \sm B'$ that are on opposite sides of $B'$.

\begin{lemma}\label{lem:no_short_paths}
    For all $d > 0$ and $\eps > 0$ there exists $n \in \N$ such that for all $x \in \Z$,
    \[
        \pr(\text{some path of length at most } d \text{ spans } b_{x-n}, \dots, b_{x-1}, b_{x} \subs B) \leq \eps.
    \]
\end{lemma}

\begin{proof}
    Let $P_y^S$ denote the Schnyder peeling process applied to a UIHPT $S$ of the half-plane whose boundary contains $(b_x: x \leq y+1)$ and whose initial root edge is $(b_y, b_{y+1})$. Let $\sparea{y}{i}{S}$ denote the revealed area of $S$ after step $i$ of $P_y^S$, and define $\spevarea{y}{S} \coloneqq \bigcup_i \sparea{y}{i}{S}$. We apply \cref{cor:boundedleftmvmt} with $\eps/d$ to obtain $k \in \N$ such that for all $y \in \Z$ and all UIHPTs $S$ whose boundary contains $(b_x: x \leq y+1)$,
    \[
        \pr\left(P_y^S \text{ eventually covers } b_z \text{ for some } z < y-k\right) \leq \frac{\eps}{d}.
    \]
    
    Let $T_0 \coloneqq T$ and $x_0 \coloneqq x$. We construct a sequence of triangulations as follows. Suppose that $T_i$ is a UIHPT whose boundary contains $\dots, b_{x_i-1}, b_{x_i}, b_{x_i+1}$. By \cref{lem:peeling_moves_right}, there almost surely exists some $x_{i+1} \in \Z$ with $x_{i+1} + 1 < x_i$ such that $\intr \spevarea{x_i}{T_i}$ contains a boundary vertex $v$ of $T_i$ if and only if $v > b_{x_{i+1} + 1}$. If this event occurs, let $T_{i+1} \coloneqq T_i \sm \intr \spevarea{x_i}{T_i}$, and note that by \cref{lem:remove_one_ended_strip_still_uihpt} we know that $T_{i+1}$ is a UIHPT rooted at $(b_{x_{i+1}}, b_{x_{i+1}+1})$.

    Almost surely, this process constructs the triangulation $T_d$. If this happens, observe that for each $0 \leq i \leq d-1$, the revealed area $\spevarea{x_i}{T_i}$ contains $\ballt{1}{v}{T_i}$ for every boundary vertex $v$ of $T_i$ with $v \ge b_{x_i}$. This implies that $\bigcup_{i=0}^{d-1} \spevarea{x_i}{T_i}$ contains $\ball{d}{b_y}$ for all $y \ge x$. Moreover, by the choice of $k$, it holds that $\pr(x_{i+1} < x_i - k) \le \eps/d$, and so it follows that $\pr(x_d \ge x - d \cdot k) \ge 1 - \eps$. Thus, with probability at least $1 - \eps$, $\bigcup_{i=0}^{d-1} \spevarea{x_i}{T_i}$ contains no boundary vertex $b_z$ with $z < x - d \cdot k$, and so there is no path of length at most $d$ that spans $b_{x - d \cdot k}, \dots, b_{x-1}, b_x \subs B$.
\end{proof}

\subsubsection{Step 2: Roots eventually coincide}

Next, we consider two Schnyder peeling processes $P_x$ and $P_y$ with $y < x$ and we want to show that both processes eventually explore from the same root edge. First, we prove the following observation which states conditions under which $P_x$ only reveals parts of the triangulation that $P_y$ has already revealed.

\begin{observation} \label{lem:area_subset}
    If $\area{x}{j-1} \subseteq \area{y}{i}$ and $\peelv{x}{j} \in \intr \area{y}{i}$, then $\area{x}{j} \subseteq \area{y}{i}$.
\end{observation}

\begin{proof}
    Note that step $j$ of $P_x$ reveals some unrevealed edges incident to the peeling vertex $\peelv{x}{j}$ as well as the part of $T$ enclosed by the chord $\chord{x}{j}$ and the boundary $\boundary{x}{j-1}$. Since $\peelv{x}{j}$ is contained in $\intr \area{y}{i}$, planarity ensures that all edges incident to $\peelv{x}{j}$ are contained in $\area{y}{i}$. Since $\chord{x}{j}$ is incident to $\peelv{x}{j}$, this also implies that $\chord{x}{j}$ is contained in $\area{y}{i}$, and so the part of $T$ enclosed by $\chord{x}{j}$ and $\boundary{x}{j-1}$ is contained in $\area{y}{i}$. Hence, the entire partial triangulation of $T$ revealed during step $j$ of $P_x$ is contained in $\area{y}{i}$, and so it follows that $\area{x}{j} \subseteq \area{y}{i}$.
\end{proof}

We can now prove that the root edges of both processes will eventually coincide. 

\begin{lemma}\label{lem:area_containment}
    Suppose that $\area{y}{i}$ contains $\inroote{x}$. Then, there exists some step $j$ of $P_x$ such that $\roote{x}{j} = \roote{y}{i}$ and $\area{x}{j} \subseteq \area{y}{i}$.
\end{lemma}

\begin{proof}
    Consider the sequence $t_0, \dots, t_\ell$ of the tails of the root edges $\roote{y}{0},\dots,\roote{y}{i}$ with duplicates removed, as depicted in \cref{fig:coinciding_root_edges}. Let $W$ be the region of $T$ that is enclosed by the path $t_0, \dots, t_\ell$, the root edge $\roote{y}{i}$, and the initial boundary. Note that $W$ contains the initial root edge $\inroote{x}$ of $P_x$. We show that the root edge of $P_x$ remains in $W$ up to the first step $j$ of $P_x$ at which the peeling vertex $\peelv{x}{j}$ is no longer contained in $\intr \area{y}{i}$. This scenario is depicted in \cref{fig:area_W}.

    \begin{figure}[h]
        \centering
        \begin{tikzpicture}
    \foreach \x in {1,...,16}{
        \node[mycircle] (bn\x) at (\x,0) {};
    }
    \foreach \x in {1,...,4,7,8,...,11,13,14,15}{
        \pgfmathtruncatemacro{\next}{\x + 1};
        \draw (bn\x) -- (bn\next);
    }
    \draw (bn5) edge[->-,redge] (bn6)
          (bn6) edge[->-,yedge] (bn7)
          (bn12) edge[->-,color=orange] (bn13);

    \node[mycircle] (t1) at (6,1.5) {};
    \node[mycircle] (t2) at (6,3) {};
    \node[mycircle] (t3) at (6,4.5) {};

    \node[mycircle] (t1u1) at (6.5,1) {};
    \draw (bn6) to (t1u1) to[bend left=20] (bn9)
          (t1u1) edge[->-,bedge] (t1);
    \node[mycircle] (t2u1) at (6.3,2.2) {};
    \node[mycircle] (t2u2) at (6.7,2.5) {};
    \node[mycircle] (t2u3) at (7.3,2.65) {};
    \draw (t1) to (t2u1) to (t2u2) to (t2u3) to[bend left=20] (bn12)
          (t2u1) edge[->-,bedge] (t2)
          (t2u2) edge[->-,bedge] (t2)
          (t2u3) edge[->-,bedge] (t2);
    \node[mycircle] (t3u1) at (6.3,3.8) {};
    \node[mycircle] (t3u2) at (6.9,4.1) {};
    \draw (t2) to (t3u1) to (t3u2) to[bend left=20] (bn15)
          (t3u1) edge[->-,bedge] (t3)
          (t3u2) edge[->-,bedge] (t3);

    \draw (t1) edge[->-,redge] (bn6)
          (t1) edge[->-,yedge,bend left=30] (bn9);
    \draw (t2) edge[->-,redge] (t1)
          (t2) edge[->-,yedge,bend left=30] (bn12);
    \draw (t3) edge[->-,redge] (t2)
          (t3) edge[->-,yedge,bend left=30,"$\roote{y}{i}$" black] (bn15);

    \node[mycircle] (t0l1) at (5,1) {};
    \node[mycircle] (t0l2) at (5.3,2) {};

    \node[mycircle] (t0l0u1) at (4,1.4) {};
    \node[mycircle] (t0l0u2) at (3,1.2) {};
    \draw (bn5) edge[->-,yedge,bend right=45] (bn2);
    \draw (bn2) to (t0l0u2) to (t0l0u1) to (t0l1)
          (t0l0u2) edge[->-,bedge,bend left=20] (bn5)
          (t0l0u1) edge[->-,bedge,bend left=15] (bn5)
          (t0l1) edge[->-,bedge] (bn5);
    \node[mycircle] (t0l1u1) at (4.7,2.2) {};
    \coordinate (t0l1l) at (2,2.3);
    \draw (t0l1) edge[->-,yedge,bend right=15,dotted end] (t0l1l)
          (t0l1) edge[->-,redge] (bn6);
    \draw (t0l1u1) edge[bend right=8,dotted end] (t0l1l) (t0l1u1) to (t0l2)
          (t0l1u1) edge[->-,bedge] (t0l1)
          (t0l2) edge[->-,bedge] (t0l1);
    \node[mycircle] (t0l2u1) at (5.3,3) {};
    \coordinate (t0l2l) at (2,3.3);
    \draw (t0l2) edge[->-,yedge,bend right=15,dotted end] (t0l2l)
          (t0l2) edge[->-,redge] (bn6);
    \draw (t0l2u1) edge[bend right=12,dotted end] (t0l2l)
          (t0l2u1) edge[->-,bedge] (t0l2)
          (t2) edge[->-,bedge] (t0l2)
          (t1) edge[->-,bedge] (t0l2);

    \node[mycircle] (t2l1) at (5.3,4) {};
    
    \node[mycircle] (t2l0u1) at (4.7,4.2) {};
    \node[mycircle] (t2l0u2) at (4,4.3) {};
    \coordinate (t2l0l) at (2.45,4.3);
    \draw (t0l2u1) edge[->-,yedge,bend right=10,dotted end] (t2l0l)
          (t0l2u1) edge[->-,redge] (t2);
    \draw (t2l0u2) edge[bend right=5,dotted end] (t2l0l) (t2l0u2) to (t2l0u1) to (t2l1)
          (t2l0u2) edge[->-,bedge] (t0l2u1)
          (t2l0u1) edge[->-,bedge] (t0l2u1)
          (t2l1) edge[->-,bedge] (t0l2u1);
    \node[mycircle] (t2l1u1) at (5.3,5) {};
    \coordinate (t2l1l) at (2.4,5.3);
    \draw (t2l1) edge[->-,yedge,bend right=15,dotted end] (t2l1l)
          (t2l1) edge[->-,redge] (t2);
    \draw (t2l1u1) edge[bend right=12,dotted end] (t2l1l) (t2l1u1) to (t3)
          (t2l1u1) edge[->-,bedge] (t2l1)
          (t3) edge[->-,bedge] (t2l1);

    \node[below left=0cm and -0.325cm of bn6] (t0l) {$\vphantom{r_y}t_0$};
    \node[above right=-0.125cm and 0cm of t1] {$t_1$};
    \node[above right=-0.125cm and 0cm of t2] {$t_2$};
    \node[above right=0cm and -0.2cm of t3] {$t_3$};
    \path let \p1=(t0l),\p2=($(bn6)!0.5!(bn7)$) in node at (\x2,\y1) {$\vphantom{t_0}r_y$};
    \path let \p1=(t0l),\p2=($(bn12)!0.5!(bn13)$) in node at (\x2,\y1) {$\vphantom{t_0r_y}r_x$};
\end{tikzpicture}
        \caption{Important features of the part of a triangulation revealed by the first $i$ steps of the Schnyder peeling process $P_y$. Note that in all figures from here on, we we will often not draw the edges in all triangulated regions, in order to better emphasise particular features of the construction.}
        \label{fig:coinciding_root_edges}
    \end{figure}
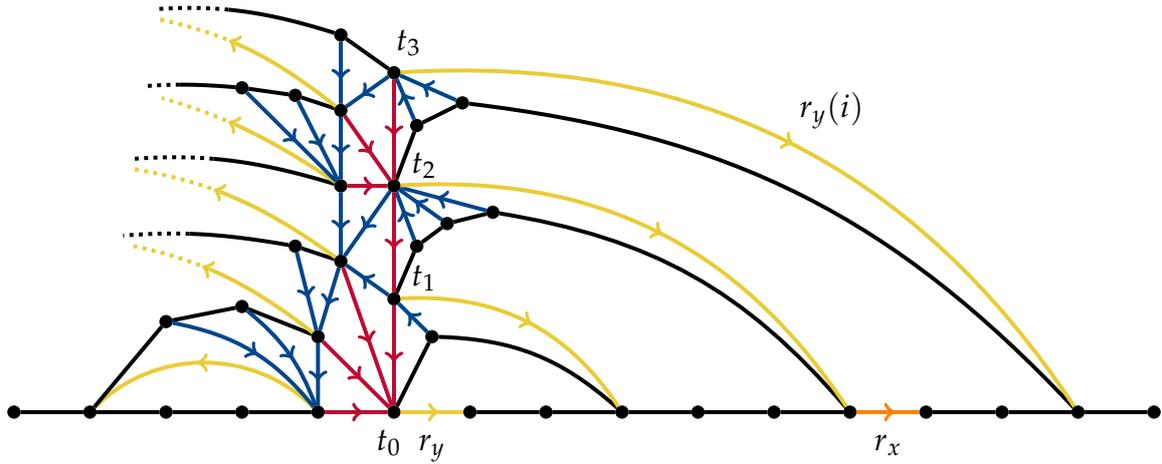
    
    \begin{figure}[h]
        \centering
        \begin{tikzpicture}
    \foreach \x in {1,...,16}{
        \node[mycircle] (bn\x) at (\x,0) {};
    }
    \foreach \x in {1,...,4,7,8,...,11,13,14,15}{
        \pgfmathtruncatemacro{\next}{\x + 1};
        \draw (bn\x) -- (bn\next);
    }

    \node[mycircle] (t1) at (6,1.5) {};
    \node[mycircle] (t2) at (6,3) {};
    \node[mycircle] (t3) at (6,4.5) {};

    \fill[fill=gray,opacity=0.1] (t3) to[bend left=30] (bn15) to (bn15.center) to (bn6.center) to (t3.center) to cycle;
    
    \draw (bn5) edge[->-,redge] (bn6)
          (bn6) edge[->-,yedge] (bn7)
          (bn12) edge[->-,color=orange] (bn13);

    \node[mycircle] (t3u1) at (6.3,3.8) {};
    \node[mycircle] (t3u2) at (6.9,4.1) {};
    \draw (t2) to (t3u1) to (t3u2) to[bend left=20] (bn15)
          (t3u1) edge[->-,bedge] (t3)
          (t3u2) edge[->-,bedge] (t3);

    \draw (t1) edge[->-,redge] (bn6);
    \draw (t2) edge[->-,redge] (t1);
    \draw (t3) edge[->-,redge] (t2)
          (t3) edge[->-,yedge,bend left=30,"$\roote{y}{i}=\roote{x}{j}$" black] (bn15);

    \node[mycircle] (t0l1) at (5,1) {};
    \node[mycircle] (t0l2) at (5.3,2) {};

    \draw 
          (t0l1) edge[->-,bedge] (bn5);
    \draw 
          (t0l1) edge[->-,redge] (bn6);
    \draw 
          (t0l2) edge[->-,bedge] (t0l1);
    \node[mycircle] (t0l2u1) at (5.3,3) {};
    \node[mycircle] (t2l0l) at (2,2.3) {};
    \draw (t0l2) edge[->-,yedge,bend right=50,looseness=1.2] (bn1)
          (t0l2) edge[->-,redge] (bn6);
    \draw (t0l2u1) edge[bend right=25,dots] (t2l0l) (t2l0l) edge[bend right=20,dots] (bn1)
          (t0l2u1) edge[->-,bedge] (t0l2)
          (t2) edge[->-,bedge] (t0l2)
          (t1) edge[->-,bedge] (t0l2);

    \node[mycircle] (t2l1) at (5.3,4) {};
    
    \node[mycircle] (t2l0u1) at (4.7,4.2) {};
    \node[mycircle] (t2l0u2) at (4,4.3) {};
    \draw (t0l2u1) edge[->-,yedge,bend right=50,looseness=1.2] (t2l0l)
          (t0l2u1) edge[->-,redge] (t2);
    \node[mycircle] (t2l1l) at (2.5,3.9) {};
    \draw (t2l0u2) edge[bend right=20,dots] (t2l1l) (t2l1l) edge[bend right=20,dots] (t2l0l) (t2l0u2) to (t2l0u1) to (t2l1)
          (t2l0u2) edge[->-,bedge] (t0l2u1)
          (t2l0u1) edge[->-,bedge] (t0l2u1)
          (t2l1) edge[->-,bedge] (t0l2u1);
    \node[mycircle] (t2l1u1) at (5.3,5) {};
    \draw (t2l1) edge[->-,yedge,bend right=60,looseness=1.2] (t2l1l)
          (t2l1) edge[->-,redge] (t2);
    \draw (t2l1u1) edge[bend right=45,looseness=1.2,dots] (t2l1l) (t2l1u1) to (t3)
          (t2l1u1) edge[->-,bedge] (t2l1)
          (t3) edge[->-,bedge] (t2l1);

    \node[below left=0cm and -0.325cm of bn6] (t0l) {$\vphantom{r_yt_0}$};
    \node[above right=0cm and -0.3cm of t3] {$\peelv{x}{j}$};
    \path let \p1=(t0l),\p2=($(bn6)!0.5!(bn7)$) in node at (\x2,\y1) {$\vphantom{t_0}r_y$};
    \path let \p1=(t0l),\p2=($(bn12)!0.5!(bn13)$) in node at (\x2,\y1) {$\vphantom{t_0r_y}r_x$};

    \node at (10,1.5) {$W$};
    \draw (t3u1) edge[->-,color=orange,bend left=20,"$\roote{x}{j-1}$"'{black,pos=0.35}] (bn14);
    \node at (3.5,1.2) {$\vdots$};
\end{tikzpicture}
        \caption{The area $\area{y}{i}$ and its subarea $W$, with an example scenario for the first step $j$ of $P_x$ at which the peeling vertex $\peelv{x}{j}$ is not contained in $\intr \area{y}{i}$.}
        \label{fig:area_W}
    \end{figure}

    \begin{claim}
        \label{claim:steprootarea}
        Let step $j$ be the first step of $P_x$ at which $\peelv{x}{j} \notin \intr \area{y}{i}$. Then for all $k < j$ we have that $\roote{x}{k}$ is contained in $W$, and that $\area{x}{k} \subseteq \area{y}{i}$.
    \end{claim}

    \begin{poc}
        We prove the claim by induction on $k$. For $k = 0$, the claim is obvious. Now suppose that it holds for some $k < j-1$, so the peeling vertex $\peelv{x}{k+1}$ is contained in $\intr \area{y}{i}$, the root edge $\roote{x}{k}$ is contained in $W$, and $\area{x}{k} \subseteq \area{y}{i}$.
        
        If step $k+1$ of $P_x$ is a left step, the root edge remains unchanged and so $\roote{x}{k+1}$ is contained in $W$. Otherwise, step $k+1$ of $P_x$ is a right step. Observe that the head of $\roote{x}{k+1}$ is some vertex $b_z$ on the initial boundary satisfying $z \geq x+1$. Moreover, since the tail of $\roote{x}{k+1}$, which is $\peelv{x}{k+1}$, is contained in $\intr \area{y}{i}$, planarity implies that $b_z$ is contained in $\area{y}{i}$. Since $z > x$ and $W$ contains the segment of the initial boundary in $\area{y}{i}$ to the right of $b_x$, it follows that $b_z$ is contained in $W-b_x$. Using again the fact that the tail of $\roote{x}{k+1}$ is in $\intr \area{y}{i}$, planarity implies that it must actually be contained in $W$ and so $\roote{x}{k+1}$ is contained in $W$.
        
        In both cases, applying \cref{lem:area_subset} shows that $\area{x}{k+1} \subseteq \area{y}{i}$. Hence, the claim also holds for $k+1$.
    \end{poc}

    Next, we argue that the peeling vertex $\peelv{x}{j}$ at step $j$ of $P_x$ must be $t_\ell$. This is due to structure of the area $\area{y}{i}$ that $P_y$ revealed in steps $1$ to $i$.

    \begin{claim}\label{claim:nbrs_interior}
        The only neighbours of $t_0, \dots, t_{\ell-1}$ which are not in $\intr \area{y}{i}$ are $t_\ell$ and potentially the head of $\roote{y}{i}$.
    \end{claim}
    
    \begin{poc}
        Consider $t_m$ for some $m < \ell$. Since $t_m$ is no longer the tail of the root edge of $P_y$ after step $i$, there must exist some right step $k$ of $P_y$ with $k \leq i$ such that $t_m$ is the tail of the root edge $\roote{y}{k-1}$. Observe that the neighbours of $t_m$ are either in $\intr \area{y}{k-1}$, in the interior of the part of $T$ enclosed by $\chord{y}{k}$ and $\boundary{y}{k-1}$, or they are the endpoints of $\chord{y}{k}$. In the first two cases, the neighbours will be in $\intr \area{y}{k} \subs \intr \area{y}{i}$. In the last case, note that the endpoints of $\chord{y}{k}$ are $t_{m+1}$ and some vertex $b_z$ on the segment of the initial boundary in $\area{y}{i}$ with $z \geq x+1$. Observe that the only such vertices which are not in $\intr \area{y}{i}$ are $t_\ell$ and the head of $\roote{y}{i}$, which proves the claim.
    \end{poc}

    Consider the tail $v$ of $\roote{x}{j-1}$. By \cref{claim:steprootarea}, $\roote{x}{j-1}$ lies in $W$ and so $v$ lies in $W$. Also, $v$ is the peeling vertex $\peelv{x}{k}$ for some $k < j$ which is in $\intr \area{y}{i}$ by the choice of $j$. In particular, $v$ cannot be $t_\ell$. Thus, the neighbours of $v$ are all either in $W$, or are neighbours of $t_0, \dots, t_{\ell-1}$. Since $\peelv{x}{j}$ is a neighbour of $v$ that is not in $\intr \area{y}{i}$, and clearly $\peelv{x}{j}$ can never be the head of a root edge $\roote{y}{i}$, it follows from \cref{claim:nbrs_interior} that $\peelv{x}{j}$ must be $t_\ell$. 

    The first chord at $t_\ell$ anticlockwise from $(t_\ell, t_{\ell-1})$ to the initial boundary must be $\roote{y}{i}$, since this was the chord chosen by step $i$ of $P_y$. It follows that $\roote{y}{i}$ is also the first chord at $t_\ell$ anticlockwise from $t_\ell v$ to the initial boundary, and hence $\roote{x}{j} = \roote{y}{i}$ as required.
\end{proof}

\subsubsection{Step 3: Segments of boundaries eventually coincide}

We now show that $P_x$ and $P_y$ will eventually agree on some segment of their boundaries left from their common root edge, and that this segment can be guaranteed not to be spanned by short paths. For all $y \leq x$, let $\fs{y}{x}$\glossarylabel{gl:fs} be the first step of $P_y$ where $\area{y}{\fs{y}{x}}$ contains $b_x$. Also, let $\head{z}{i}$\glossarylabel{gl:head} denote the head of $\roote{z}{i}$. Note that $\head{z}{i}$ always lies on both the initial boundary $\boundary{z}{0}$ and the updated boundary $\boundary{z}{i}$.

\begin{lemma}\label{lem:subsumedarea}
    If $z < y < x$, then $\area{z}{\fs{z}{x}} \sups \area{y}{\fs{y}{x}}$.
\end{lemma}

\begin{proof}
    Note that $\area{z}{\fs{z}{x}}$ contains $b_z, b_{z+1}, \dots, b_x$, and so it contains $\inroote{y}$. By \cref{lem:area_containment} it follows that there exists a step $i$ of $P_y$ such that $\roote{y}{i}$ coincides with $\roote{z}{\fs{z}{x}}$ and $\area{y}{i} \subs \area{z}{\fs{z}{x}}$. Since $\roote{y}{i}$ coincides with $\roote{z}{\fs{z}{x}}$, we have that $\area{y}{i}$ contains all boundary vertices between $b_y$ and $\head{z}{\fs{z}{x}}$, and so it must contain $b_x$. In particular, $\fs{y}{x} \le i$ and so $\area{y}{\fs{y}{x}} \subs \area{y}{i} \subs \area{z}{\fs{z}{x}}$.
\end{proof}

If $x$ tends to $\infty$, then both $\fs{z}{x}$ and $\fs{y}{x}$ tend to $\infty$, and so it follows from \cref{lem:subsumedarea} that a given Schnyder wood segment is always contained within any Schnyder wood segment rooted further left on the initial boundary.

\begin{corollary}\label{cor:nested_segments}
    If $z < y$, then $\evarea{z} \sups \evarea{y}$.
\end{corollary}

We next show that processes rooted sufficiently far left will eventually share a common boundary segment. For any $y \leq x$, let $\segment{y}{x}{k}$\glossarylabel{gl:segment} be the segment of the first $k$ edges left of $\head{y}{\fs{y}{x}}$ along $\boundary{y}{\fs{y}{x}}$.

\begin{lemma}\label{lem:common_bdry}
    For all $x \in \Z$, $k > 0$, and $\eps > 0$, there exist $z \leq y \leq x$ such that
    \[
        \pr\left(\head{w}{\fs{w}{y}} \leq b_x \text{ and } \segment{w}{y}{k} = \segment{z}{y}{k} \text{ for all } w \leq z\right) \geq 1 - \eps.
    \]
\end{lemma}

\begin{proof}
    Applying \cref{lem:no_short_paths} with $d = 3$ and $\eps/2$, there exists $y \le x$ such that
    \[
        \pr(\text{some path of length at most } 3 \text{ spans } b_y, b_{y+1}, \dots, b_x \subs B) < \frac{\eps}{2}.
    \]
    For now, suppose that no path of length at most $3$ spans $b_y,\dots,b_x$. We claim that $\head{w}{\fs{w}{y}} \leq b_x$ for all $w < y$. Indeed, by definition of $\fs{w}{y}$ we know that $\head{w}{\fs{w}{y}-1} \leq b_y$ while $\head{w}{\fs{w}{y}} \geq b_y$. Moreover, for all $i$ we have that $\head{w}{i}$ and $\head{w}{i+1}$ are connected by a path of length at most $3$, so this also holds for $\head{w}{\fs{w}{y}-1}$ and $\head{w}{\fs{w}{y}}$. Given that no path of length at most $3$ spans $b_y, b_{y+1}, \dots, b_x$, this implies that $\head{w}{\fs{w}{y}} \leq b_x$.

    \Cref{lem:subsumedarea} implies that $\area{v}{\fs{v}{y}} \sups \area{w}{\fs{w}{y}}$ for all $v < w < y$, so we have an infinite chain
    \[
        \area{y-1}{\fs{y-1}{y}} \subseteq \area{y-2}{\fs{y-2}{y}} \subseteq \area{y-3}{\fs{y-3}{y}} \subseteq \dots
    \]
    of nested areas. Since the ball $\ball{x-y+k}{b_x}$ is almost surely finite, this implies that regions $\area{w}{\fs{w}{y}} \cap \ball{x-y+k}{b_x}$ must eventually converge to a common region as $w \to -\infty$. Formally, this means that almost surely there is some $z^\ast < y$ such that $\area{w}{\fs{w}{y}} \cap \ball{x-y+k}{b_x} = \area{z^\ast}{\fs{z^\ast}{y}} \cap \ball{x-y+k}{b_x}$ for all $w \leq z^\ast$, and therefore also $\boundary{w}{\fs{w}{y}} \cap \ball{x-y+k}{b_x} = \boundary{z^\ast}{\fs{z^\ast}{y}} \cap \ball{x-y+k}{b_x}$ for all $w \leq z^\ast$. Since the segment $\segment{w}{y}{k}$ is contained in the ball $\ball{x-y+k}{b_x}$ and consists of the first $k$ edges clockwise along the boundary $\boundary{w}{\fs{w}{y}}$ that are not on the initial boundary, this implies that $\segment{w}{y}{k} = \segment{z^\ast}{y}{k}$ for all $w \leq z^\ast$.

    For all $z < y$, let $E_z$ be the event that $\head{w}{\fs{w}{y}} \leq b_x$ and $\segment{w}{y}{k} = \segment{z}{y}{k}$ for all $w \leq z$. As shown above, if no path of length $3$ spans $b_y, b_{y+1}, \dots, b_x \subs B$, then one of these events $E_z$ occurs, and so $\pr(\bigcup_{z < y} E_z) \geq 1 - \eps/2$. Since $E_z \subs E_{z-1}$, it follows that
    \[
        \lim_{z \to -\infty} \pr(E_z) = \pr\left(\bigcup_{z < y} E_z\right) \geq 1 - \frac{\eps}{2},
    \]
    and so there exists some $z < y$ such that $\pr(E_z) \geq 1 - \eps$. In particular,
    \[
        \pr\left(\head{z}{\fs{z}{y}} \leq b_x \text{ and } \segment{w}{y}{k} = \segment{z}{y}{k} \text{ for all } w \leq z\right) \geq \pr(E_z) \ge 1 - \eps. \qedhere
    \]
\end{proof}

As a consequence of the preceding lemma and the fact that no short paths span the boundary of the UIHPT, we show that there will be no short paths spanning the common boundary segment of the Schnyder peeling processes.

\begin{lemma}\label{lem:no_short_paths_spanning_common_bdry}
    For all $x \in \Z$, $d > 0$, and $\eps > 0$, there exist $k \in \N$ and $z \leq y \leq x$ satisfying the following. Let $E_w$ be the event that $\head{w}{\fs{w}{y}} \leq b_x$, $\segment{w}{y}{k} = \segment{z}{y}{k}$, and there is no path of length at most $d$ spanning $\segment{w}{y}{k} \subs \boundary{w}{\fs{w}{y}}$ in $T \sm \intr \area{w}{\fs{w}{y}}$. Then,
    \[
        \pr\left(\bigcap_{w \leq z} E_w\right) \geq 1 - \eps.
    \]
\end{lemma}

\begin{proof}
    First, apply \cref{lem:no_short_paths} with $d$ and $\eps/2$ to obtain $n$, and let $k = n$. Then, apply \cref{lem:common_bdry} with $x$, $k$, and $\eps/2$ to obtain $z \leq y \leq x$ such that
    \[
        \pr\left(\head{w}{\fs{w}{y}} \leq b_x \text{ and } \segment{w}{y}{k} = \segment{z}{y}{k} \text{ for all } w \leq z\right) \geq 1 - \frac{\eps}{2},
    \]
    and denote the event in the above expression by $E$.
    
    It follows from \cref{lem:hole_uihpt} that $T \sm \intr \area{w}{\fs{w}{y}}$ is a UIHPT. In particular, by \cref{lem:no_short_paths} we get that, for all $w \leq z$,
    \[
        \pr\left(\text{some path of length at most } d \text{ spans } \segment{w}{y}{k} \subs \boundary{w}{\fs{w}{y}}\right) \leq \frac{\eps}{2}.
    \]
    It follows that $\pr(E \cap E_w) \geq 1 - \eps$. By definition $\bigcap_{w \leq z} E_w \subs E$. Moreover, by \cref{lem:area_containment} we have that $\area{w-1}{\fs{w-1}{y}} \sups \area{w}{\fs{w}{y}}$ and so when $E$ occurs, every path that spans $\segment{z}{y}{k} = \segment{w}{y}{k}$ as a segment of the boundary $\boundary{w}{\fs{w}{y}}$ must also span $\segment{z}{y}{k} = \segment{w-1}{y}{k}$ as a segment of the boundary $\boundary{w-1}{\fs{w-1}{y}}$. Hence, $E \cap E_{w-1} \subs E \cap E_w$. This implies that
    \[
        \pr\left(\bigcap_{w \leq z} E_w\right) = \pr\left(E \cap \bigcap_{w \le z} E_w\right) = \lim_{w \to -\infty} \pr(E \cap E_w) \geq 1-\eps. \qedhere
    \]
\end{proof}

\subsubsection{Step 4: Existence of the Schnyder wood strip}

We are now prepared to show that we can define a consistent colouring on a two-ended infinite strip along the initial boundary. We say $P_x$ produces a \defn{consistent} colouring on the ball $\ball{\rho}{v}$ if for all $y < x$, the two Schnyder peeling processes $P_x$ and $P_y$ produce the same colouring and orientation on $\ball{\rho}{v}$. Consider the vertex $b_0$ on the initial boundary, and let $C_{x,\rho}$ be the event that $P_x$ produces a consistent colouring on $\ball{\rho}{b_0}$. Note that $C_{x,\rho+1} \subs C_{x,\rho} \subs C_{x-1,\rho}$. Using the results of steps 1 to 3 in this section, we show that some Schnyder peeling process will produce a consistent colouring. An illustration of the overall proof strategy is provided in \cref{fig:proof_components}.

\begin{theorem}\label{thm:consistent_strip_colouring}
    For all $\rho \in \N$,
    \[
        \lim_{z \to -\infty} \pr(C_{z,\rho}) = 1. 
    \]
\end{theorem}

\begin{proof}
    Let $\rho \in \mathbb{N}$ and $\eps > 0$. First, choose $n \in \N$ according to \cref{lem:no_short_paths} such that
    \begin{equation}\label{eq:prob_no_short_path}
        \pr(\text{no path of length at most } \rho + 3 \text{ spans } b_{-n}, b_{-n+1}, \dots, b_0 \subs B) \ge 1 - \frac{\eps}{2}.
    \end{equation}
    Then, apply \cref{lem:no_short_paths_spanning_common_bdry} with $x = -n$, $d = 3$, and $\eps/2$ to obtain $k \in \Z$ and $z \leq y \leq x$ satisfying the following. Let $E_w$ be the event that $\head{w}{\fs{w}{y}} \leq b_x$, $\segment{w}{y}{k} = \segment{z}{y}{k}$, and there is no path of length at most $3$ spanning $\segment{w}{y}{k} \subs \boundary{w}{\fs{w}{y}}$. Then,
    \begin{equation}\label{eq:prob_no_short_path_plus}
        \pr\left(\bigcap_{w \leq z} E_w\right) \geq 1 - \frac{\eps}{2}.
    \end{equation}
    Let $F$ be the intersection of the events in \eqref{eq:prob_no_short_path} and \eqref{eq:prob_no_short_path_plus}. Thus, $\pr(F) \geq 1 - \eps$.
    \begin{claim}\label{clm:startirrelevant}
        If $F$ occurs then $\area{w}{\fs{w}{y}} \cap \ball{\rho}{b_0} = \emptyset$ for all $w \leq z$.
    \end{claim}

    \begin{poc}
        Suppose for a contradiction that $\area{w}{\fs{w}{y}} \cap \ball{\rho}{b_0} \neq \emptyset$, so there is a path of length at most $\rho$ from $b_0$ to the boundary of $\area{w}{\fs{w}{y}}$. Note that all vertices on that boundary have a path of length at most $3$ to a vertex on the segment $\area{w}{\fs{w}{y}} \cap B$ of the initial boundary. Since that segment is left of $\head{w}{\fs{w}{y}} \leq b_x$, it follows that there is a path of length at most $\rho + 3$ spanning $b_x, b_{x+1}, \dots, b_0$, giving a contradiction.
    \end{poc}

    Let $v$ be the leftmost vertex on $\segment{z}{y}{k} \subs \boundary{z}{\fs{z}{y}}$.

    \begin{claim}\label{clm:noleftcover}
        If $F$ occurs, then for all $w \leq z$ and for all $i \geq \fs{w}{y}$, we have that $\boundary{w}{i}$ contains $v$ and is identical to $\boundary{w}{\fs{w}{y}}$ left of $v$.
    \end{claim}

    \begin{poc}
        We prove the claim by induction on $i$. Since $F$ occurs, $E_w$ occurs, and so the claim holds for $i = \fs{w}{y}$. Now suppose that it holds for some $i \geq \fs{w}{y}$.
        
        Let $u$ be the head of $\chord{w}{i+1}$. Note that by the definition of the Schnyder peeling process, there is a path of length at most $3$ from $u$ to the head $\head{w}{i}$ of the root edge $\roote{w}{i}$. We know that $\head{w}{i}$ lies on the right of $\segment{z}{y}{k}$ on $\boundary{w}{\fs{w}{y}}$. Since there is no path of length $3$ spanning $\segment{z}{y}{k} \subs \boundary{w}{\fs{w}{y}}$, it follows that $u$ cannot lie left of $v$ on $\boundary{w}{\fs{w}{y}}$. Since $\boundary{w}{i}$ coincides with $\boundary{w}{\fs{w}{y}}$ left of $v$, it follows that $u$ lies right of $v$ on $\boundary{w}{i}$. In particular, step $i+1$ of $P_w$ will not explore anything left of $v$ on $\boundary{w}{i}$, and so the claim holds for $i+1$.
    \end{poc}

    \begin{claim}\label{clm:eventuallyconsistent}
        If $F$ occurs, then for all $w \le z$, and for all $\ell \ge 0$, the boundaries $\boundary{w}{\fs{w}{y}+\ell}$ and $\boundary{z}{\fs{z}{y}+\ell}$ only differ left of $v$. As a result, for all $\ell > 0$, steps $\fs{w}{y} + \ell$ of $P_w$ and $\fs{z}{y} + \ell$ of $P_z$ are the same.
    \end{claim}

    \begin{poc}
        Since $F$ occurs, we have that $\segment{w}{y}{k} = \segment{z}{y}{k}$, so the boundaries $\boundary{w}{\fs{w}{y}}$ and $\boundary{z}{\fs{z}{y}}$ only differ left of $v$ which shows that the claim holds for $\ell = 0$. Now suppose that it holds for some $\ell \geq 0$.
        
        Observe that the root edge $\roote{z}{\fs{z}{y}+\ell}$ is the rightmost edge of $\boundary{z}{\fs{z}{y}+\ell}$ not on the initial boundary, and similarly for $w$. Since $\boundary{z}{\fs{z}{y}+\ell}$ and $\boundary{w}{\fs{w}{y}+\ell}$ only differ left of $v$, it follows that $\roote{z}{\fs{z}{y}+\ell} = \roote{w}{\fs{w}{y}+\ell}$, and $\peelv{z}{\fs{z}{y}+\ell} = \peelv{w}{\fs{w}{y}+\ell}$.

        We claim that $\chord{z}{\fs{z}{y}+\ell} = \chord{w}{\fs{w}{y}+\ell}$. Indeed, by \cref{clm:noleftcover}, step $\fs{z}{y} + \ell$ does not explore any part of $\boundary{z}{\fs{z}{y}+\ell}$ left of $v$. Thus, $\chord{z}{\fs{z}{y}+\ell}$ is a chord of $\boundary{z}{\fs{z}{y}+\ell}$ with both endpoints right of $v$. Since $\boundary{w}{\fs{w}{y}+\ell}$ coincides with $\boundary{z}{\fs{z}{y}+\ell}$ right of $v$, it follows that $\chord{z}{\fs{z}{y}+\ell}$ is also a chord of $\boundary{w}{\fs{w}{y}+\ell}$. Similarly, $\chord{w}{\fs{w}{y}+\ell}$ is a chord of $\boundary{z}{\fs{z}{y}+\ell}$. Thus, since the Schnyder peeling process always picks the first chord anticlockwise from the boundary, it follows that $\chord{z}{\fs{z}{y}+\ell} = \chord{w}{\fs{w}{y}+\ell}$, as claimed.

        As a result, step $\fs{z}{y} + \ell$ of $P_z$ and step $\fs{w}{y} + \ell$ of $P_w$ are the same. In particular, the boundaries $\boundary{z}{\fs{z}{y}+\ell+1}$ and $\boundary{w}{\fs{w}{y}+\ell+1}$ will be the identical right of $v$, and so the claim holds for $\ell+1$.
    \end{poc}

    If $F$ occurs, then by \cref{clm:eventuallyconsistent}, for all $w \leq z$, the Schnyder peeling processes $P_z$ and $P_w$ produce the same colouring starting from step $\fs{z}{y}$ and from step $\fs{w}{y}$ respectively. Moreover, by \cref{clm:startirrelevant}, they only colour edges outside of $\ball{\rho}{b_0}$ before these steps. Thus, $P_z$ and $P_w$ produce the same colouring on $\ball{\rho}{b_0}$. Since $w$ was arbitrary, it follows that when $F$ occurs, $P_z$ produces a consistent colouring on $\ball{\rho}{b_0}$, and therefore
    \[
        \pr(C_{z,\rho}) \geq 1 - \eps. \qedhere
    \]
\end{proof}

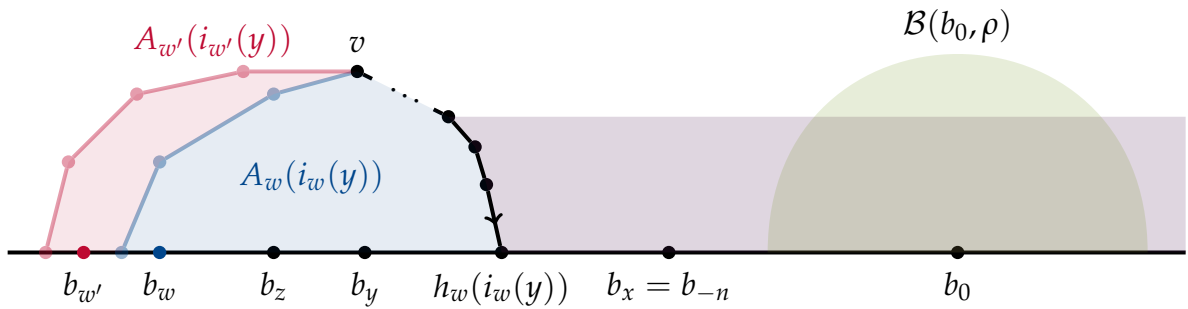
\begin{figure}[h]
    \centering
    \begin{tikzpicture}
    \coordinate (left) at (-7.5,0){};
    \coordinate (right) at (8,0);
    \node[] (rightup) at (8,1.8){};
    \path[every node/.style={font=\sffamily\small}]
        (left) edge (right);

    \node[mycircle, color=blue, opacity=0.4] (w) at (-6,0){}; 
    \node[mycircle, label=below:$b_w$, color=blue] (wlbl) at (-5.5,0){};
    \node[mycircle, color=red, opacity=0.4] (wp) at (-7,0){};
    \node[mycircle, label=below:$b_{w'}$, color=red] (wplbl) at (-6.5,0){};
    \node[mycircle, opacity=0.4, color=blue] (w1) at (-5.5,1.2){};
    \node[mycircle, opacity=0.4, color=blue] (w2) at (-4,2.1){};
    \node[mycircle, opacity=0.4, color=red] (wp1) at (-6.7,1.2){};
    \node[mycircle, opacity=0.4, color=red] (wp2) at (-5.8,2.1){};
    \node[mycircle, opacity=0.4, color=red] (wp3) at (-4.4,2.4){};
    \node[mycircle, label=below:$b_z$] (z) at (-4,0){};
    \node[mycircle, label=below:$b_y$] (y) at (-2.8,0){};
    \node[mycircle, label=below:$h_w(i_w(y))$] (h) at (-1,0){};
    \node[mycircle] (v1) at (-1.2,0.9){};
    \node[mycircle] (v2) at (-1.35,1.4){};
    \node[mycircle] (v3) at (-1.7,1.8){};
    \node[mycircle, label=above:$v$] (v4) at (-2.9,2.4){};
    \node[mycircle, label={below:$b_x=b_{-n}$}] (x) at (1.2,0){};
    \node (nr) at (2.5,0){};
    \node[mycircle, label=below:$b_0$] (0) at (5,0){};
    \node (pr) at (7.5,0){};
    \draw[dots] (v3) to (v4);
    
    \fill[fill=blue, opacity=0.1]  (w.center) to (w1.center) to (w2.center) to (v4.center) to (v3.center) to (v2.center) to (v1.center) to (h.center) to cycle;
    
    \fill[fill=red, opacity=0.1]  (wp.center) to (wp1.center) to (wp2.center) to (wp3.center) to (v4.center) to (w2.center) to (w1.center) to (w.center) to cycle;

    \fill[fill=red, opacity=0.1]  (v3.center) to (rightup.center) to (right.center) to (h.center) to (v1.center) to (v2.center) to cycle;
    \fill[fill=blue, opacity=0.1]  (v3.center) to (rightup.center) to (right.center) to (h.center) to (v1.center) to (v2.center) to cycle;

    \fill[fill=green, opacity=0.15]  (nr.center) to[bend left=90, looseness=1.8] (pr.center) to cycle;

    \node (lbl) at (5,3) {$\ball{\rho}{b_0}$};

    \node[blue] (Aw) at (-3.5,1) {$\area{w}{\fs{w}{y}}$};
    \node[red] (Aw) at (-4.8,2.8) {$\area{w'}{\fs{w'}{y}}$};
    
    \path[every node/.style={font=\sffamily\small}]
        (v1) edge[->-] (h)
        (v1) edge (v2)
        (v2) edge (v3)
        (w) edge[opacity=0.4, color=blue] (w1)
        (w1) edge[opacity=0.4, color=blue] (w2)
        (w2) edge[opacity=0.4, color=blue] (v4)
        (wp) edge[opacity=0.4, color=red] (wp1)
        (wp1) edge[opacity=0.4, color=red] (wp2)
        (wp2) edge[opacity=0.4, color=red] (wp3)
        (wp3) edge[opacity=0.4, color=red] (v4);

\end{tikzpicture}
    \caption{Processes $P_w$ and $P_{w'}$ started from $b_w,b_{w'} < b_z$ may colour the UIHPT differently prior to the steps at which a $k$ edge long segment of their boundaries agree between $\head{w}{\fs{w}{y}}$ and $v$. After this step, all further edges are coloured consistently, and in particular the processes agree on every explored edge within $\ball{\rho}{b_0}$.}
    \label{fig:proof_components}
\end{figure}

Consider now the event $C_\rho = \bigcup_{x \in \Z} C_{x,\rho}$ that some Schnyder process $P_x$ produces a consistent colouring on $\ball{\rho}{b_0}$, and let $C = \bigcap_{\rho \in \N} C_\rho$ be the event that this happens for all $\rho$. Since $C_{x,\rho} \subs C_{x-1,\rho}$, it follows from \cref{thm:consistent_strip_colouring} that
\[
    \pr(C_\rho) = \lim_{x \to -\infty} \pr(C_{x,\rho}) = 1.
\]
Because $C_{x,\rho+1} \subs C_{x,\rho}$, we have $C_{\rho+1} \subs C_\rho$, and so
\[
    \pr(C) = \lim_{\rho \to \infty} \pr(C_\rho) = 1.
\]
Thus, almost surely the event $C$ occurs. If that is the case, then for any $\rho \in \N$ there exists some $x(\rho) \in \Z$ such that $P_{x(\rho)}$ produces a consistent colouring on $\ball{\rho}{b_0}$.

We now formally define the Schnyder wood strip. Let $P_x(e)$\glossarylabel{gl:pxe} denote the colour and orientation assigned to edge $e$ by $P_x$. If $\rho$ is larger than the distance from $e$ to $b_0$, then by definition of $x(\rho)$ we know that $P_x(e)$ is the same for all $x \le x(\rho)$, and so $\lim_{x\to-\infty} P_x(e)$ exists. We define the \defn{Schnyder wood strip} $S$ of $T$ to be the coloured and oriented region of $T$ where each edge $e$ is assigned the colouring and orientation given by $\lim_{x\to-\infty} P_x(e)$. Note that it follows from the nested structure of the Schnyder wood segments described in \cref{cor:nested_segments} that $S = \bigcup_x \evarea{x}$. 

\subsubsection{Removing the Schnyder wood strip}\label{sssec:uihpt:schnyderstripremoved}

We now show that after removing the Schnyder wood strip from the UIHPT, the remaining region is again distributed like a UIHPT. Let $S$ denote the Schnyder wood strip, and define its interior as $\intr S \coloneqq \bigcup_x \intr \evarea{x}$. For an edge $e = (a,b)$, let the distance from $e$ to some vertex $v$ be defined by $d(e, v) \coloneqq \min (d(a,v), d(b,v))$.  

\begin{lemma}\label{lem:plane_minus_stripUIHPT}
    Let $e_u$ denote the leftmost edge on the boundary of $T \sm \intr \evarea{u}$ of minimum distance from $b_0$. Then almost surely $\lim_{u \to -\infty} e_u = e$ for some edge $e$ in $T$, and $T \sm \intr S$ has the same distribution as a UIHPT rooted at $e$.
\end{lemma}

\begin{proof}
    \cref{thm:consistent_strip_colouring} established that for all $\eps > 0$ and $\rho \in \mathbb{N}$, there exists $z \in \mathbb{Z}$ such that $P_z$ produces a consistent colouring on $\ball{\rho}{b_0}$ with probability at least $1- \eps$. Specifically, in \cref{clm:noleftcover} we showed that with probability at least $1 - \eps$, there exists a vertex $v$ so that for all $w \le z$ we have that $v$ lies on the boundary $\boundary{w}{i}$ of $P_w$ for all sufficiently large steps $i$, and so $v$ is contained in $T \sm \intr \evarea{w}$. In this case, the distance $d(b_0, e_w)$ is at most $d \coloneqq d(b_0, v)$. Since the ball $\ball{d}{b_0}$ is almost surely finite and \cref{cor:nested_segments} implies that $\evarea{w} \subs \evarea{w-1}$, it follows that there is some $u < z$ such that $e_w = e_u$ for all $w \le u$. As $z \to -\infty$ (and hence $u \to -\infty$), we have that $\eps \to 0$ by \cref{thm:consistent_strip_colouring}, and so $\lim_{u \to -\infty} e_u \eqqcolon e$ almost surely exists.
    
    Let $T_u \coloneqq T \sm \intr \evarea{u}$ be the unexplored area of $T$ after running the Schnyder peeling process $P_u$, and consider the ball $\ballt{\rho}{e}{T_u}$ of radius $\rho$ around $e$ in $T_u$. Since $\ballt{\rho}{e}{T}$ is almost surely finite and $\evarea{u} \subs \evarea{u-1}$ for all $u$, there almost surely exists some $y$ such that $\ballt{\rho}{e}{T_u} = \ballt{\rho}{e}{T \sm \intr S}$ for all $u \le y$. By choosing $y$ sufficiently small we may also assume that $e_u = e$ for all $u \le y$. In particular, as $u$ decreases, $\ballt{\rho}{e_u}{T_u}$ almost surely converges to $\ballt{\rho}{e}{T \sm \intr S}$.

    For every fixed $u$, \cref{lem:remove_one_ended_strip_still_uihpt} implies that $\ballt{\rho}{e_u}{T_u}$ is distributed like a ball of radius $\rho$ in a UIHPT rooted at $e_u$. Since $\ballt{\rho}{e_u}{T_u}$ almost surely converges to $\ballt{\rho}{e}{T \sm \intr S}$, it follows that $\ballt{\rho}{e}{T \sm \intr S}$ is almost surely distributed like a ball of radius $\rho$ in a UIHPT rooted at $e$. Since this holds for all $\rho$, it follows that $T \sm \intr S$ has the same distribution as a UIHPT rooted at $e$.
\end{proof}

\section{The maximal Schnyder wood of the UIHPT}\label{sec:structure}
 
We are now finally able to show that a unique maximal Schnyder wood exists on the UIHPT. To achieve this, we define an algorithm, called the Schnyder wood chiselling algorithm, that constructs a maximal Schnyder wood of the UIHPT using the Schnyder wood strips and segments we have just defined in the previous section. 

In \cref{ssec:structure:existence} we show that the Schnyder wood chiselling algorithm constructs a Schnyder wood,  which establishes the existence of a Schnyder wood of the UIHPT. While the properties of a Schnyder wood only prescribe a certain local structure, in \cref{ssec:structure:forests} we describe some of the global structure of the red, yellow, and blue subgraphs in the Schnyder wood constructed by the chiselling algorithm. 

Finally, in \cref{ssec:structure:maximal} we use the previous results to prove that the chiselling algorithm produces the unique maximal Schnyder wood of the UIHPT. These results for the UIHPT provide an analogue of the results for the peeling algorithm on finite triangulations given in \cref{thm:unique_maximal_wood}, and on infinite triangulations with finite boundary given in \cref{thm:unique_maximal_schnyder_wood_finite_bdry}.

\begin{theorem}\label{thm:peeling_process_uihpt_maximal_wood}
    Almost surely, the maximal Schnyder wood of the UIHPT is unique and is constructed by the Schnyder wood chiselling algorithm.
\end{theorem}

\subsection{Existence of the Schnyder wood of the UIHPT}\label{ssec:structure:existence}

In \cref{sec:infiniteschnyderwoods} we defined Schnyder woods of infinite triangulations.
To show that such an object exists when the underlying triangulation is the UIHPT, we show that it is possible to construct one using the Schnyder wood segments and Schnyder wood strips defined in \cref{ssec:uihpt:segment,ssec:uihpt:strip} respectively.

To motivate our construction, consider for now just the Schnyder root condition. This root condition will always be satisfied by the root edge of a Schnyder segment coloured by the Schnyder peeling process initiated at that root edge. In contrast, the Schnyder strip does not have a distinguished edge on its boundary satisfying the root condition as, informally, it is defined by letting the root edge tend to $-\infty$. This suggests the following \defn{Schnyder wood chiselling algorithm} on the UIHPT rooted at $(b_x, b_{x+1})$. 

First, run the Schnyder peeling process $P_x$ on the UIHPT to colour a Schnyder segment with a distinguished root edge $(b_x, b_{x+1})$. We know that almost surely, this colours only a one-ended infinite region to the right along the initial boundary, and by \cref{lem:remove_one_ended_strip_still_uihpt}, the uncoloured region is still distributed like a UIHPT.  Now, the remaining uncoloured region of the UIHPT should no longer have a distinguished edge satisfying the root condition. Hence, we colour the remainder of the plane by iteratively colouring and removing Schnyder strips, a process that we can repeat by \cref{lem:plane_minus_stripUIHPT}.

We now work towards establishing that this algorithm constructs a Schnyder wood.

\begin{theorem}\label{thm:uihpt_schnyder_exists}
    The Schnyder wood chiselling algorithm almost surely constructs a Schnyder wood of the UIHPT.
\end{theorem}

 It suffices to show that the colouring on the UIHPT defined by the Schnyder wood chiselling algorithm satisfies the Schnyder condition, Schnyder root condition, and Schnyder boundary condition. We consider first the colouring assigned to an individual Schnyder wood strip or Schnyder wood segment, and, second, how pairs of these coloured regions join together along their shared boundaries.

Consider a Schnyder wood segment (resp. strip) obtained from a UIHPT by exploring a one-ended (resp. two-ended) infinite region along its initial boundary. Define the \defn{upper boundary} of the segment (resp. strip) as the set of vertices and edges on its boundary that are not on the initial boundary, and let the leftmost explored vertex on the initial boundary of the segment be the \defn{corner vertex}. Finally, let the \defn{lower boundary} of the Schnyder wood segment (resp. strip) be the set of vertices and edges bordering the infinite explored region that are also on the initial boundary of the UIHPT. Note that when the Schnyder wood chiselling algorithm iteratively explores Schnyder wood strips of a UIHPT $T$, the vertices on the upper boundary of a Schnyder wood strip or segment are interior vertices of $T$ and should therefore satisfy the Schnyder condition in the final colouring. Since a subset of the edges incident to these vertices will be coloured according to the condition on the lower boundary of the Schnyder wood strip that is explored next (that is, according to the Schnyder boundary condition), it is straightforward to check that satisfying the Schnyder condition at these vertices is implied by the following condition.

\textbf{The Schnyder upper boundary condition:}
For any upper boundary vertex $v$,
\begin{itemize}
    \item There is exactly one outgoing edge at $v$ which is coloured blue (which is directed to the interior of the Schnyder wood strip or segment), and
    \item In anticlockwise order starting from the upper boundary, edges incident with $v$ consist of incoming red edges, the outgoing blue edge, and incoming yellow edges.
\end{itemize}

Furthermore, the corner vertex of the Schnyder wood segment is a boundary vertex of both the Schnyder wood and the Schnyder wood strip that the Schnyder wood chiselling algorithm explores immediately after exploring the segment. It is straightforward to check that we therefore require the corner vertex to satisfy the following condition within the Schnyder wood segment.

\textbf{The Schnyder corner condition:}
\begin{itemize}
    \item Every edge that is incident with the corner vertex is an incoming yellow edge.
\end{itemize}

We now show that the Schnyder wood segment satisfies all necessary conditions.

\begin{lemma}\label{lem:schnyder_segment_satisfies_conditions}
    The Schnyder wood segment almost surely satisfies the Schnyder, Schnyder root, Schnyder boundary, Schnyder upper boundary, and Schnyder corner conditions.
\end{lemma}

\begin{proof}
    The proof that the colouring produced by the Schnyder peeling process satisfies the Schnyder, Schnyder root, and Schnyder boundary conditions is almost identical to the proof of \cref{lem:infinite_triangulation_noanticlockwise}: any finite collection of edges of $T$ explored by the Schnyder peeling process is explored after a finite number of steps, and since the same colouring can be produced by the finite peeling process on some finite triangulation, the Schnyder, Schnyder root, and Schnyder boundary conditions follow from \cref{thm:peeling_process_maximal_wood}.

    It remains to check the Schnyder upper boundary and Schnyder corner conditions. Recall that we have already shown in \cref{obs:finitestepsupperboundary} that these two conditions are satisfied by all upper boundary vertices and the corner vertex of the revealed area after any finite number of steps of the Schnyder peeling process. Since an upper boundary or corner vertex of the Schnyder wood segment is, respectively, an upper boundary or corner vertex of the revealed area after a sufficiently large finite number of steps, this implies that the Schnyder wood segment also satisfies the Schnyder upper boundary and Schnyder corner conditions.
\end{proof}

We now use the fact that Schnyder wood segments satisfy all necessary conditions to show that Schnyder wood strips also satisfy all necessary conditions. It will then follow from the definition of the Schnyder upper boundary condition that the Schnyder wood chiselling algorithm constructs a Schnyder wood of the UIHPT.

\begin{proof}[Proof of \cref{thm:uihpt_schnyder_exists}]
     First, recall that $P_x(e)$ denotes the colour and orientation assigned to edge $e$ by $P_x$, and that the Schnyder wood strip of some UIHPT $T$ is defined to be the coloured and oriented region of $T$ whose edges are assigned the colouring and orientation given by $\lim_{x \to -\infty}P_x(e)$. In \cref{ssec:uihpt:strip}, we proved that this limit almost surely exists. Since $T$ is almost surely locally finite, it follows that for every vertex $v$ in the Schnyder wood strip, there almost surely exists some $x$ such that every $P_y$ with $y < x$ assigns the same colouring and orientation to every edge incident with $v$. Hence, checking the relevant Schnyder conditions at any given vertex $v$ in the Schnyder wood strip amounts to confirming the Schnyder conditions at $v$ in a Schnyder wood segment rooted sufficiently far left on the initial boundary. It follows by \cref{lem:schnyder_segment_satisfies_conditions} that the Schnyder wood strip almost surely satisfies the Schnyder, Schnyder root, Schnyder boundary, and Schnyder upper boundary conditions. (Note that the Schnyder corner condition is not relevant to Schnyder wood strips, as by definition, these regions do not have a corner.) 
    
    Applying the Schnyder wood chiselling algorithm, the Schnyder wood of the UIHPT rooted at $(b_x, b_{x+1})$ is now constructed by first taking the Schnyder wood segment explored by $P_x$, and then repeatedly taking Schnyder wood strips on the remaining unexplored region. Observe that the colouring and orientation assigned to the UIHPT satisfies the Schnyder root condition at $(b_x, b_{x+1})$ because this condition is satisfied by the Schnyder wood segment. The colouring produced by the chiselling algorithm also satisfies the Schnyder boundary condition because the Schnyder wood segment and each successive Schnyder wood strip all satisfy the boundary condition on their lower boundaries, and the corner vertex of the Schnyder wood segment satisfies the Schnyder corner condition. Finally, the Schnyder condition is satisfied at every interior vertex of the UIHPT, because each interior vertex is either interior in a segment or strip that itself satisfies the Schnyder condition, or lies on the upper boundary of one segment or strip and the lower boundary of another Schnyder wood strip, in which case the Schnyder condition is satisfied by definition of the Schnyder boundary and upper boundary conditions.
\end{proof}

\subsection{Red, yellow and blue forests} \label{ssec:structure:forests}

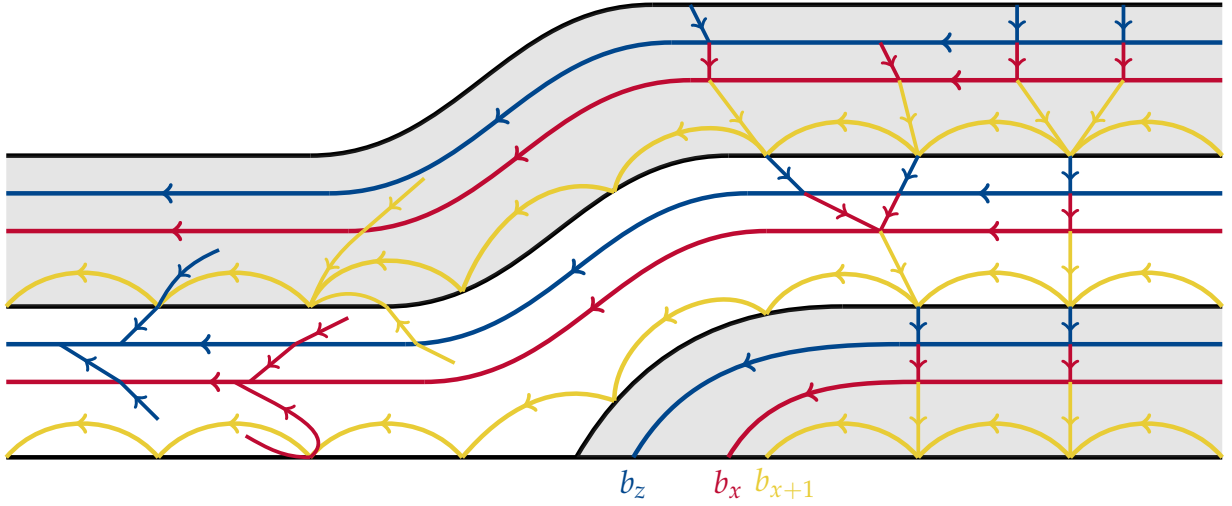
\begin{figure}[h]
    \centering
    \begin{tikzpicture}[line width=1.2]
    \draw (0,0) edge[line width=1.7] (16,0);

    \node[text=myyellow] at (10.25,-0.35) {$b_{x+1}$};
    \node[text=red] at (9.5,-0.35) {$b_{x}$};
    \node[text=blue] at (8.25,-0.35) {$b_{z}$};

    \path[every node/.style={font=\sffamily\small}]
        (16,2) edge[line width=1.7] (11,2)
        (11,2) edge[line width=1.7, out=180, in=60] (7.5,0);

    \fill[fill=gray,opacity=0.2] (7.5,0) to (16,0) to (16,2) to (11,2) to[out=180, in=60] cycle;

    \path[every node/.style={font=\sffamily\small}]
        (16,4) edge[line width=1.7] (9.5,4)
        (9.5,4) edge[line width=1.7, in=0, out=180] (5,2)
        (5,2) edge[line width=1.7] (0,2);

    \path[every node/.style={font=\sffamily\small}]
        (16,6) edge[line width=1.7] (8.5,6)
        (8.5,6) edge[line width=1.7, in=0, out=180] (4,4)
        (4,4) edge[line width=1.7] (0,4);

    \fill[fill=gray,opacity=0.2] (0,2) to (5,2) to[out=0,in=180] (9.5,4) to (16,4) to (16,6) to (8.5,6) to[in=0, out=180] (4,4) to (0,4) to cycle;

    \path[every node/.style={font=\sffamily\small}]
        (16,1.5) edge[bedge, line width=1.7] (11.75,1.5)
        (11.75,1.5) edge[bedge, ->-, line width=1.7, out=180, in=60] (8.25,0);
        
    \path[every node/.style={font=\sffamily\small}]
        (16,1) edge[redge, line width=1.7] (12,1)
        (12,1) edge[redge, ->-, line width=1.7, out=180, in=60] (9.5,0);

    \path[every node/.style={font=\sffamily\small}]
        (16,0) edge[yedge, ->-, line width=1.7, out=130, in=50] (14,0)
        (14,0) edge[yedge, ->-, line width=1.7, out=130, in=50] (12,0)
        (12,0) edge[yedge, ->-, line width=1.7, out=130, in=50] (10,0);

    \path[every node/.style={font=\sffamily\small}]
        (16,3.5) edge[bedge, ->-, line width=1.7] (9.75,3.5)
        (9.75,3.5) edge[bedge, ->-, line width=1.7, in=0, out=180] (5.25,1.5)
        (5.25,1.5) edge[bedge, ->-, line width=1.7] (0,1.5);
    \path[every node/.style={font=\sffamily\small}]
        (16,3) edge[redge, ->-, line width=1.7] (10,3)
        (10,3) edge[redge, ->-, line width=1.7, in=0, out=180] (5.5,1)
        (5.5,1) edge[redge, ->-, line width=1.7] (0,1);
    \path[every node/.style={font=\sffamily\small}]
        (16,2) edge[yedge, ->-, line width=1.7, out=130, in=50] (14,2)
        (14,2) edge[yedge, ->-, line width=1.7, out=130, in=50] (12,2)
        (12,2) edge[yedge, ->-, line width=1.7, out=130, in=60] (10,1.9)
        (10,1.9) edge[yedge, ->-, line width=1.7, out=140, in=90] (8,0.75)
        (8,0.75) edge[yedge, ->-, line width=1.7, out=160, in=50] (6,0)
        (6,0) edge[yedge, ->-, line width=1.7, out=130, in=50] (4,0)
        (4,0) edge[yedge, ->-, line width=1.7, out=130, in=50] (2,0)
        (2,0) edge[yedge, ->-, line width=1.7, out=130, in=50] (0,0);

    \path[every node/.style={font=\sffamily\small}]
        (16,5.5) edge[bedge, ->-, line width=1.7] (8.75,5.5)
        (8.75,5.5) edge[bedge, ->-, line width=1.7, in=0, out=180] (4.25,3.5)
        (4.25,3.5) edge[bedge, ->-, line width=1.7] (0,3.5);
    \path[every node/.style={font=\sffamily\small}]
        (16,5) edge[redge, ->-, line width=1.7] (9,5)
        (9,5) edge[redge, ->-, line width=1.7, in=0, out=180] (4.5,3)
        (4.5,3) edge[redge, ->-, line width=1.7] (0,3);
    \path[every node/.style={font=\sffamily\small}]
        (16,4) edge[yedge, ->-, line width=1.7, out=130, in=50] (14,4)
        (14,4) edge[yedge, ->-, line width=1.7, out=130, in=50] (12,4)
        (12,4) edge[yedge, ->-, line width=1.7, out=130, in=50] (10,4)
        (10,4) edge[yedge, ->-, line width=1.7, out=120, in=80] (8,3.52)
        (8,3.52) edge[yedge, ->-, line width=1.7, out=160, in=70] (6,2.2)
        (6,2.2) edge[yedge, ->-, line width=1.7, out=130, in=70] (4,2)
        (4,2) edge[yedge, ->-, line width=1.7, out=130, in=50] (2,2)
        (2,2) edge[yedge, ->-, line width=1.7, out=130, in=50] (0,2);

    \path[every node/.style={font=\sffamily\small}]
        (14.7,6) edge[bedge, ->-, line width=1.4] (14.7,5.5)
        (14.7,5.5) edge[redge, ->-, line width=1.4] (14.7,5)
        (14.7,5) edge[yedge, ->-, line width=1.4] (14,4)
        (14,4) edge[bedge, ->-, line width=1.4] (14,3.5)
        (14,3.5) edge[redge, ->-, line width=1.4] (14,3)
        (14,3) edge[yedge, ->-, line width=1.4] (14,2)
        (14,2) edge[bedge, ->-, line width=1.4] (14,1.5)
        (14,1.5) edge[redge, ->-, line width=1.4] (14,1)
        (14,1) edge[yedge, ->-, line width=1.4] (14,0);

    \path[every node/.style={font=\sffamily\small}]
        (13.3,6) edge[bedge, ->-, line width=1.4] (13.3,5.5)
        (13.3,5.5) edge[redge, ->-, line width=1.4] (13.3,5)
        (13.3,5) edge[yedge, ->-, line width=1.4] (14,4);

    \path[every node/.style={font=\sffamily\small}]
        (9,6) edge[bedge, ->-, line width=1.4] (9.25,5.5)
        (9.25,5.5) edge[redge, ->-, line width=1.4] (9.25,5)
        (9.25,5) edge[yedge, ->-, line width=1.4] (10,4)
        (10,4) edge[bedge, ->-, line width=1.4] (10.5,3.5)
        (10.5,3.5) edge[redge, ->-, line width=1.4] (11.5,3);

    \path[every node/.style={font=\sffamily\small}]
        (11.5,5.5) edge[redge, ->-, line width=1.4] (11.75,5)
        (11.75,5) edge[yedge, ->-, line width=1.4] (12,4)
        (12,4) edge[bedge, ->-, line width=1.4] (11.75,3.5)
        (11.75,3.5) edge[redge, ->-, line width=1.4] (11.5,3)
        (11.5,3) edge[yedge, ->-, line width=1.4] (12,2)
        (12,2) edge[bedge, ->-, line width=1.4] (12,1.5)
        (12,1.5) edge[redge, ->-, line width=1.4] (12,1)
        (12,1) edge[yedge, ->-, line width=1.4] (12,0);

    \path[every node/.style={font=\sffamily\small}]
        (5.5,3.7) edge[yedge, ->-, line width=1.4] (4.7,3)
        (4.7,3) edge[yedge, ->-, line width=1.4, bend right=15] (4,2);
    \path[every node/.style={font=\sffamily\small}]
        (5.9,1.25) edge[yedge, line width=1.4] (5.4,1.5)
        (5.4,1.5) edge[yedge, ->-, line width=1.4] (5,2)
        (5,2) edge[yedge, , bend right=35, line width=1.4] (4,2);

    \path[every node/.style={font=\sffamily\small}]
        (2.8,2.75) edge[bedge, ->-, bend right=20, line width=1.4] (2,2)
        (2,2) edge[bedge, ->-, line width=1.4] (1.5,1.5);
    \path[every node/.style={font=\sffamily\small}]
        (2,0.5) edge[bedge, ->-, line width=1.4] (1.5,1)
        (1.5,1) edge[bedge, ->-, line width=1.4] (0.7,1.5);

    \path[every node/.style={font=\sffamily\small}]
        (4.5,1.85) edge[redge, ->-, line width=1.4] (3.8,1.5)
        (3.8,1.5) edge[redge, ->-, line width=1.4] (3.2,1);
    \path[every node/.style={font=\sffamily\small}]
        (3.15,0.285) edge[redge, out=-30, in=180, line width=1.4] (4,0)
        (4,0) edge[redge, ->-, out=40, in=-30, line width=1.4] (3,1);
\end{tikzpicture}
    \caption{The structure of the Schnyder wood on the UIHPT constructed by the chiselling algorithm. The Schnyder wood segment and second Schnyder wood strip are shaded, and the infinite monochromatic paths in each segment/strip are shown. On the left the additional coloured edges indicate which monochromatic trees can intersect which monochromatic paths. On the right, the additional edges indicate examples of the paths obtained by following the unique outgoing edge from each path to the preceding infinite path, down to the initial boundary.}
    \label{fig:uihpt_structure}
\end{figure}

We have now shown that the Schnyder wood chiselling algorithm constructs a Schnyder wood on the UIHPT. We showed in \cref{lem:infiniteschnyderwood_monochromatic} that the monochromatic subgraphs of this Schnyder wood must be forests. In this section we describe the structure of the monochromatic forests, proving \cref{thm:monochromatic_schnyder_wood_structure}. This structure is depicted in \cref{fig:uihpt_structure}. We begin with a description of the monochromatic structures produced on the Schnyder wood segment. 

\begin{lemma}\label{lem:monochromatic_segment_subgraphs}
    Consider the Schnyder wood segment explored by the process $P_x$. The monochromatic subgraphs of the segment consist of
    \begin{itemize}
        \item An infinite yellow tree rooted at $b_{x+1}$, and a forest of finite yellow trees rooted at vertices on the upper boundary or at the corner vertex of the segment;
        \item An infinite red tree rooted at $b_x$; 
        \item An infinite blue tree rooted at some vertex $b_z$ with $z < x$, and a forest of finite blue trees rooted at non-corner vertices on the lower boundary. 
    \end{itemize}
    Furthermore, each monochromatic infinite tree contains exactly one infinite path. Every vertex on the upper boundary of the Schnyder wood segment has an outgoing blue edge to the infinite blue path, every vertex on the infinite blue path has an outgoing red edge to the infinite red path, and every vertex on the infinite red path has an outgoing yellow edge to the infinite yellow path. The infinite yellow path consists entirely of vertices on the lower boundary of the Schnyder wood segment.
\end{lemma}

\begin{proof}
    Our proof proceeds very similarly to the proof of \cref{thm:monochromatic_subgraph_finite_bdry}. In particular, we make frequent use of \cref{lem:schnyderareayellowforest,lem:schnyderarearedforest,lem:schnyderareablueforest}, as well as other observations about the structure produced by the peeling process at each step, as already discussed in \cref{sec:infiniteschnyderwoods} and \cref{sec:infinite_triang_finite_bdry}. For convenience, we again include \cref{fig:example_left_right_steps}, demonstrating the important components added to the explored region in left or right steps of the Schnyder peeling process.

    \begin{figure}[h]
        \centering
        \vspace{-1cm}
        \begin{tikzpicture}[line width=1.2]

    \foreach \x in {2,...,8,11,12}{
        \node[mycircle, opacity=0.5] (bn\x) at (\x,0) {};
    }
    \foreach \x in {15,16,17}{
        \node[mycircle] (bn\x) at (\x,0) {};
    }
    \node[mycircle, label={[red]below:{$t_1$}}, color=red, opacity=0.5] (bn9) at (9,0) {};
    \node[mycircle, label={[xshift=-0.2cm, red]below:{$t_0$}}, color=red, opacity=0.5] (bn10) at (10,0) {};
    \node[mycircle, opacity=0, label={[xshift=0.2cm]below:{$h_0$}}] (bn11) at (11,0) {};
    \node[mycircle, opacity=0.5, label=below:{$h_1$}] (bn13) at (13,0) {};
    \node[mycircle, label=below:{$h_2$}] (bn14) at (14,0) {};
    \node[mycircle, label=below:{$h_3$}] (bn18) at (18,0) {};
    \node[color=myyellow] (lbl0) at (10.5, -0.3) {\small{$c_x(0)$}};
    
    \foreach \x in {2,...,7,11,12,13}{
        \pgfmathtruncatemacro{\next}{\x + 1};
        \draw (bn\x) edge[opacity=0.5] (bn\next);
    }
    \foreach \x in {14,...,17}{
        \pgfmathtruncatemacro{\next}{\x + 1};
        \draw (bn\x) -- (bn\next);
    }
    
    \draw (bn8) edge[->-,redge, opacity=0.5] (bn9)
          (bn9) edge[->-,redge, opacity=0.5] (bn10)
          (bn10) edge[->-,yedge, opacity=0.5] (bn11);

    \draw (bn9) edge[->-, yedge, opacity=0.5, bend left=55, looseness=1.55] (bn13);
    \node[color=myyellow] (lbl1) at (11.1, 1.8) {\small{$c_x(1)$}};
    \node[mycircle, opacity=0.5] (u1r1) at (10.5,0.6) {};
    \node[mycircle, opacity=0.5] (u2r1) at (11,1.2) {};
    \draw (bn10) edge[opacity=0.5] (u1r1)
    (u1r1) edge[opacity=0.5] (u2r1)
    (u2r1) edge[opacity=0.5] (bn13);
    
    \draw (u1r1) edge[->-, bedge, opacity=0.5] (bn9);
    \draw (u2r1) edge[->-, bedge, opacity=0.5] (bn9);
    \draw (bn13) edge[->-, yedge, bend right=25, opacity=0.5] (bn11);

    \draw (bn8) edge[->-, yedge, opacity=0.5, bend right=30] (bn2);
    \node[color=myyellow] (lbl2) at (5, 0.6) {\small{$c_x(2)$}};
    \node[mycircle,color=red] (u1l1) at (8,1.3) {};
    \node[mycircle,color=red] (u2l1) at (7.2,1.5) {};
    \node[mycircle, opacity=0.5] (u3l1) at (6.4,1.7) {};
    \node[mycircle, opacity=0.5] (u4l1) at (5.4,1.8) {};
    \node[mycircle, opacity=0.5] (u5l1) at (4.3,1.6) {};
    
    \draw (u2l1) edge[opacity=0.5] (u3l1)
    (u3l1) edge[opacity=0.5] (u4l1)
    (u4l1) edge[opacity=0.5] (u5l1)
    (u5l1) edge[opacity=0.5, bend right=10] (bn2);
    
    \draw (u1l1) edge[->-, bedge, opacity=0.5, bend left=10] (bn8);
    \draw (u2l1) edge[->-, bedge, opacity=0.5, bend left=10] (bn8);
    \draw (u3l1) edge[->-, bedge, opacity=0.5, bend left=10] (bn8);
    \draw (u4l1) edge[->-, bedge, opacity=0.5, bend left=10] (bn8);
    \draw (u5l1) edge[->-, bedge, opacity=0.5, bend left=10] (bn8);
    
    \draw (u1l1) edge[->-, redge, opacity=0.5] (bn9);
    \draw (u1l1) edge[->-, yedge, opacity=0.5, bend left=60, looseness=1.2] (bn14);
    \node[color=myyellow] (lbl3) at (13, 2.2) {\small{$c_x(3)$}};
    \node[mycircle, opacity=0.5] (u1r2) at (8.8,0.8) {};
    \node[mycircle, opacity=0.5] (u2r2) at (9.5,1.8) {};
    \draw (bn9) edge[opacity=0.5] (u1r2)
    (u1r2) edge[opacity=0.5] (u2r2)
    (u2r2) edge[opacity=0.5, bend left=50, looseness=1] (bn14);
    
    \draw (u1r2) edge[->-, bedge, opacity=0.5] (u1l1);
    \draw (u2r2) edge[->-, bedge, opacity=0.5] (u1l1);
    \draw (bn14) edge[->-, yedge, bend right=60, opacity=0.5] (bn13);


    \draw (u2l1) edge[->-, redge, line width=1.7] (u1l1);
    \draw (u2l1) edge[->-, yedge, bend left=80, looseness=1.5, line width=1.7] (bn18);
    \node[mycircle] (u1r3) at (8,2.1) {};
    \node[mycircle] (u2r3) at (10,4.3) {};
    \draw (u1l1) to (u1r3) to (u2r3) edge[bend left=55] (bn18);
    \draw (u1r3) edge[->-, bedge,line width=1.7] (u2l1);
    \draw (u2r3) edge[->-, bedge, line width=1.7, bend right=10, bend right=13] (u2l1);
    \draw (bn18) edge[->-, yedge, bend right=75, line width=1.7] (bn16);
    \draw (bn16) edge[->-, yedge, bend right=75, line width=1.7] (bn14);

    \node[color=red] (lbl) at (7, 1.8) {$t_3$};
    \node[color=red] (lbl3) at (7.8, 1.1) {$t_2$};
    \node[color=myyellow] (lbl2) at (13.5, 5.85) {\small{$c_x(4)$}};
    
    \draw (10,3) edge[redge,dotted,line width=1.7,bend left=10] (11.5,3.2)
                 (11.5,3.2) edge[redge,dotted,line width=1.7,bend left=5] (12.5,3.3)
                 (11.5,3.2) edge[redge,dotted,line width=1.7,bend right=5] (12.5,2.9)
          (10,3) edge[redge,dotted,line width=1.7,bend left=10] (12,2.5)
          (10,3) edge[->-, redge, dotted, bend right=20, line width=1.7] (u1l1);

    \draw (11, 3.8) edge[->-, bedge, dotted, bend right=15, line width=1.7] (u1r3);
    
    \draw (12, 4.2) edge[->-, bedge, dotted, bend right=15, line width=1.7] (u2r3);
    \draw (14.5, 3.5) edge[-, bedge, dotted, bend right=10, line width=1.7] (12, 4.2);
    \draw (13.5, 3.2) edge[-, bedge, dotted, bend right=10, line width=1.7] (12, 4.2);
    
    \draw (15.5, 2) edge[->-, bedge, dotted, bend left=25, line width=1.7] (bn18);

    \draw (14.5, 1.8) edge[->-, yedge, dotted, bend right=10, line width=1.7] (bn14);
    \draw (14, 2.5) edge[-, yedge, dotted, bend right=10, line width=1.7] (14.5, 1.8);
    \draw (15, 2.7) edge[-, yedge, dotted, bend right=10, line width=1.7] (14.5, 1.8);

    \coordinate (bluetree3) at (15, 0.5); 
    \coordinate (bluetree4) at (15.5, 1.5); 
    \draw (bluetree3) edge[->-, bedge, dotted, line width=1.7] (bn15);
    \draw (bluetree4) edge[->-, bedge, bend left=15, dotted, line width=1.7] (bn16);
\end{tikzpicture}
        \begin{tikzpicture}[line width=1.2]

    \node[mycircle] (bn2) at (2,0) {};
    \foreach \x in {3,...,8,12,17,18}{
        \node[mycircle, opacity=0.5] (bn\x) at (\x,0) {};
    }
    \foreach \x in {15,16}{
        \node[mycircle, opacity=0.5] (bn\x) at (\x,0) {};
    }
    \node[mycircle, label={[red]below:{$t_1$}}, color=red, opacity=0.5] (bn9) at (9,0) {};
    \node[mycircle, label={[xshift=-0.2cm, red]below:{$t_0$}}, color=red, opacity=0.5] (bn10) at (10,0) {};
    \node[mycircle, opacity=0.5, label={[xshift=0.2cm]below:{$h_0$}}] (bn11) at (11,0) {};
    \node[mycircle, opacity=0.5, label=below:{$h_1$}] (bn13) at (13,0) {};
    \node[mycircle, opacity=0.5, label=below:{$h_2$}] (bn14) at (14,0) {};
    \node[color=myyellow] (lbl0) at (10.5, -0.3) {\small{$c_x(0)$}};
    
    \foreach \x in {2,...,7,11,12,13}{
        \pgfmathtruncatemacro{\next}{\x + 1};
        \draw (bn\x) edge[opacity=0.5] (bn\next);
    }
    \foreach \x in {14,...,17}{
        \pgfmathtruncatemacro{\next}{\x + 1};
        \draw (bn\x) -- (bn\next);
    }
    
    \draw (bn8) edge[->-,redge, opacity=0.5] (bn9)
          (bn9) edge[->-,redge, opacity=0.5] (bn10)
          (bn10) edge[->-,yedge, opacity=0.5] (bn11);

    \draw (bn9) edge[->-, yedge, opacity=0.5, bend left=55, looseness=1.55] (bn13);
    \node[color=myyellow] (lbl1) at (11.1, 1.8) {\small{$c_x(1)$}};
    \node[mycircle, opacity=0.5] (u1r1) at (10.5,0.6) {};
    \node[mycircle, opacity=0.5] (u2r1) at (11,1.2) {};
    \draw (bn10) edge[opacity=0.5] (u1r1)
    (u1r1) edge[opacity=0.5] (u2r1)
    (u2r1) edge[opacity=0.5] (bn13);
    
    \draw (u1r1) edge[->-, bedge, opacity=0.5] (bn9);
    \draw (u2r1) edge[->-, bedge, opacity=0.5] (bn9);
    \draw (bn13) edge[->-, yedge, bend right=25, opacity=0.5] (bn11);

    \draw (bn8) edge[->-, yedge, opacity=0.5, bend right=30] (bn2);
    \node[color=myyellow] (lbl2) at (5, 0.6) {\small{$c_x(2)$}};
    \node[mycircle,color=red] (u1l1) at (8,1.3) {};
    \node[mycircle,color=red] (u2l1) at (7.2,1.5) {};
    \node[mycircle] (u3l1) at (6.4,1.7) {};
    \node[mycircle] (u4l1) at (5.4,1.8) {};
    \node[mycircle] (u5l1) at (4.3,1.6) {};

    \draw (u2l1) edge[->-, redge, line width=1.7] (u1l1);
    \draw (u2l1) edge[] (u3l1)
    (u3l1) edge[] (u4l1)
    (u4l1) edge[] (u5l1)
    (u5l1) edge[bend right=10] (bn2);
    
    \draw (u1l1) edge[->-, bedge, opacity=0.5, bend left=10] (bn8);
    \draw (u2l1) edge[->-, bedge, opacity=0.5, bend left=10] (bn8);
    \draw (u3l1) edge[->-, bedge, opacity=0.5, bend left=10] (bn8);
    \draw (u4l1) edge[->-, bedge, opacity=0.5, bend left=10] (bn8);
    \draw (u5l1) edge[->-, bedge, opacity=0.5, bend left=10] (bn8);
    
    \draw (u1l1) edge[->-, redge, opacity=0.5] (bn9);
    \draw (u1l1) edge[->-, yedge, opacity=0.5, bend left=60, looseness=1.2] (bn14);
    \node[color=myyellow] (lbl3) at (13, 2.2) {\small{$c_x(3)$}};
    \node[mycircle, opacity=0.5] (u1r2) at (8.8,0.8) {};
    \node[mycircle, opacity=0.5] (u2r2) at (9.5,1.8) {};
    \draw (bn9) edge[opacity=0.5] (u1r2)
    (u1r2) edge[opacity=0.5] (u2r2)
    (u2r2) edge[opacity=0.5, bend left=50, looseness=1] (bn14);
    
    \draw (u1r2) edge[->-, bedge, opacity=0.5] (u1l1);
    \draw (u2r2) edge[->-, bedge, opacity=0.5] (u1l1);
    \draw (bn14) edge[->-, yedge, bend right=60, opacity=0.5] (bn13);


    \draw (u2l1) edge[->-, yedge, bend right=70, looseness=1.7, line width=1.6] (bn2);
    \node[color=myyellow] (lbl4) at (5.2, 3.5) {\small{$c_x(4)$}};
    \node[mycircle] (u1l2) at (7.5,2.8) {};
    \node[mycircle] (u2l2) at (6,3.9) {};

    \draw (u1l1) to (u1l2) to (u2l2) edge[bend right=50, looseness=1.2] (bn2);
    \draw (u1l2) edge[->-, bedge,line width=1.7] (u2l1);
    \draw (u2l2) edge[->-, bedge,line width=1.7] (u2l1);

    \node[color=red] (lbl3) at (7.8, 1.1) {$t_2$};

    \draw (3, 1.8) edge[->-, yedge, dotted, bend right=10, line width=1.7] (bn2);
    \draw (3.7, 2.8) edge[-, yedge, dotted, bend right=10, line width=1.7] (3, 1.8);
    \draw (3.5, 2) edge[-, yedge, dotted, bend right=10, line width=1.7] (3, 1.8);
    
    \draw (5.6, 2.6) edge[->-, redge, dotted, bend left=10, line width=1.7] (u2l1);
    \draw (4.2, 2.9) edge[->-, redge, dotted, bend left=10, line width=1.7] (5.6, 2.6);
    \draw (4.6, 2.4) edge[->-, redge, dotted, bend left=10, line width=1.7] (5.6, 2.6);
    
    \draw (3.8, 2.2) edge[->-, bedge, dotted, line width=1.7] (u5l1);
    \draw (4.5, 2.2) edge[->-, bedge, dotted, line width=1.7] (u5l1);
    \draw (5, 2.2) edge[->-, bedge, dotted, line width=1.7] (u4l1);
    \draw (6, 2.1) edge[->-, bedge, dotted, line width=1.7] (u3l1);
    
\end{tikzpicture}
        \caption{The various monochromatic components added to the revealed Schnyder wood after $P_x$ takes either a right (top) or left (bottom) step. These components are indicated in bold.}
        \label{fig:example_left_right_steps}
    \end{figure}
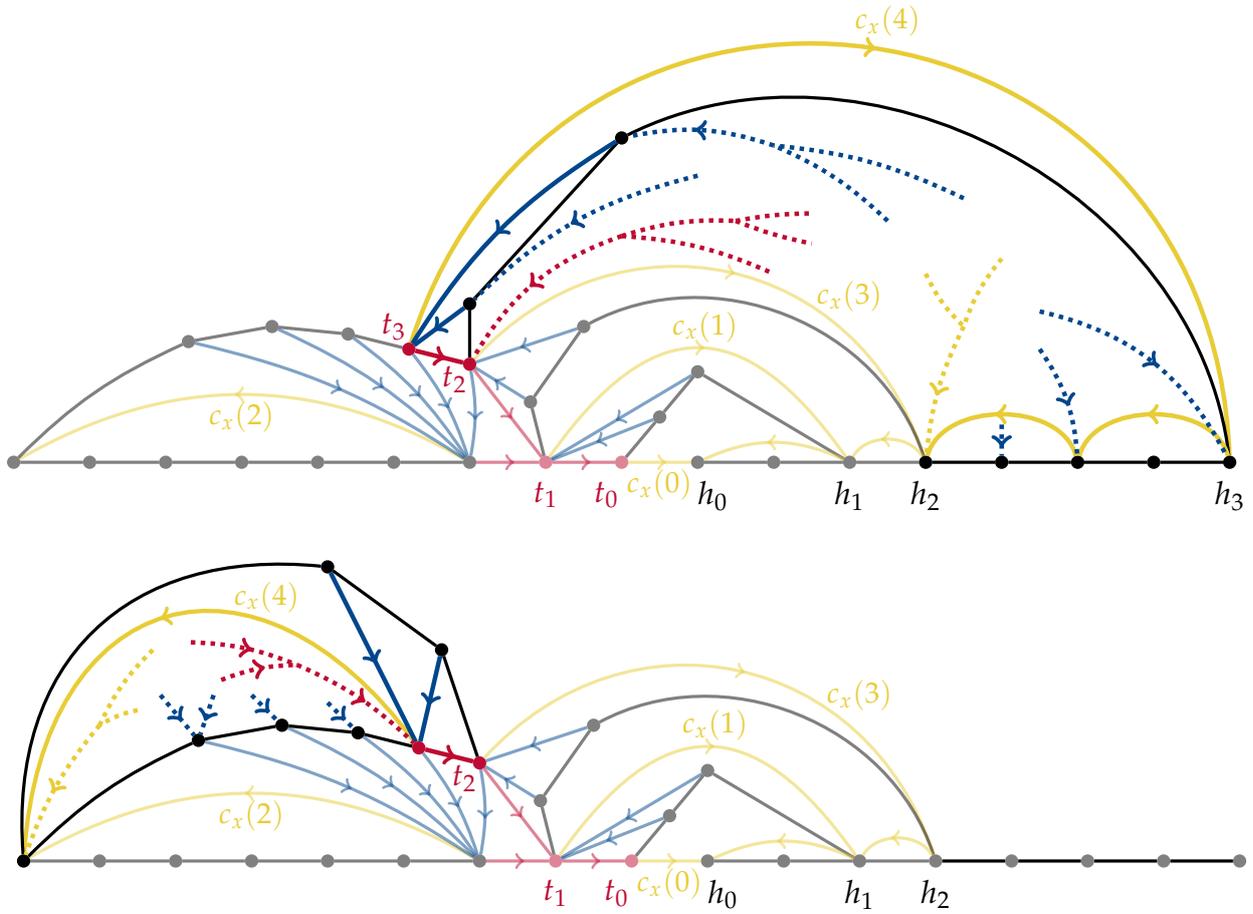

    We now confirm the structure of each monochromatic subgraph in turn. 
    \begin{claim}\label{clm:schnydersegmentyellowforest}
        The yellow subgraph of the Schnyder wood segment explored by $P_x$ consists of
        \begin{itemize}
            \item An infinite tree rooted at $b_{x+1}$ that contains a unique infinite path whose vertices are on the lower boundary of the segment right of $b_{x+1}$ and include the heads of every chord chosen by a right step of $P_x$; and
            \item A forest of finite trees rooted at vertices on the upper boundary or at the corner vertex of the segment.
        \end{itemize}
    \end{claim}
    
    \begin{poc}
        We showed in \cref{lem:schnyderareayellowforest} that after $i$ steps of $P_x$, the yellow subgraph of the revealed area $\area{x}{i}$ consists of a tree rooted at $b_{x+1}$ that is disjoint from the current upper boundary and corner vertex, and that contains a path whose vertices include the heads $h_\ell$ of every chord chosen by a right step of $P_x$ up to step $i$, along with a forest of trees rooted at vertices on the upper boundary or at the corner vertex of $\area{x}{i}$. By \cref{obs:infiniterightsteps}, $P_x$ will take an infinite number of right steps while constructing the Schnyder wood segment, so the yellow subgraph of this segment contains an infinite yellow tree rooted at $b_{x+1}$ that contains an infinite path whose vertices include $(h_\ell,\ell \ge 0)$. Note that $h_0 = b_{x+1}$. Moreover, after every step $i$, all vertices on $\boundary{x}{i}$ to the right of the head $h_\ell$ of the current root edge are on the initial boundary right of $b_{x+1}$. By \cref{lem:schnyderareayellowforest}, this implies that the yellow path between $h_{\ell+1}$ and $h_\ell$ consists entirely of vertices on the lower boundary of the Schnyder wood segment right of $b_{x+1}$.

        Consider now a finite yellow tree rooted at a vertex $v$ on the upper boundary or at the corner vertex of $\area{x}{i}$. If $v$ is later covered by a right step of $P_x$, the structure of the right step implies that $v$ will be added to the tree rooted at $b_{x+1}$. Alternatively, if $v$ is later covered by a left step, it will afterwards be contained in a finite tree rooted further left on $\boundary{x}{i}$. Since we know by \cref{lem:peeling_moves_right} that only a finite number of vertices left of the current root edge on $\boundary{x}{i}$ will eventually be covered by the Schnyder wood segment, this process must eventually terminate with $v$ being contained in a finite yellow tree rooted at a vertex on the upper boundary or at the corner vertex of the Schnyder wood segment.
        
        Finally, recall that if $P_x$ has performed $\ell$ right steps after step $i$, then the tree rooted at $b_{x+1}$ contains a directed path from $h_\ell$ to $h_0 = b_{x+1}$, and this tree intersects the boundary $\boundary{x}{i}$ only in the root edge $\roote{x}{i}$, the head of which is $h_\ell$. In particular, a path from any vertex revealed after step $i$ to $b_{x+1}$ must include $h_\ell$. This shows that the infinite path whose vertices include $(h_\ell,\ell \ge 0)$ is the unique infinite path in the yellow tree rooted at $b_{x+1}$.
    \end{poc}
    \begin{claim}\label{clm:schnydersegmentredforest}
        The red subgraph of the Schnyder wood segment explored by $P_x$ consists of an infinite tree rooted at $b_x$ that contains a unique infinite path whose vertices are the tails of every chord chosen by a right step of $P_x$. In particular, every vertex on the infinite red path has an outgoing yellow edge to the infinite yellow path.
    \end{claim}
    
    \begin{poc}
        It follows immediately from \cref{lem:schnyderarearedforest} that the red subgraph in the Schnyder wood segment consists of a single infinite red tree rooted at $b_x$, and that this tree contains the infinite path whose vertices are the tails $t_\ell$ of every chord chosen by a right step of $P_x$.

        In particular, if $P_x$ has performed $\ell$ right steps after step $i$, then the tree rooted at $b_x$ contains the directed path $t_\ell, t_{\ell-1}, \dots, t_0$. Moreover, by \cref{obs:finitestepsupperboundary}, no upper boundary or corner vertex of $\area{x}{i}$ can be contained in this tree, and so the red tree can only intersect the boundary $\boundary{x}{i}$ at $t_\ell$ and $h_\ell$. The Schnyder root condition implies that no red edges can be added to $h_\ell$ during subsequent steps, so a path from any vertex revealed after step $i$ to $b_x$ must include $t_\ell$. This shows that the infinite path $\dots, t_2, t_1, t_0$ is the unique infinite path in the red tree rooted at $b_x$. Finally, every vertex $t_\ell$ on the infinite red path has an outgoing yellow edge to $h_\ell$ which is on the infinite yellow path.
    \end{poc}

    To determine the structure of the blue subgraph of the Schnyder wood segment, let $i_0 \coloneqq 0$, and inductively define $i_\ell$ as the last step of $P_x$ such that $\peelv{x}{i_\ell}$ is on the boundary $\boundary{x}{i_{\ell - 1}}$. We first need to show that each $i_\ell$ exists, but note that \emph{if} $i_\ell$ exists, then with $i = i_\ell$, the steps $i_0, \dots, i_\ell$ coincide with the steps defined in \cref{lem:schnyderareablueforest}.

    To show that each $i_\ell$ exists, note that a peeling vertex which is on the initial boundary is never strictly right of a previous peeling vertex. By \cref{lem:peeling_moves_right}, only a finite number of vertices left of the first peeling vertex on the initial boundary will eventually be covered by $P_x$. This shows that there exists a step $i$ of $P_x$ such that $\peelv{x}{i}$ is on the initial boundary, but all subsequent steps use a peeling vertex that is not on the initial boundary, and so $i_1 = i$. After step $i_1$, it follows from \cref{lem:hole_uihpt} that $T \setminus \area{x}{i_1}$ is distributed like a UIHPT, so we may repeat the argument. \cref{lem:peeling_moves_right} once again implies the existence of a step $i' > i_1$ such that $\peelv{x}{i'}$ is on the boundary $\boundary{x}{i_1}$, but all subsequent steps use a peeling vertex that is not on $\boundary{x}{i'}$. Hence, $i_2 = i'$. By repeating this argument, it follows that each $i_\ell$ exists. Also note that each step $i_\ell$ is a left step, as otherwise $\peelv{x}{i_\ell+1}$ would still be on $\boundary{x}{i_{\ell - 1}}$.

    \begin{claim}\label{clm:schnydersegmentblueforest}
        The blue subgraph of the Schnyder wood segment explored by $P_x$ consists of
        \begin{itemize}
            \item An infinite tree rooted at some vertex $b_z$ with $z < x$ that contains a unique infinite path whose vertices are the peeling vertices $\peelv{x}{i_\ell}$ with $\ell \ge 1$, and
            \item A forest of finite trees rooted at vertices on the lower boundary of the segment.
        \end{itemize}
        Moreover, every vertex on the infinite blue path has an outgoing red edge to the infinite red path, and every vertex on the upper boundary of the Schnyder wood segment has an outgoing blue edge to the infinite blue path.
    \end{claim}
    
    \begin{poc}
        It follows immediately from \cref{lem:schnyderareablueforest} that the blue subgraph of the Schnyder wood segment consists of blue trees rooted at vertices on the initial boundary, and that the blue tree rooted at $\peelv{x}{i_1}$ contains an infinite path whose vertices are the peeling vertices $\peelv{x}{i_\ell}$ with $\ell \geq 1$. Note that $\peelv{x}{i_1} = b_z$ for some $z < x$.

        Furthermore, \cref{lem:schnyderareablueforest} shows that after step $i_1$, every vertex on the upper boundary $U$ of $\area{x}{i_1}$ has an outgoing blue edge to the peeling vertex $\peelv{x}{i_1}$. In particular, every vertex on $U$ is contained in the tree rooted at $\peelv{x}{i_1}$. Since any blue tree added to the initial boundary in subsequent steps is contained in a finite region under a chord chosen by $P_x$, such a tree must be finite, and so it follows that all blue trees of the Schnyder wood segment rooted on the initial boundary at vertices other than $\peelv{x}{i_1}$ will be finite.

        Note that running $P_x$ after step $i_\ell$ for some $\ell \ge 1$ corresponds to running the Schnyder peeling process on the unexplored region of $T$ with initial boundary $\boundary{x}{i_\ell}$ and root edge $\roote{x}{i_\ell}$. So, by the same arguments as above, every blue (sub-)tree that is rooted at a vertex of $\boundary{x}{i_\ell} \sm \set{\peelv{x}{i_{\ell+1}}}$ will be finite. That is, for only a finite number of vertices $v$ in the blue tree of the Schnyder wood segment rooted at $\peelv{x}{i_1}$, the blue path from $v$ to $\peelv{x}{i_1}$ will not include $\peelv{x}{i_{\ell+1}}$. This shows that the infinite path $\dots, \peelv{x}{i_3}, \peelv{x}{i_2}, \peelv{x}{i_1}$ is the unique infinite path in the blue tree rooted at $\peelv{x}{i_1}$.

        Finally, recall that in any peeling step $i$, an outgoing red edge is added from $\peelv{x}{i}$ to the tail of $\roote{x}{i-1}$. We have already shown in \cref{clm:schnydersegmentredforest} that the infinite red path consists of the tails of $\roote{x}{i}$ for each right step $i$. It follows that every vertex on the infinite blue path has an outgoing red edge to the infinite red path. Also, every vertex on the upper boundary of the Schnyder wood segment is on the upper boundary $U$ of $\area{x}{i_k}$ for some sufficiently large $k$. By \cref{lem:schnyderareablueforest}, every vertex on $U$ has an outgoing blue edge to one of the peeling vertices $\peelv{x}{i_\ell}$ for $\ell \le k$, which is on the infinite blue path. Therefore, every vertex on the upper boundary of the Schnyder wood segment has an outgoing blue edge to the infinite blue path.
    \end{poc}
    The fact that the Schnyder wood segment has the structure asserted in the lemma now follows directly from the claims, so this completes the proof.
\end{proof}

Let $\infpath{y}$\glossarylabel{gl:infpath}, $\infpath{r}$, and $\infpath{b}$ denote, respectively, the unique infinite yellow, red, and blue paths in the Schnyder wood segment, as given in \cref{lem:monochromatic_segment_subgraphs}.

\begin{lemma} \label{lem:disjoint_paths}
    The paths $\infpath{y}$, $\infpath{r}$, and $\infpath{b}$ are vertex-disjoint.
\end{lemma}

\begin{proof}
    We showed in the proof of \cref{lem:monochromatic_segment_subgraphs} that $\infpath{b}$ consists of the peeling vertices of some subset of the left steps of $P_x$, $\infpath{r}$ consists of the tails of every chord chosen by a right step of $P_x$, and $\infpath{y}$ consists of vertices on the lower boundary of the segment right of $b_{x+1}$. Since a peeling vertex of a left step can never be the tail of a chord chosen by a right step of $P_x$, it follows immediately that $\infpath{r}$ and $\infpath{b}$ are vertex-disjoint. Also, since neither peeling vertices nor tails of chords are on the lower boundary of the segment right of $b_{x+1}$, $\infpath{b}$ and $\infpath{r}$ are both disjoint from $\infpath{y}$.
\end{proof}

Recall that a two-ended infinite directed path in a triangulation $T$ of the half-plane is \defn{right-directed} (resp.~\defn{left-directed}) if it divides $T$ into two regions so that the infinite region containing the boundary lies to the right (resp.~left) of the path. The structure of the Schnyder wood segment described in \cref{lem:monochromatic_segment_subgraphs} implies the following result about the monochromatic substructures in the Schnyder wood strip.

\begin{lemma}\label{lem:monochromatic_strip_subgraphs}
    The monochromatic subgraphs of the Schnyder wood strip are as follows.
    \begin{itemize}
        \item An infinite yellow tree, and a forest of finite yellow trees rooted at vertices on the upper boundary of the strip,
        \item An infinite red tree, and 
        \item An infinite blue tree, and a forest of finite blue trees rooted at vertices on the lower boundary of the strip. 
    \end{itemize}
    Furthermore, the infinite yellow, red, and blue trees each contain a unique left-directed path $\infpath{y}$\glossarylabel{gl:infpathstrip}, $\infpath{r}$, and $\infpath{b}$ respectively, and these paths are vertex-disjoint. Every vertex on the upper boundary of the Schnyder wood strip has an outgoing blue edge to $\infpath{b}$, every vertex on $\infpath{b}$ has an outgoing red edge to $\infpath{r}$, and every vertex on $\infpath{r}$ has an outgoing yellow edge to $\infpath{y}$. $\infpath{y}$ consists of vertices on the lower boundary of the Schnyder wood strip.
\end{lemma}

\begin{proof}
    Once again, we infer the structure of the Schnyder wood strip from the structure of the Schnyder wood segment, using our results from \cref{ssec:uihpt:strip} that defined the Schnyder wood strip as the colouring assigned to the explored region by the process $P_x$ as $x$ tends to $-\infty$. We know from the proof of \cref{thm:consistent_strip_colouring} that any two processes rooted sufficiently far left will eventually have a step at which both processes share a common boundary segment and every edge coloured after this step is coloured identically by both processes. By \cref{lem:monochromatic_segment_subgraphs,lem:disjoint_paths}, it follows that the Schnyder wood strip right of the common boundary has the following structure. It contains one infinite tree of each colour, each with a unique infinite path that is vertex-disjoint from the infinite paths in the two other colours. It also contains a forest of finite yellow trees rooted at the upper boundary of the strip, and a forest of finite blue trees rooted at the lower boundary of the strip and at the common boundary segment of the two processes.

    Note that the initial boundary vertices right of $b_{x+1}$ are on the left of the paths $\infpath{r}$, $\infpath{b}$, and $\infpath{y}$. Hence, for every boundary vertex $b_y$, there exists a sufficiently small $x$ so that $b_y$ is on the left of $\infpath{r}$, $\infpath{b}$, and $\infpath{y}$. It follows that, as $x \to -\infty$, these paths are left-directed. The vertex-disjointness and the existence of the edges between these paths follow directly from \cref{lem:monochromatic_segment_subgraphs} and \cref{lem:disjoint_paths}.

    For any finite monochromatic tree of the Schnyder wood strip, we can choose $x$ sufficiently small so that for all $y \leq x$, the finite tree is strictly to the right of the common boundary segment of $P_x$ and $P_y$. Therefore, by the above arguments, the finite monochromatic tree is either a yellow tree rooted at the upper boundary of the strip or a blue tree rooted at the lower boundary of the strip. It cannot be a blue tree rooted at a vertex on the common boundary segment since it is strictly to the right of that segment.
    
    It remains to check that in the $x \to -\infty$ limit, no additional infinite monochromatic paths can appear. Suppose for a contradiction that there exists another infinite monochromatic path $P$ in some colour $c \in \set{r, b, y}$. If $P$ intersects $\infpath{c}$, both paths will coincide from then on. Since $P$ is different from $\infpath{c}$, it follows that there is a one-ended infinite subpath of $P$ that never intersects $\infpath{c}$. If $P$ intersects an infinite path $\infpath{d}$ for $d \neq c$, then by the Schnyder condition $P$ must cross $\infpath{d}$ from one side to the other and cannot intersect $\infpath{d}$ anywhere else. Also note that the Schnyder upper boundary condition implies that $P$ cannot intersect the upper boundary more than once. Thus, there exists a one-ended infinite subpath $P' \subs P$ that does not intersect the upper boundary, $\infpath{b}$, $\infpath{r}$, or $\infpath{y}$.

    On the other hand, recall that every vertex on the upper boundary has an edge to $\infpath{b}$, every vertex on $\infpath{b}$ has an edge to $\infpath{r}$, every vertex on $\infpath{r}$ has an edge to $\infpath{y}$, and $\infpath{y}$ consists of vertices on the lower boundary. So, if we remove all vertices of the upper boundary, $\infpath{b}$, $\infpath{r}$, and $\infpath{y}$, as well as all edges between these vertices, this partitions the Schnyder wood strip into finite components. In particular, $P'$ must be contained in one of these components, contradicting that $P'$ is infinite.
\end{proof}

We remark that in both the Schnyder wood segment and the Schnyder wood strip, the following conditions are consequences of the properties of embeddings of the infinite trees and of the Schnyder condition.
\begin{itemize}
    \item The finite blue trees may intersect $\infpath{y}$, but not $\infpath{r}$,
    \item The infinite yellow tree may intersect $\infpath{r}$ but not $\infpath{b}$,
    \item The infinite red tree may intersect both $\infpath{y}$ and $\infpath{b}$,
    \item The infinite blue tree may intersect $\infpath{r}$ but not $\infpath{b}$, and
    \item The finite yellow trees may intersect $\infpath{b}$ but not $\infpath{r}$.
    \item The edges from every vertex on the upper boundary to $\infpath{b}$ are blue.
    \item The edges from every vertex on $\infpath{b}$ to $\infpath{r}$ are red.
    \item The edges from every vertex on $\infpath{r}$ to $\infpath{y}$ are yellow.
\end{itemize}
In particular, there is a unique blue edge from any vertex on the upper boundary of any Schnyder wood strip or segment to the infinite blue path, and a unique yellow path from any vertex on the lower boundary of any Schnyder wood strip or segment to the infinite yellow path. This implies that as the Schnyder wood chiselling algorithm repeatedly takes Schnyder wood strips, all of the previously finite trees other than the finite blue trees rooted on the initial boundary are connected to one of the infinite trees in the overall Schnyder wood. Thus, the following theorem now follows directly from our results for the Schnyder wood strip and Schnyder wood segment. 

\begin{theorem}\label{thm:monochromatic_schnyder_wood_structure}
    The monochromatic subgraphs of the Schnyder wood defined on the UIHPT by the Schnyder wood chiselling algorithm are as follows.
    \begin{itemize}
        \item A set of finite blue trees rooted at vertices on the initial boundary,
        \item Three one-ended yellow, red, and blue trees rooted at $b_{x+1}$, $b_x$, and $b_z$ respectively, for some $z < x$, each containing a unique infinite path, and
        \item An infinite sequence of two-ended yellow, red, and blue trees, each containing a unique left-directed path, so that the order of the paths moving away from the initial boundary cycles yellow, red, blue, yellow, red, blue, and so on.
    \end{itemize}
    Furthermore, every vertex on the one-ended infinite yellow path is on the initial boundary, and every vertex on the first two-ended infinite yellow path is either on the initial boundary or has an outgoing blue edge to the one-ended infinite blue path. Every vertex on the remaining infinite paths has an outgoing edge to the path immediately preceding it (and of the same colour as the path immediately preceding it).
\end{theorem}
We henceforth refer to the one-ended monochromatic paths in the Schnyder wood segment, and the two-ended monochromatic paths in the Schnyder wood strip as \defn{distinguished paths}.

We conclude this section by determining the structure of the monochromatic paths from any fixed vertex in a UIHPT $T$, proving a result analogous to \cref{thm:finite_paths_from_v_wind,thm:finite_boundary_paths_from_v_wind}. Consider the unique directed path $Q$ of colour $c$ that starts at a vertex $v$ in the Schnyder wood of $T$ constructed by the chiselling algorithm. If $Q$ is finite, let $\vpath{c}{v} \coloneqq Q$\glossarylabel{gl:vpathuihpt}. Otherwise, after some finite distance $Q$ meets one of the distinguished paths $P$ of colour $c$ contained in a Schnyder wood strip and follows $P$ from that point onwards. In this case, we define $\vpath{c}{v}$ to be the union of $Q$ and all of the distinguished paths of colour $c$ contained in a Schnyder wood strip or segment below $P$. Note that if $T$ has root edge $(b_x, b_{x+1})$, then $\vpath{y}{v}$ terminates at $b_{x+1}$, $\vpath{r}{v}$ terminates at $b_x$, and $\vpath{b}{v}$ terminates at an initial boundary vertex of $T$ other than $b_x$ or $b_{x+1}$. The following theorem is the analogue of \cref{thm:finite_paths_from_v_wind} and \cref{thm:finite_boundary_paths_from_v_wind} for the UIHPT.

\begin{theorem}\label{thm:uihpt_paths_from_v_wind}
     Let $v$ be a vertex in the Schnyder wood of a UIHPT $T$ produced by the Schnyder wood chiselling algorithm. Then every vertex on $\vpath{b}{v}$ has an outgoing red edge to $\vpath{r}{v}$, every vertex on $\vpath{r}{v}$ has an outgoing yellow edge to $\vpath{y}{v}$, and every vertex on $\vpath{y}{v}$ is either a boundary vertex of $T$ or has an outgoing blue edge to $\vpath{b}{v}$.
\end{theorem}

\begin{proof}
    Suppose first that $v$ is contained in the Schnyder wood segment. Observe that both $\vpath{r}{v}$ and $\vpath{b}{v}$ are contained entirely in the Schnyder wood segment, while $\vpath{y}{v}$ is either contained entirely in the Schnyder wood segment, or it exits the segment at some upper boundary vertex. Let $\vpath{y}{v}^*$ denote $\vpath{y}{v}$ truncated after the first vertex at which it leaves the segment, if it does. Observe that $\vpath{y}{v}^*$, $\vpath{r}{v}$, and $\vpath{b}{v}$ are finite.
    
    \begin{claim}\label{clm:segment_paths_from_v_wind}
        Suppose $v$ is contained in the Schnyder wood segment. Then every vertex on $\vpath{b}{v}$ has an outgoing red edge to $\vpath{r}{v}$ and every vertex on $\vpath{y}{v}^*$ is either a boundary vertex of $T$ or has an outgoing blue edge to $\vpath{b}{v}$. If $\vpath{y}{v}^*$ terminates at $b_{x+1}$, then every vertex on $\vpath{r}{v}$ has an outgoing yellow edge to $\vpath{y}{v}^*$. Otherwise, if $\vpath{y}{v}^*$ terminates on the upper boundary, then every vertex on $\vpath{r}{v}$ has an outgoing yellow edge to either a vertex on $\vpath{y}{v}^*$, or a vertex on the infinite yellow path contained in the Schnyder wood segment.
    \end{claim}
    
    \begin{poc}
        Since each of $\vpath{y}{v}^*$, $\vpath{r}{v}$, and $\vpath{b}{v}$ is finite, it follows that after a finite number of steps, the peeling process $P_x$ has explored every vertex on $\vpath{y}{v}^*$, $\vpath{r}{v}$, and $\vpath{b}{v}$, and all of their incident edges that are contained in the Schnyder wood segment. Consider the explored region after this step. By \cref{obs:finitestepscoversegmentofboundary}, the explored region intersects the initial boundary of $T$ in a segment from some vertex $b_\ell$ to $b_r$, where $b_r$ is the head of the current root edge. By replacing the remainder of the initial boundary with a finite boundary joining $b_\ell$ and $b_r$, and adding an arbitrary finite triangulation to the region between this boundary and the explored region, we may consider the explored region to be contained in some finite triangulation $T'$. Note that the finite explored region is assigned the same colouring by $P_x$ as if it were explored by the finite peeling process in $T'$. For $c \in \{y, r, b\}$, let $Q_c$ be the unique directed path of colour $c$ in $T'$ that starts at $v$. Note that $\vpath{b}{v} = \qpath{b}$, $\vpath{r}{b} = \qpath{r}$, and $\vpath{y}{v}^* \subseteq \qpath{y}$.
        
        By \cref{thm:finite_paths_from_v_wind}, it follows that every vertex on $\vpath{y}{v}^* \subseteq \qpath{y}$ that is not a boundary vertex of $T$ has an outgoing blue edge to $\qpath{b} = \vpath{b}{v}$, every vertex on $\vpath{b}{v} = \qpath{b}$ has an outgoing red edge to $\qpath{r} = \vpath{r}{v}$, and every vertex on $\vpath{r}{v} = \qpath{r}$ has an outgoing yellow edge to $\qpath{y}$. If $\vpath{y}{v}^*$ terminates at $b_{x+1}$, then $\qpath{y} = \vpath{y}{v}^*$, and so every vertex on $\vpath{r}{v}$ has an outgoing yellow edge to $\vpath{v}{y}^*$. Otherwise, $\vpath{y}{v}^*$ terminates at a vertex $u$ on the upper boundary. However, $\qpath{y}$ terminates at $b_{x+1}$, and so after leaving at $u$, $\qpath{y}$ needs to enter the finite explored region again later. By \cref{obs:finitestepsupperboundary} and the Schnyder root condition applied to the unexplored region of $T'$, $\qpath{y}$ can only enter the finite explored region at $b_r$. Since $b_r$ is the head of a root edge of $P_x$, we know by \cref{clm:schnydersegmentyellowforest} that the directed yellow path from $b_r$ to $b_{x+1}$ is contained in the infinite yellow path of the Schnyder wood segment. Therefore, $\qpath{y}$ intersects the Schnyder wood segment only on $\vpath{y}{v}^*$ and on the infinite yellow path contained in the Schnyder wood segment. Since all neighbours of $\qpath{r}$ are contained in the Schnyder wood segment, it follows that every vertex on $\vpath{r}{v} = \qpath{r}$ has an outgoing yellow edge to either a vertex on $\vpath{y}{v}^*$, or a vertex on the infinite yellow path contained in the Schnyder wood segment.
    \end{poc}
    
    Now suppose instead that $v$ is contained in a Schnyder wood strip of $T$. In this case, for $c \in \{r,y,b\}$ denote by $\vpath{c}{v}^*$ the unique (finite or infinite) directed path of colour $c$ that starts at $v$, truncated after the first vertex at which it leaves the Schnyder wood strip, if it does. Note that $\vpath{r}{v}^*$ is necessarily an infinite red path contained entirely in the strip, while both $\vpath{y}{v}^*$ and $\vpath{b}{v}^*$ may be either finite (if they terminate on the upper or lower boundary, respectively) or infinite (if they do not leave the strip).
    
    \begin{claim}\label{clm:strip_paths_from_v_wind}
        Suppose $v$ is contained in a Schnyder wood strip. Every vertex on $\vpath{b}{v}^*$ has an outgoing red edge to $\vpath{r}{v}^*$ and every vertex on $\vpath{y}{v}^*$ is either a lower boundary vertex of the Schnyder wood strip, or has an outgoing blue edge to $\vpath{b}{v}^*$. If $\vpath{y}{v}^*$ is infinite, then every vertex on $\vpath{r}{v}^*$ has an outgoing yellow edge to a vertex on $\vpath{y}{v}^*$. Otherwise, if $\vpath{y}{v}^*$ terminates on the upper boundary of the strip, then every vertex on $\vpath{r}{v}^*$ has an outgoing yellow edge to either a vertex on $\vpath{y}{v}^*$, or a vertex on the infinite yellow path on the lower boundary of the Schnyder wood strip.
    \end{claim}
    
    \begin{poc}
        This follows directly from \cref{clm:segment_paths_from_v_wind} and the fact that the Schnyder wood strip is the colouring assigned to the explored region by the process $P_x$ as $x \to -\infty$.
    \end{poc}
    
    We now show that the claimed structure holds for $\vpath{b}{v}$, $\vpath{r}{v}$, and $\vpath{y}{v}$ on the entirety of $T$. We will repeatedly use \cref{clm:segment_paths_from_v_wind,clm:strip_paths_from_v_wind,lem:monochromatic_segment_subgraphs,lem:monochromatic_strip_subgraphs}. For simplicity, we will assume that $v$ is in a Schnyder wood strip since the case where $v$ is in the Schnyder wood segment is very similar.

    First, consider $\vpath{b}{v}$. By \cref{clm:strip_paths_from_v_wind}, every vertex on $\vpath{b}{v}^*$ has an outgoing red edge to $\vpath{r}{v}^* \subseteq \vpath{r}{v}$. If $\vpath{b}{v}^*$ is infinite, then by definition $\vpath{b}{v} \setminus \vpath{b}{v}^*$ consists of all of the distinguished blue paths in a strip or segment below $v$. By \cref{lem:monochromatic_segment_subgraphs,lem:monochromatic_strip_subgraphs}, every vertex on these paths has an outgoing red edge to the distinguished red path in the same strip or segment, and this red path is part of $\vpath{r}{v}$ by definition. Otherwise, if $\vpath{b}{v}^*$ is finite and intersects the lower boundary of the strip at some vertex $u$, then either $u$ is on the boundary of $T$ and $\vpath{b}{v} = \vpath{b}{v}^*$, or the outgoing blue edge from $u$ goes directly to the distinguished blue path in the strip or segment immediately below $v$. In the latter case, the vertices of $\vpath{b}{v} \setminus \vpath{b}{v}^*$ are all on a distinguished blue path in a strip or segment below $v$ and so, as above, they have an outgoing red edge to $\vpath{r}{v}$. In all cases, every vertex on $\vpath{b}{v}$ has an outgoing red edge to $\vpath{r}{v}$, as required.

    Next, consider $\vpath{r}{v}$. If $\vpath{y}{v}^*$ is infinite, then by \cref{clm:strip_paths_from_v_wind} every vertex on $\vpath{r}{v}^*$ has an outgoing yellow edge to $\vpath{y}{v}^* \subseteq \vpath{y}{v}$. Otherwise, if $\vpath{y}{v}^*$ is finite and intersects the upper boundary of the strip containing $v$, then by \cref{clm:strip_paths_from_v_wind} every vertex on $\vpath{r}{v}^*$ has an outgoing yellow edge to $\vpath{y}{v}^* \subseteq \vpath{y}{v}$ or to the distinguished yellow path on the lower boundary of the strip. As $\vpath{y}{v}^*$ intersects the upper boundary, this yellow path is also part of $\vpath{y}{v}$. Lastly, $\vpath{r}{v} \setminus \vpath{r}{v}^*$ consists of all of the distinguished red paths in a strip or segment below $v$, and so, similarly to above, the vertices on these paths have an outgoing yellow edge to $\vpath{y}{v}$. In all cases, every vertex on $\vpath{r}{v}$ has an outgoing yellow edge to $\vpath{y}{v}$, as required.

    Finally, consider $\vpath{y}{v}$. By \cref{clm:strip_paths_from_v_wind}, every vertex on $\vpath{y}{v}^*$ that is not on the lower boundary of the strip has an outgoing blue edge to $\vpath{b}{v}^* \subseteq \vpath{b}{v}$. Consider a vertex $w \in \vpath{y}{v}^*$ that is on the lower boundary of the strip but is not on the boundary of $T$. Then, by  \cref{lem:monochromatic_segment_subgraphs} or \cref{lem:monochromatic_strip_subgraphs}, $w$ has an outgoing blue edge to the distinguished blue path in the strip or segment immediately below $v$. If $\vpath{b}{v}^*$ does not intersect the lower boundary of the strip, this blue path is part of $\vpath{b}{v}$. Otherwise, $\vpath{b}{v}^*$ intersects the lower boundary of the strip at some vertex $u$. Suppose that $w$ is to the right of $u$. Then $\vpath{y}{v}^*$ first goes from $v$ to $w$ and then follows a yellow path along the lower boundary of the strip until it reaches $u$. So, $\vpath{y}{v}^*$ and $\vpath{b}{v}^*$ intersect, which implies that they form a cycle $C$. Since the Schnyder conditions imply that all vertices other than $v$ and $u$ on $C$ have an outgoing edge to the interior of $C$, this contradicts \cref{lem:directed_edges_from_cycle}. Therefore, $u$ must be to the right of $w$, and so the outgoing blue edge from $u$ to the distinguished blue path in the strip or segment immediately below $v$ will intersect that blue path to the right of the outgoing blue edge from $w$. In particular, the outgoing blue edge from $w$ goes to a vertex on $\vpath{b}{v}$. So, every vertex on $\vpath{y}{v}^*$ that is not on the boundary of $T$ has an outgoing blue edge to $\vpath{b}{v}$.

    It remains to consider the vertices of $\vpath{y}{v} \setminus \vpath{y}{v}^*$. If $\vpath{y}{v}^*$ is infinite, then $\vpath{y}{v} \setminus \vpath{y}{v}^*$ consists of all of the distinguished yellow paths in a strip or segment below $v$. By \cref{lem:monochromatic_segment_subgraphs} or \cref{lem:monochromatic_strip_subgraphs}, these vertices are either on the boundary of $T$ or have an outgoing blue edge to the distinguished blue path in the strip or segment immediately below, which is part of $\vpath{b}{v}$. Otherwise, $\vpath{y}{v}^*$ is finite and intersects the upper boundary of the strip at some vertex $w$. Then, $\vpath{y}{v} \setminus \vpath{y}{v}^*$ consists of a one-ended infinite yellow path $Q$ starting at $w$ along the lower boundary of the strip immediately above $v$, as well as the distinguished yellow path in the strip containing $v$ and all of the distinguished yellow paths in a strip or segment below $v$. Similarly to before, every vertex on the distinguished yellow paths is either on the boundary of $T$ or has an outgoing blue edge to $\vpath{b}{v}$, so it only remains to check the vertices of $Q$. Note that $w$ has an outgoing blue edge to the distinguished blue path in the strip containing $v$, and by \cref{clm:strip_paths_from_v_wind} we know that this neighbour $u$ must also be on $\vpath{b}{v}^* \subseteq \vpath{b}{v}$. In particular, $\vpath{b}{v}$ contains every vertex left of $u$ on the distinguished blue path in the strip of $v$. Since any vertex on $Q$ is to the left of $w$, its outgoing blue edge to the distinguished blue path in the strip containing $v$ will intersect that distinguished blue path at a vertex left of $u$, and so the outgoing blue edge goes to $\vpath{b}{v}$. In all cases, every vertex of $\vpath{y}{v}$ is either on the boundary of $T$ or has an outgoing blue edge to $\vpath{b}{v}$.
\end{proof}

\subsection{Maximality and uniqueness} \label{ssec:structure:maximal}

In this section we prove \cref{thm:peeling_process_uihpt_maximal_wood}. Recall that a Schnyder wood of an infinite triangulation $T$ is called maximal if every directed cycle in the Schnyder wood is oriented clockwise, and $T$ contains no right-directed paths.  Recall also that in \cref{lem:oriented_triangles}, we showed that if a 3-orientation of a triangulation contains a directed cycle, then it also contains a directed triangle of the same orientation. As we noted after the proof, the proof relies only on the interior of this cycle being finite. Since the interior of a (finite) cycle in a UIHPT is almost surely finite, the following result suffices to show the Schnyder wood chiselling algorithm constructs no anticlockwise cycles.

\begin{lemma} \label{lem:no_anticlockwise_cycles}
    The Schnyder wood chiselling algorithm almost surely constructs no anticlockwise triangles on the UIHPT.
\end{lemma}

\begin{proof}
    The proof that no individual Schnyder wood segment contains an anticlockwise triangle is almost identical to the proof of \cref{lem:infinite_triangulation_noanticlockwise}: Any finite collection of edges of $T$ explored by the Schnyder peeling process is explored after a finite number of steps, and since the same colouring can be produced by the finite peeling process on some finite triangulation, \cref{thm:peeling_process_maximal_wood} shows that the segment contains no anticlockwise triangles. It follows from the construction of the Schnyder wood strip in \cref{ssec:uihpt:strip} that this result also applies to each Schnyder wood strip. The only remaining case to check is when one or more edges of the triangle lies on the lower boundary of one Schnyder wood strip, and the remaining edge(s) lie in the interior of the Schnyder wood strip or segment below it.

    Since each Schnyder wood strip and Schnyder wood segment satisfies the Schnyder upper boundary condition, we immediately see that this case could only occur when two edges of the triangle are in the strip or segment below the boundary, with one of them being an outgoing blue edge from some vertex $u$ on the upper boundary, and the other being an incoming yellow edge to some vertex $v$ on the upper boundary. For this triangle to be oriented anticlockwise, $u$ must be left of $v$ on the upper boundary. Again by the Schnyder upper boundary condition, we know that $v$ has an outgoing blue edge, which by planarity, must be directed into the interior of this triangle, contradicting \cref{lem:directed_edges_from_cycle}. Thus, the Schnyder wood defined by the chiselling algorithm contains no anticlockwise triangles.
\end{proof}

Having established \cref{thm:monochromatic_schnyder_wood_structure}, we are now ready to show that the Schnyder chiselling algorithm produces the unique maximal Schnyder wood of the UIHPT. Our proof will be very similar to the proof given for infinite triangulations with finite boundary in \cref{sec:infinite_triang_finite_bdry}. First, we establish that the leftmost walk from a fixed edge $e$ reaches the initial boundary of the UIHPT in a finite number of steps. This takes slightly more work than for triangulations with finite boundary because $e$ might lie within a Schnyder wood strip, in which case $e$ is not explored by the Schnyder peeling process after a finite number of steps. We instead make use of the structure of the layers established in \cref{thm:monochromatic_schnyder_wood_structure}. 

Recall from \cref{sec:infinite_triang_finite_bdry} that a \defn{leftmost walk} is a maximal directed walk that at each vertex always takes the first outgoing edge clockwise from the edge at which it entered, if an outgoing edge exists, and recall that $L(e)$ is the leftmost walk starting from edge $e$. In the following theorem, we show that once a leftmost walk intersects a distinguished path, it then repeatedly follows the unique outgoing edge described in \cref{thm:monochromatic_schnyder_wood_structure} to the preceding distinguished path, until it reaches the initial boundary. Examples of this are indicated on the right in \cref{fig:uihpt_structure}.

\begin{lemma}\label{lem:leftmost_path}
    Fix an edge $e$ in the Schnyder wood constructed by the chiselling algorithm on the UIHPT. Then either 
    \begin{enumerate}
        \item $L(e)$ terminates at the head of the initial root edge within a finite number of steps, or
        \item $L(e)$ intersects the initial boundary within a finite number of steps, and then follows the distinguished yellow path left along the initial boundary.
    \end{enumerate}
    In particular, at some point $L(e)$ intersects a distinguished path above or below $e$, and afterwards it repeatedly follows the unique outgoing edge to the preceding distinguished path until it reaches the initial boundary. 
\end{lemma}

\begin{proof}
    Let $T$ be a UIHPT and let $e = (u,v)$ be an edge of $T$. We showed in the proof of \cref{lem:leftmostwalk_finite_bdry} that in the Schnyder wood constructed by the peeling process on any infinite triangulation with finite boundary, no leftmost walk contains a directed cycle. An identical argument applies to the UIHPT, and it follows that $L(e)$ is a directed path in $T$. Furthermore, the Schnyder condition implies that whenever $L(e)$ is not incident with the initial boundary, the edge colours on $L(e)$ cycle, in order, red, blue, yellow, red, and so on, with the initial colour determined by the choice of $e$.

    We begin by showing that within a finite number of steps, $L(e)$ intersects a distinguished yellow path on the initial boundary of $T$ and enters it via a yellow edge, at which point orientation rules imply that it will continue to follow the distinguished yellow path along the boundary. (Note that we may intersect this path at some point via a blue edge, but then orientation rules imply that $L(e)$ will immediately leave the yellow path via a red edge.) First, suppose that $e$ lies on one of the distinguished paths in the Schnyder wood of $T$, say a distinguished blue path. Then, the next edge of $L(e)$ is the unique outgoing red edge from $v$, and we know that this edge is directed to the preceding distinguished red path by \cref{thm:monochromatic_schnyder_wood_structure}. By the same argument, $L(e)$ afterwards follows the unique outgoing yellow edge to the preceding distinguished yellow path, then the unique outgoing blue edge to the preceding distinguished blue path, and so on, until $L(e)$ eventually reaches an distinguished yellow path on the initial boundary via an outgoing yellow edge. Similarly, if $e$ is initially on an outgoing edge from one distinguished path to the preceding distinguished path, then $L(e)$ follows the finite sequence of outgoing edges with the same structure to the initial boundary. Finally, suppose that $e$ is not on a distinguished path and is not an outgoing edge from one of these paths to the preceding path. Then, $e$ is contained in some region bounded by two distinguished paths, say $Q_a$ above and $Q_b$ below. Furthermore, $e$ is contained in some finite region bounded by $Q_a$, $Q_b$, and two outgoing edges $f_1$ and $f_2$ from $Q_a$ to $Q_b$ with tails at consecutive vertices on $Q_a$. Since this region is finite and $L(e)$ is acyclic, it follows that eventually $L(e)$ must exit this region. Let $e'$ be the last edge of $L(e)$ contained in the finite region. If the head of $e'$ is on $Q_a$, it follows from the definition of $L(e)$ that the next edge of $L(e)$ is either on $Q_a$, or is $f_1$ or $f_2$. Otherwise, the head of $e'$ is on $Q_b$, and it follows from the definition of $L(e)$ and the fact that $L(e)$ leaves the finite region that the next edge is the unique outgoing edge directed from the head of $e'$ on $Q_b$ to the preceding distinguished path below $Q_b$. (If there is no preceding distinguished path below $Q_b$ then it follows that $Q_b$ is already a distinguished yellow path along the initial boundary, and $e'$ is a yellow edge.) In either case, we again follow the finite sequence of outgoing edges with the same structure as earlier to a distinguished yellow path on the initial boundary of $T$.

    Note that there is an edge case where $e$ is initially contained in the finite region underneath a yellow chord on the initial boundary of $T$. As above, $L(e)$ must eventually exit this region. By the Schnyder boundary condition, it can only do so by entering the distinguished yellow path that contains this chord via a yellow or red edge.

    Hence, we may assume that $L(e)$ intersects a distinguished yellow path on the initial boundary within a finite number of steps and enters it via a yellow or red edge. By the definition of $L(e)$, this implies that afterwards $L(e)$ simply follows this distinguished yellow path along the initial boundary. If $L(e)$ has intersected the distinguished yellow path of the Schnyder wood segment, it then terminates at the head of the initial root edge. Otherwise, it follows the distinguished yellow path of the first Schnyder wood strip along the initial boundary without terminating.
\end{proof}

We can now use this result as in the proof of \cref{lem:maximal_wood_finite_bdry} to show that the Schnyder wood chiselling algorithm constructs no right-directed paths.

\begin{lemma} \label{lem:no_right_directed_paths}
    The Schnyder wood chiselling algorithm almost surely constructs no right-directed paths on the UIHPT.
\end{lemma}

\begin{proof}
    Suppose for a contradiction that the Schnyder wood chiselling algorithm constructs a right-directed path $R$ on a UIHPT $T$, and let $e$ be any edge on $R$. Observe first that the Schnyder boundary condition implies that $R$ cannot contain any edge on or below one of the distinguished yellow paths along the initial boundary of $T$. $R$ also cannot contain the head of the initial root edge since that vertex has no outgoing edge.
    
    By \cref{lem:leftmost_path}, the leftmost walk $L(e)$ will eventually contain either the head of the initial root edge or an edge on one of the distinguished yellow paths along the initial boundary. So, there must be an edge $e'$ on $L(e)$ that is directed from a vertex $v$ on $R$ towards the side of $R$ containing the initial boundary. As in \cref{lem:maximal_wood_finite_bdry}, consider the first such edge. The previous edge on $L(e)$ lies either on $R$, or on the side of $R$ not containing the initial boundary. In either case, the outgoing edge from $v$ that is on the path $R$ is further left than $e'$, contradicting the definition of $L(e)$. It follows that no right-directed path exists.
\end{proof}

It follows that the Schnyder wood constructed on the UIHPT by the Schnyder wood chiselling algorithm is maximal and so, by \cref{thm:max_schnyder_unique_infinite}, is the unique maximal Schnyder wood. This concludes the proof of \cref{thm:peeling_process_uihpt_maximal_wood}.

\subsection{The UIHPT as a limit of finite triangulations}\label{ssec:structure:uihptlimitoffinite}

While we have shown that the UIHPT has a unique maximal Schnyder wood, further work is required to prove \cref{conj:uihpt_as_limit}, that is, to show that the unique maximal Schnyder wood of the UIHPT is the weak limit of the maximal Schnyder wood of a uniformly random triangulation from $\cT_n^m$ as $n \to \infty$ and then $m \to \infty$. Verifying this is considerably more involved than the analogous result for the UIPT, \cref{thm:uipt_schnyder_wood_convergence}, due to the increased complexity of the exploration process required on the UIHPT. Because the boundary of the UIPT is finite, we were able to demonstrate in \cref{sec:infinite_triang_finite_bdry} that the Schnyder peeling process almost surely explores any ball of radius $r$ around the root edge of the UIPT by circling around the unexplored boundary. By contrast, for the UIHPT, the boundary is infinite and so we have to repeatedly reinitiate the peeling process to explore the entire ball of radius $r$ around the root edge. In order to verify \cref{conj:uihpt_as_limit}, one would need to check that this process colours the ball of radius $r$ in a manner consistent with the single peeling process circling around a finite boundary as that boundary length tends to $\infty$.

Computational results suggest that this should hold. They show that in maximal Schnyder woods of large finite triangulations, the unique directed monochromatic paths from a uniformly random vertex to the boundary `spiral' around each other before terminating at the boundary. This is illustrated in \cref{fig:some_zoom} and in \cref{fig:tutte_embedding_big}, below, for a Tutte embedding of a large random triangulation\footnote{This triangulation was obtained by first generating a random triangulation of a triangle on $500,000$ vertices using Benedikt Stufler's simtria program \cite{simtria}, followed by performing a breadth-first search within a radius of $20$ of the root vertex of this triangulation and only keeping the largest connected component of the unexplored region. The root edge of this component was then chosen uniformly at random from its boundary, and the figures depict the unique maximal Schnyder wood of this rooted triangulation.}. The observed behaviour is consistent with the structure of the UIHPT demonstrated in \cref{thm:monochromatic_schnyder_wood_structure} if we once again imagine that the infinite red trees in the maximal Schnyder wood are part of a single infinite red tree rooted at $b_x$ that spirals through successive Schnyder wood strips before terminating in the Schnyder wood segment, and similarly for the infinite yellow and blue trees.

We now attempt to describe this spiralling behaviour more precisely. In a triangulation $T$, fix a vertex $v$ and consider two paths $P$ and $Q$ from $v$ to the boundary of $T$. If $T$ is a triangulation of the half-plane, we also allow each of $P$ and $Q$ to be a \defn{generalised path}, which we define to be an ordered collection of paths $P_1,\dots,P_\ell$ such that $P_1$ is a one-ended path that starts at $v$, $P_\ell$ is a one-ended path that ends at the boundary of $T$, and each $P_i$ for $i \in [2,\ell-1]$ is a two-ended path ordered from one end to the other. The \defn{winding number} $w(P,Q)$ is the signed number of occurrences of $Q$ strictly crossing $P$ from left to right.\footnote{We have not used the standard definition of a winding number, in order to apply winding numbers to paths on the UIHPT; however, we note that if $T$ is a triangulation of the half-plane and we consider the one-point compactification of the half-plane, then this definition is equivalent to the more standard, topological one, up to an additive error of at most $1$.} Note that $w(P,Q)$ is independent of the choice of embedding for $T$. One natural way to quantify the extent to which the monochromatic trees `spiral' in a given Schnyder wood is to consider $w(G(v),\vpath{c}{v})$ for some uniformly random vertex $v$ of $T$, where $G(v)$ denotes a geodesic path from $v$ to the boundary of $T$, and $\vpath{c}{v}$ denotes the unique directed path of colour $c$ starting from $v$, where we again define $\vpath{c}{v}$ to by the generalised path consisting of the colour-$c$ path starting from $v$, and all the infinite paths of colour $c$ in the strips and segment below $v$ when $T$ is a UIHPT. 

For now, let us only consider $w(L(e), \vpath{c}{v})$, where $L(e)$ denotes the leftmost walk starting from a chosen outgoing edge $e$ of $v$. Recall that the Schnyder conditions imply that any leftmost walk starting from $v$ must follow outgoing edges whose colours cycle through red, yellow, blue, red, yellow, blue, and so on (where the first edge $e$ could be any of the three colours). It therefore follows from \cref{thm:finite_paths_from_v_wind,thm:finite_boundary_paths_from_v_wind,thm:uihpt_paths_from_v_wind} that on the maximal Schnyder wood of any finite triangulation, any infinite triangulation with finite boundary, or the UIHPT, the vertices of $L(e)$ after $v$ cycle through being on the paths $\vpath{r}{v}, \vpath{y}{v}, \vpath{b}{v}, \vpath{r}{v}$, and so on, where the first colour is again determined by the colour of $e$. It immediately follows that $\abs{w(L(e), \vpath{c}{v}) - \abs{L(e)}/3} \leq 1$, where $\abs{L(e)}$ denotes the length of the subwalk of $L(e)$ between $v$ and the first vertex of $L(e)$ on the initial boundary. (Such a vertex exists trivially on finite triangulations, by \cref{lem:leftmostwalk_finite_bdry} on infinite triangulations with finite boundary, and by \cref{lem:leftmost_path} on the UIHPT). The precise value of $w(L(e), \vpath{c}{v})$ depends on the colour of the chosen edge $e$. 

On a UIHPT $T$, the structure of the geodesic paths is closely linked to the structure of leftmost walks. Suppose for now that $v$ is a vertex on one of the distinguished paths of $T$. If $e$ is the unique outgoing edge from $v$ to the infinite path below $v$, then the leftmost walk $L(e)$ passes through each of the infinite paths between $v$ and the boundary by passing through exactly one distinguished path per edge. Since any path from $v$ to the boundary must pass through each of the distinguished paths contained in the Schnyder wood strips below $v$, and can pass through at most one of these per edge, it follows that the geodesics in $T$ must be very similar to the leftmost walks. The only possible difference is that for some vertices $v$, $L(e)$ may differ from $G(v)$ because a slightly different path to the initial boundary could allow $G(v)$ to avoid passing through the Schnyder wood segment, reducing the overall length of the path by at most $3$. In particular, $\abs{L(e)} - 3 \leq \abs{G(v)} \leq \abs{L(e)}$ whenever $v$ is on an infinite monochromatic path. By a similar analysis for the winding number, we obtain that for each colour $c$, the infinite monochromatic path $\vpath{c}{v}$ satisfies $w(L(e),\vpath{c}{v}) - 1 \leq w(G(v), \vpath{c}{v}) \leq w(L(e),\vpath{c}{v})$. In particular, $w(G(v), \vpath{c}{v}) = \abs{G(v)}/3 + \cO(1)$. We note that exploiting the similarity in structure between geodesics and leftmost walks has already previously been of use in the study of large random maps. In particular, in \cite{addario2017scaling}, this relationship was used as an essential component of the proof of convergence of large random simple triangulations to the Brownian map, after rescaling.

Finally, if $v$ is not on one of the distinguished paths of $T$, then it is in one of the finite regions contained between two consecutive such paths. From the structure of the UIHPT, we expect that for a typical vertex $v$ in such a region, the number of edges on $L(e)$ before $L(e)$ intersects a distinguished path is $\cO(1)$. Any given vertex on this path can contribute at most $1$ to the winding number of $\vpath{c}{v}$ for any $c$, so it follows that for a uniformly random vertex $v$ of $T$, we also have $w(G(v), \vpath{c}{v}) = \abs{G(v)}/3 + \cO(1)$. If \cref{conj:uihpt_as_limit} holds, we expect to see approximately this behaviour also for winding numbers of uniformly random vertices on large finite triangulations. Computational results (with $n$ close to $500,000$, and $300 \leq m \leq 600$) seem to suggest that the winding numbers of a uniformly random vertex $v$ are indeed close to $\abs{G(v)}/3$. For instance, in the example given in \cref{fig:tutte_embedding_big}, a geodesic path from $v$ has length $36$, and the blue, red, and yellow paths from the chosen vertex have winding numbers $11$, $12$, and $13$ respectively. \cref{fig:most_zoom} provides a `zoomed-in' view of the paths close to $v$, confirming that the paths immediately establish the structure described in \cref{thm:finite_paths_from_v_wind}.

\begin{figure}[h!]
\begin{center}
\includegraphics[width=\textwidth]{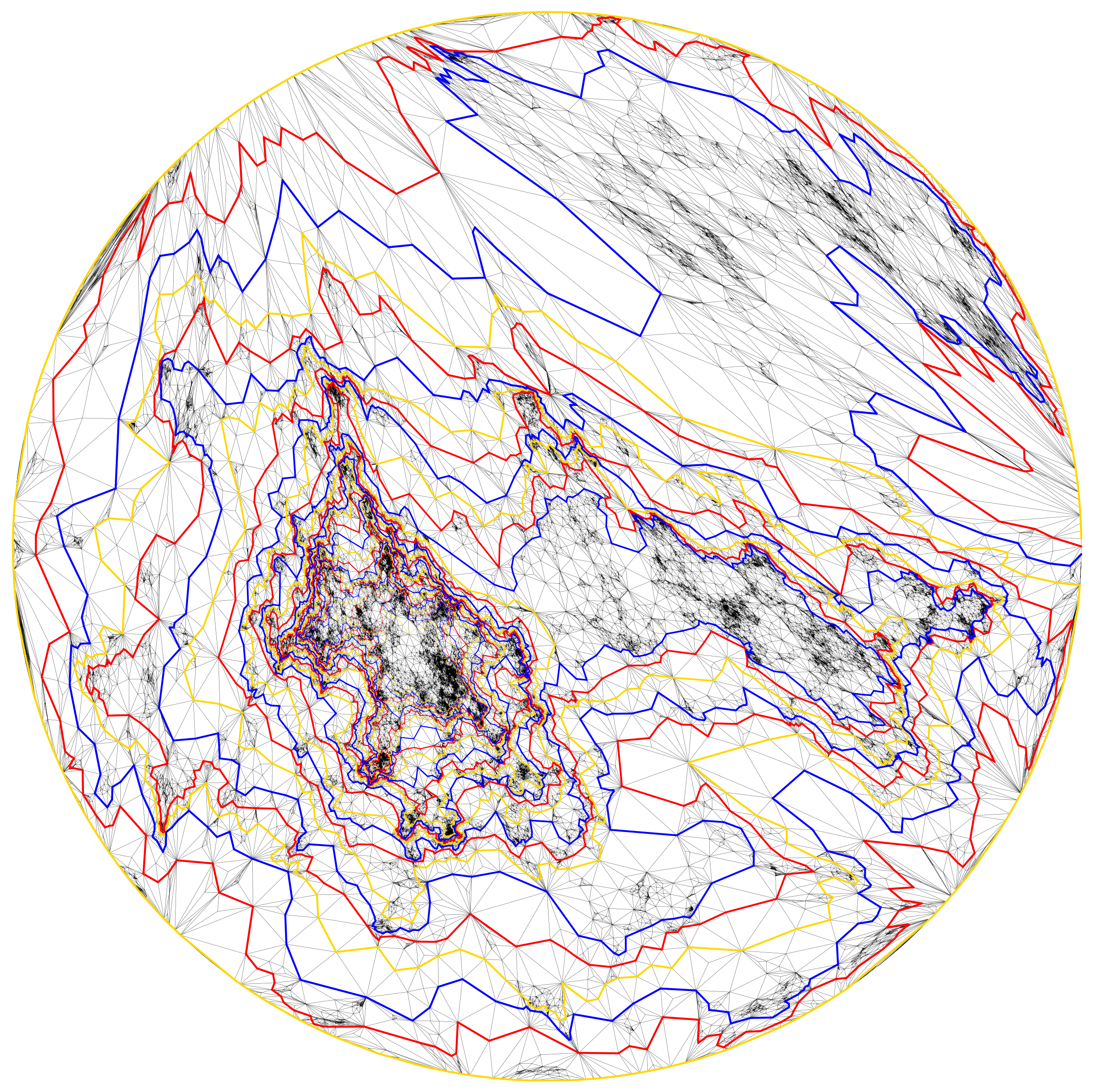}
\end{center}
\caption{A Tutte embedding of a random triangulation of a $432$-gon on $484,848$ vertices, and the monochromatic paths from a uniform random vertex to its boundary in the unique maximal Schnyder wood of the triangulation.}
\label{fig:tutte_embedding_big}
\end{figure}

\begin{figure}[h!]
\begin{center}
\includegraphics[width=\textwidth]{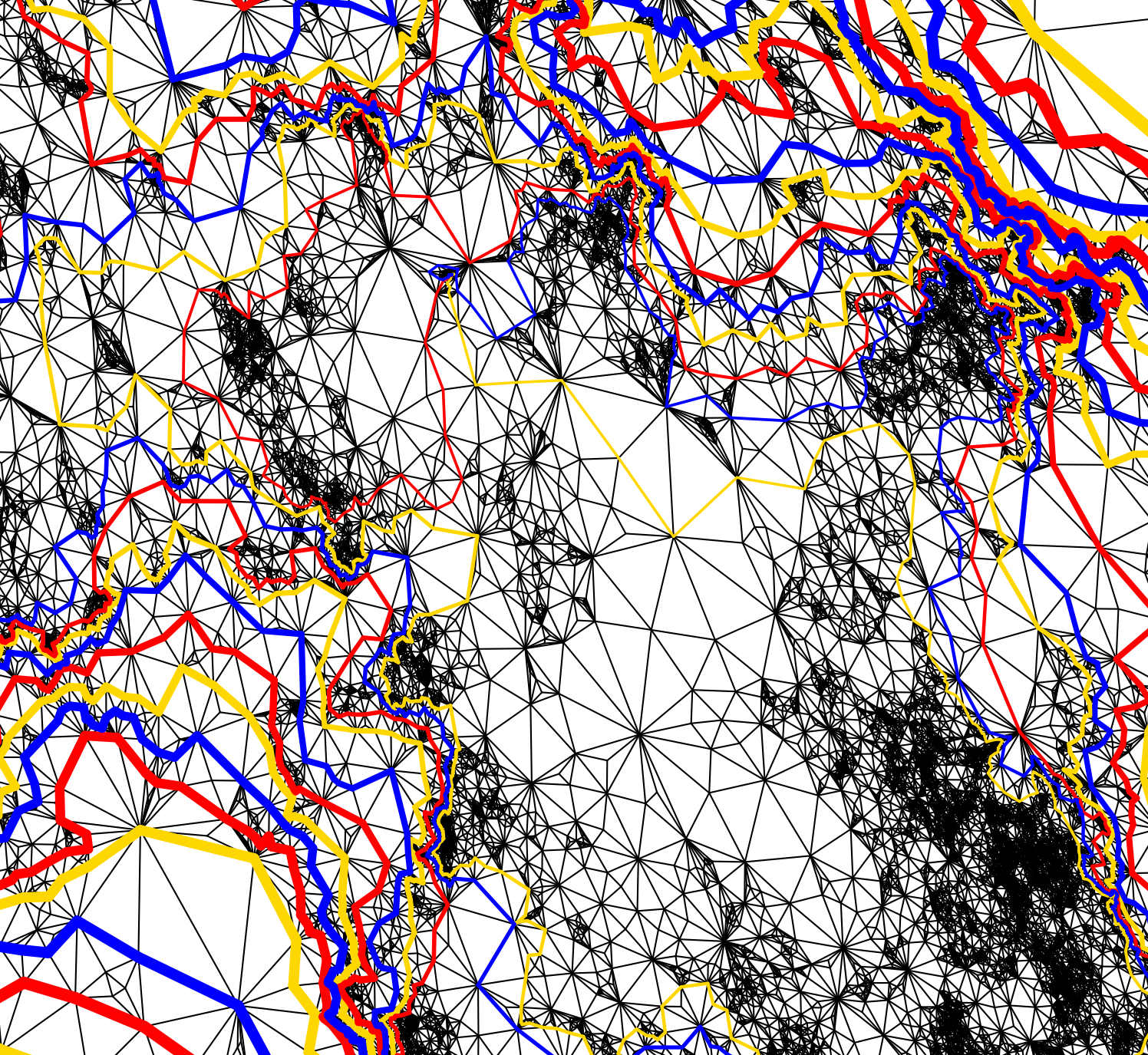}
\end{center}
\caption{A view of the paths in \cref{fig:tutte_embedding_big}, shown close to the uniform vertex $v$. The monochromatic paths from $v$ immediately establish the path structure described in \cref{thm:finite_paths_from_v_wind}, in which each vertex on a path is at distance one from the path immediately below it.}
\label{fig:most_zoom}
\end{figure}

\clearpage
\section{Notation}

\renewcommand\arraystretch{1.4}
\begin{tabularx}{\textwidth}{cXc}
    $\cT_n^m$ & The set of simple, $2$-connected, rooted triangulations of an $(m+2)$-gon with $n$ interior vertices. & \pageref{gl:ctnm} \\
    $\nt_{m,n}$ & The number of triangulations in $\cT_n^m$. & \pageref{gl:ntnm} \\
    $\ball{\rho}{v}$ & The ball of radius $\rho$ centred at vertex $v$. & \pageref{gl:ball} \\
    $\ballt{\rho}{v}{T}$ & The ball of radius $\rho$ centred at vertex $v$ in a specified triangulation $T$. & \pageref{gl:ballt} \\
    $\vpath{c}{v}$ & The directed path of colour $c$ from a vertex $v$ in a maximal Schnyder wood. & \pageref{gl:vpath}, \pageref{gl:vpathuipt}, \pageref{gl:vpathuihpt}\\ 
    $\infpath{c}$ & The distinguished path of colour $c$ in a given Schnyder wood segment or Schnyder wood strip. & \pageref{gl:infpath}, \pageref{gl:infpathstrip}\\
    $\intr T$ & The interior of a triangulation $T$, which includes its lower boundary. & \pageref{gl:intr}\\
    $B = \boundary{x}{0}$ & The initial boundary $(b_k, k\in\Z)$ of some triangulation rooted at $(b_x, b_{x+1})$. & \pageref{gl:initialboundary} \\
    $P_x$ & The peeling process initiated with root edge $(b_x, b_{x+1})$. & \pageref{gl:px} \\
    $\boundary{x}{i}$ & The updated boundary after step $i$ of $P_x$. & \pageref{gl:boundary} \\
    $\roote{x}{i}$ & The updated root edge after step $i$ of $P_x$. & \pageref{gl:roote} \\
    $\peelv{x}{i}$ & The peeling vertex used during step $i$ of $P_x$. & \pageref{gl:peelv}\\
    $\chord{x}{i}$ & The chord used during step $i$ of $P_x$. & \pageref{gl:chord}\\
    $\area{x}{i}$ & The revealed area up to and including step $i$ of $P_x$. & \pageref{gl:area}\\
    $\evarea{x}$ & The area consisting of all edges that are revealed by some step of $P_x$. & \pageref{gl:evarea} \\
    $L(e)$ & The leftmost walk starting from directed edge $e$. & \pageref{gl:le} \\
    $\fs{y}{x}$ & The first step of $P_y$ such that $\area{y}{\fs{y}{x}}$ contains $b_x$. & \pageref{gl:fs} \\
    $\head{x}{i}$ & The head of the root $\roote{x}{i}$ after step $i$ of $P_x$. & \pageref{gl:head} \\
    $\segment{y}{x}{k}$ & The segment of the first $k$ edges left of $\head{y}{\fs{y}{x}}$ along the boundary $\boundary{y}{\fs{y}{x}}$. & \pageref{gl:segment} \\
    $P_x(e)$ & The colour and orientation assigned to edge $e$ by process $P_x$. & \pageref{gl:pxe} \\
\end{tabularx}
\clearpage
\bibliography{references}
\end{document}